\documentclass[12pt,a4paper,reqno]{amsart}   
\title{
Atiyah sequences of braided Lie algebras \\ ~ \\ and their splittings   
}
\date{v2 October 2024}
\author{Paolo Aschieri, Giovanni Landi, Chiara Pagani}

\usepackage{amsmath, amssymb, amsthm,bbm} 
\usepackage[english]{babel} 
\usepackage{fullpage}  
\usepackage{enumerate}
\usepackage[all]{xy}
\usepackage{graphicx}
\usepackage{mathtools}
\usepackage{mathrsfs} 
\usepackage{hyperref}
\usepackage{rotating} 
\usepackage{caption}
\usepackage{setspace}
\usepackage{graphicx}
\usepackage{float}
\usepackage{dsfont}  
\usepackage{stmaryrd}  
\usepackage[T1]{fontenc}
\usepackage{etoolbox}   

\numberwithin{equation}{section}
\let\oldtocsection=\tocsection
\let\oldtocsubsection=\tocsubsection
\let\oldtocsubsubsection=\tocsubsubsection
 \renewcommand{\tocsection}[2]{\hspace{0em}\oldtocsection{#1}{#2}}
\renewcommand{\tocsubsection}[2]{\hspace{1em}\oldtocsubsection{#1}{#2}}
\renewcommand{\tocsubsubsection}[2]{\hspace{2em}\oldtocsubsubsection{#1}{#2}}

\address[]{\textit{Paolo Aschieri}  \newline \indent 
Universit{\`a} del Piemonte Orientale, 
\newline \indent 
Dipartimento di Scienze e Innovazione Tecnologica
\newline \indent   viale T.~Michel~11,~15121~Alessandria,~Italy,
\newline \indent  and INFN Torino, via P.~Giuria~1, 10125~Torino,~Italy,
\newline \indent  and Arnold-Regge centre, Torino, via P.~Giuria~1, 10125~Torino,~Italy}
\email{paolo.aschieri@uniupo.it}

\address[]{\textit{Giovanni Landi }
\newline \indent Matematica, Universit\`a di Trieste, Via A. Valerio 12/1, 34127 Trieste, Italy,
\newline \indent and INFN Sezione di Trieste, Trieste, Italy. }
\email{landi@units.it}

\address[]{\textit{Chiara Pagani} 
\newline \indent    Alma Mater Studiorum Universit\`a di Bologna,
\newline \indent Dipartimento di
Matematica,  Piazza di Porta S. Donato, 5, 40126 Bologna, Italy.
\newline \indent \textit{Current address:} Universit\`a degli Studi di Napoli Federico II, 
\newline \indent  Dipartimento di Matematica e Applicazioni, Via Cintia 21, 80126 Napoli, Italy. }
\email{cpagani@units.it}

\theoremstyle{plain}
\newtheorem{thm}{Theorem}[section]
\newtheorem{lem}[thm]{Lemma}
\newtheorem{prop}[thm]{Proposition}
\newtheorem{cor}[thm]{Corollary}
\theoremstyle{definition}
\newtheorem{defi}[thm]{Definition}
\theoremstyle{remark}
\newtheorem{rem}[thm]{Remark}

\AtEndEnvironment{ex}{\null\hfill \tiny{$\blacksquare$}}

\newcommand{\II}{\mathds{1}}
\newcommand{\nn}{\nonumber}
\newcommand{\ot}{\otimes}
\newcommand{\beq}{\begin{equation}}
\newcommand{\eeq}{\end{equation}}
\newcommand{\trl}{\vartriangleright}
\newcommand{\bbK}{\mathds{k}}
\newcommand{\id}{\mathrm{id}}
\newcommand{\M}{\mathcal{M}}
\renewcommand{\O}{\mathcal{O}}
\newcommand{\IR}{\mathbb{R}}
\newcommand{\U}{\mathcal{U}}

\renewcommand{\r}{\mathsf{R}}
\newcommand{\br}{{\overline{\r}}}
\newcommand{\dott}{{\scriptstyle{\bullet}_\theta}}
\newcommand{\cun}{\varepsilon}
\newcommand{\zero}[1]{{#1}_{\scriptscriptstyle{(0)}}}
\newcommand{\one}[1]{{#1}_{\scriptscriptstyle{(1)}}}
\newcommand{\mone}[1]{{#1}_{\scriptscriptstyle{(-1)}}}
\newcommand{\two}[1]{{#1}_{\scriptscriptstyle{(2)}}}
\newcommand{\three}[1]{{#1}_{\scriptscriptstyle{(3)}}}

\newcommand{\aut}[1]{\mathrm{aut}^\r_B(#1)}
\newcommand{\Der}[1]{\mathrm{Der}{(#1)}}
\newcommand{\DerR}[1]{\mathrm{Der}^\r{(#1)}}
\newcommand{\DerH}[1]{\mathrm{Der}^\r_{{\M^H}}{(#1)}}
\newcommand{\g}{\mathfrak{g}}
\newcommand{\Hom}{{\rm{Hom}}}
\newcommand{\stwo}{\tfrac{1}{\sqrt{2}}}
\newcommand{\half}{\tfrac{1}{2}}
\renewcommand{\lq}{\llbracket}
\renewcommand{\rq}{\rrbracket}

\newcommand{\bb}{b}

\newcommand{\wl}{{L}}
\newcommand{\ch}{{T}}
\newcommand{\wy}{{Y}}
\newcommand{\wk}{K}
\newcommand{\vp}{\omega}
\newcommand{\defp}{\lambda}
\newcommand{\cI}{\mathcal{I}}
\newcommand{\cl}{\mathcal{L}}
\newcommand{\ct}{\mathcal{T}}

 \usepackage{color}

\begin{document}

\begin{abstract} 
{Given an equivariant noncommutative principal bundle, we construct}
an Atiyah sequence of braided derivations whose splittings give
connections on the bundle. Vertical braided derivations act as
infinitesimal gauge transformations on connections.
{In the case of the principal $SU(2)$-bundle} over the sphere $S^{4}_\theta$ an
equivariant splitting of the Atiyah sequence recovers the instanton connection. An infinitesimal action of the braided conformal Lie algebra $so_\theta(5,1)$ yields a five parameter family of splittings.
On the principal $SO_\theta(2n,\mathbb{R})$-bundle of orthonormal frames over the sphere $S^{2n}_\theta$,  
{a} 
 splitting of the sequence 
leads to the Levi-Civita connection for the
`round' metric on $S^{2n}_\theta$. The corresponding Riemannian geometry of $S^{2n}_\theta$ is worked out.
\end{abstract}

\maketitle

\begin{spacing}{0.01}
  \tableofcontents
\end{spacing}

\parskip =.75 ex
\allowdisplaybreaks[4]

\section[Intro]{Introduction}
The Atiyah sequence of a principal bundle over a manifold $M$ is an important tool in the study of the geometry of Yang--Mills theories  \cite{AB}. It is given as a sequence of vector bundles over $M$ with Lie algebra structures on the corresponding modules of sections, 
{resulting} in a short exact sequence of (infinite dimensional) Lie algebras,   
\beq\label{atiyah-sections}
 0 \to   \Gamma {\operatorname{ad}(P) }  \to \mathfrak{X}_G(P) \to \mathfrak{X}(M) \to 0
\eeq
with principal $G$-bundle $P \to M$.  
Here $\mathfrak{X}(M)$ are the vector fields over $M$ (the section of the tangent bundle $TM$), while
$\mathfrak{X}_G(P)$  consists of $G$-invariant vector fields on $P$ (the $G$-invariant sections of the tangent bundle $TP$ 
along the fibres of $P$), and $\Gamma {\operatorname{ad}(P) } $ its Lie subalgebra made of vertical ones (with ${\operatorname{ad}(P) }$ the bundle associated to $P$ via the
adjoint representation of $G$ on its Lie algebra).  

A splitting of the sequence corresponds to a $G$-invariant direct sum decomposition of the tangent space $TP$   in horizontal and vertical parts, that is, to a connection on $P \to M$. 
The space $\Gamma {\operatorname{ad}(P) }$  {is the (infinite dimensional) Lie algebra of the gauge group of the principal bundle $P$. Its elements are infinitesimal gauge transformations, \cite[\S 3]{AB}.}

Exact sequences of vector bundles (named after Atiyah) and their splittings were introduced in \cite{At57} in the complex analytic context, motivated by the study of connections  and obstructions to their existence. Later versions are in \cite{Tel75}, \cite{Le85}, with a braided case coming from a $\mathbb{Z}_2$-grading in \cite{LM87}. The sequence can also be seen as a Lie algebroid \cite{Mac10}. 
 
In \S \ref{sec:AS} of the present paper, we study a quantum version 
$
0 \to \g  \to {P}   \to {T} \to 0
$
of the 
Lie algebroid \eqref{atiyah-sections} in the context of braided Lie algebras (of derivations) associated with a triangular Hopf algebra $(K,\r)$. 
\footnote{
Indeed in the present paper we only consider triangular Hopf algebras $(K,R)$ and the associated symmetric monoidal category of $K$-modules. Thus, rather than braided Lie algebras, a more appropriate term would be $(K,R)$-symmetric Lie algebras. The latter terminology is the one adopted in \cite{pgc-main}.
} 
Each term in the sequence is a braided Lie--Rinehart pair for a braided commutative algebra $B$.  Given any such an exact sequence, we define  a connection  to be a splitting of 
{the underlying short exact sequence of vector spaces}. 
The splitting does not need to be  a braided Lie algebra morphism: the curvature is defined as the $\g$-valued braided two-form on $T$ that measures the extent to which the connection fails to be such. 
In \S \ref{se:scgt} elements of $\g$ are explicitly seen as infinitesimal gauge transformations which act   on the set of connections (splittings) of the sequence. 
Having a connection one can define the covariant derivative of $\g$-valued braided forms on $P$.  The connection and the curvature satisfy a structure equation and a Bianchi identity (Proposition \ref{prop:bianchi}).  In this braided context, the covariant derivative of the curvature does not vanish in general (in agreement with the result in \cite{Du}) and \cite{RZ21}), while it does for a $K$-equivariant connection. 

In \S \ref{sec:AS-HG} we consider $K$-equivariant noncommutative principal bundles, that is $K$-equivariant Hopf--Galois extensions $B \subset A$, {where $(K,\r)$ is a  triangular Hopf algebra. }
We have then a corresponding sequence of braided Lie algebras of derivations as in \eqref{as},
$$
0 \to \aut{A} \to \DerH{A} \to \DerR{B}   \to 0 
$$ 
which, when exact, is a noncommutative version of the 
Atiyah sequence \eqref{atiyah-sections}. 
The braided Lie algebra of infinitesimal gauge transformations
$\aut{A}$ of a Hopf--Galois extension consists
of vertical braided derivations of the algebra $A$; it was introduced and studied in \cite{pgc-main, pgc-examples}.   
The theory is exemplified with the construction in \S \ref{sec:ASib} of the Atiyah sequence of
braided Lie algebras for the instanton bundle on the noncommutative sphere $S^4_\theta$ and a five parameter family of splittings (connections) of it. 
In a $C^*$-algebra context an Atiyah sequence for noncommutative principal bundles was given in \cite{SW}. 

The last part of the paper in \S \ref{sec:riem-geo} is dedicated to
the example of the noncommutative principal $SO_\theta(2n,
\mathbb{R})$-bundle $SO_\theta(2n+1, \mathbb{R})\to S^{2n}_\theta$
on  the noncommutative sphere
$S^{2n}_\theta$.
This is the noncommutative orthogonal frame bundle 
{through an }
identification of
braided derivations of $\O(S^{2n}_\theta)$ as sections of the
 bundle associated to the
principal bundle  via
the fundamental corepresentation of  $\O(SO_\theta(2n,\mathbb{R}))$ on the algebra $\O(\IR^{2n}_\theta)$.

The corresponding  Atiyah sequence is constructed and an equivariant
splitting determined.
This  leads to a novel and very direct construction
of the Levi-Civita connection on $S^{2n}_\theta$; the connection is 
explicitly presented and globally defined using global coordinate
functions on $S^{2n}_\theta$.
Indeed the principal equivariant connection induces a covariant derivative on the
associated tangent bundle which is torsion free and compatible with
the `round' metric.
We then work out the corresponding Riemannian geometry of
$S^{2n}_\theta$. The latter is an Einstein space (the Ricci tensor being
proportional to the metric, 
{see}  \eqref{ricci}) and a space form (the
scalar curvature being constant, 
{see} \eqref{sc}).

The study of Levi-Civita connections in noncommutative geometry is a
very active field of research.  
In the braided context uniqueness and existence of
Levi-Civita connections has been actively pursued, especially via Koszul formulae,
see \cite{AscLC} and the different contributions there referred to.
Our new contribution to this subject --in this paper on connections as splitting
of Atiyah sequences-- is the explicit and globally  defined expression
of the Levi-Civita connection on $S^{2n}_\theta$ given in \eqref{cov-der-T}.

\section{Algebraic preliminaries}\label{sec:alg-pre}
 
We work in the category of $\bbK$-modules with $\bbK$ a commutative field.
All algebras  are  unital and associative and morphisms of
algebras preserve the unit. 
{All coalgebras satisfy the corresponding dual conditions.}
We use standard terminologies and notations in
Hopf algebra theory. 
For $H$ a bialgebra, we  call  {\textit{$H$-equivariant}} a map of $H$-modules or $H$-comodules.

 Recall that a bialgebra (or Hopf algebra) $K$ is {\textit{quasitriangular}} if there exists an invertible element  $\r \in K \ot K$ (the universal $R$-matrix of $K$) with respect to which the coproduct
  $\Delta$ of $K$ is quasi-cocommutative
\beq\label{iiR} 
\Delta^{cop} (k) = \r \Delta(k) \br
\eeq for each $k \in K$, with $\Delta^{cop} := \tau \circ \Delta$, $\tau$ the flip map, 
 and $\br \in K \ot K$ the inverse of $\r$, $\r \br = \br \r = 1 \ot 1 $.
Moreover $\r$ is required to satisfy, 
\beq \label{iii} (\Delta \ot \id) \r=\r_{13} \r_{23} \qquad \mbox{and} \qquad 
(\id \ot \Delta)\r=\r_{13} \r_{12}.
 \eeq
 We write $\r:=\r^\alpha \ot \r_\alpha $ with an implicit sum. Then
$\r_{12}=\r^\alpha \ot \r_\alpha \ot 1$,  and similarly for $\r_{23}$ and $\r_{13}$. 
From conditions \eqref{iiR} and \eqref{iii} it follows that $\r$ satisfies the quantum Yang--Baxter equation
$
\r_{12} \r_{13} \r_{23}= \r_{23} \r_{13}\r_{12} 
$. 
The $\r$-matrix  of a quasitriangular bialgebra $(K,\r)$ is unital:
$(\cun \ot \id) \r = 1 =  (\id \ot \cun) \r$, 
with $\cun$ the counit of $K$.
When $K$ is a  Hopf algebra, the {quasitriangularity} implies  that its antipode $S$ is invertible and satisfies
$$ 
(S \ot \id) (\r)= \br  \; ; \quad
(\id \ot S) (\br)= \r \; ; \quad
(S \ot S) (\r)= \r \; .
$$
The Hopf algebra $K$ is said to be {\textit{triangular}} when 
$\br= {\r}_{21},$ with ${\r}_{21}= \tau(\r)=\r_\alpha \ot \r^\alpha$.

\subsection{Braided Lie algebras of derivations}\label{se:bla}

For the purpose of the present paper we only need to
consider braided Lie algebras associated with a triangular 
Hopf algebra $(K, \r)$. 
A $K$-braided Lie algebra is a $K$-module $\g$, with action $\trl: K \ot \g \to \g$, which is endowed with a bilinear map
(a bracket)
$$
[~,~]: \g\otimes \g\to \g
$$
which is equivariant, that is, $k\trl [u,v]=[\one{k}\trl u,\two{k}\trl v]$,  
for $\Delta(k) = \one{k} \ot \two{k}$ the coproduct of $K$, 
and satisfies the conditions
\begin{enumerate}[(i)]
\item braided antisymmetry: $[u,v]=-[\r_\alpha\trl v, \r^\alpha \trl u ]$,  \vspace{3pt}
\item braided Jacobi identity: $[u,[v,w]]=[[u,v],w]+[\r_\alpha\trl v, [\r^\alpha \trl u,w] ]$,
\end{enumerate}
for all $u,v,w \in \g$, $k\in K$.
A morphism of braided Lie algebras is a morphism of $K$-modules that in addition commutes with the brackets.

Any $K$-module algebra $A$  is a $K$-braided Lie algebra for the braided commutator 
\beq\label{Abla}
[~,~]: A\ot A\to A, \qquad a\ot b\mapsto [a,b] := ab - (\r_\alpha \trl b) \, (\r^\alpha\trl a) \, .
\eeq 
Also, for $A$ a $K$-module algebra,   the $K$-module algebra $(\Hom(A,A), \trl_{\Hom(A,A)})$
of linear maps from $A$ to $A$ with
action
\begin{align}\label{action-hom}
\trl_{\Hom(A,A)}:  K \ot \Hom(A,A) &\to \Hom(A,A) 
\nn \\
k \ot \psi & \mapsto  k \trl_{\Hom(A,A)} \psi : \; a \mapsto \one{k} \trl ( \psi(S(\two{k})\trl a) )
\end{align}
is a  braided Lie algebra with the braided commutator;  here  $S$ is the antipode of $K$.
 An important example of a braided Lie algebra associated to a $K$-module algebra is that  of braided derivations,  \cite[\S 5.2]{pgc-main}. A braided derivation is any
$Y\in \Hom(A,A)$ which satisfies  
\beq\label{Der}
Y(aa')= Y(a)
a' + (\r_\alpha\trl a)\,(\r^\alpha 
\trl_{\Hom(A,A)}
Y)(a') 
\eeq
for any $a,a'$ in $A$. We denote ${\mathrm{Der}}^\r(A)$
the  $\bbK$-module
of {braided derivations} of $A$ (to lighten notation we often drop the label $\r$).
It is a $K$-submodule of $\Hom(A,A)$, 
with action given by the restriction of 
 $\trl_{\Hom(A,A)}$
 \begin{align}\label{action-der}
\trl_{\Der{A}}:  K \ot \Der{A} &\to \Der{A}
\nn \\
k \ot Y & \mapsto  k \trl_{\Der{A}} Y : \; a \mapsto \one{k} \trl Y (S(\two{k})\trl a) 
\end{align}
and  a  braided Lie subalgebra of  $\Hom(A,A)$ with braided commutator
\begin{align}\label{bracket-der}
[~ , ~]  :& \, \Der{A}\ot \Der{A} \to \Der{A}
\nn \\
& Y\ot Y'\mapsto [Y,Y'] :=Y\circ Y'-(\r_\alpha\trl_{\Der{A}} Y') \circ (\r^\alpha\trl_{\Der{A}} Y).
\end{align}

{
\begin{rem}\label{rem:topo}
  When $K$ is not finite
  dimensional over $\bbK$ the matrix $\r$ is in general not an
  element of $K\otimes K$ but rather belongs to a suitable topological
  completion of the tensor product algebra. In the
  examples of the present paper this fact does not constitute a problem since 
  $\r$ acts diagonally in the fundamental representation and we
  consider representations  that are algebraic direct sums of
  this one. Representation-wise we
  hence consider the braided monoidal category of $K$-modules that are algebraic direct sums of
  finite dimensional representations of $K$-modules.
\end{rem} 
}

\subsection{Brackets of Lie algebra valued maps}\label{se:bblg}

Let  $N$ be a $K$-module and $\g$ a braided Lie algebra. The space of $\g$-valued multilinear maps $\eta : N \ot \dots \ot N \to \g$ which are  braided antisymmetric in their arguments ($\g$-valued forms on $N$), can be given the structure of a (super) braided Lie algebra.
In the present paper we just need the bracket between  one-forms and between a one-form and a two-form. 

Given any two linear maps $\eta, \phi : N \to \g$, their bracket is defined by 
\begin{align}\label{comm-gen}
\lq \eta, \phi \rq (Y, Y') 
& := \half 
\Big( [\eta (\r_\gamma \trl  Y), (\r^\gamma \trl \phi)(Y')] -  ( Y \ot Y' \to \r_\alpha \trl Y' \ot  \r^\alpha \trl Y ) \Big) \nn \\
& \:= \half 
\big( [\eta (\r_\gamma \trl  Y), (\r^\gamma \trl \phi)(Y')] - [\eta (\r_\gamma\trl  Y'), \r^\gamma \trl (\phi(Y))] \big)
\end{align}
for $Y, Y' \in N$.
The action on maps is the adjoint one in \eqref{action-hom}. By definition, the bracket is braided antisymmetric in the arguments, 
$
  \lq \eta, \phi \rq (Y, Y') =-\lq \eta, \phi \rq  (\r_\alpha \trl Y', \r^\alpha \trl Y) \, ,
$
while a short computation shows that 
$$
 \lq \eta, \phi \rq    = \lq \r_\tau \trl \phi, \r^\tau \trl \eta \rq  \, . 
$$
 In particular, when $\phi=\eta$ the formula \eqref{comm-gen} can be written as 
 \beq\label{etaeta}
\lq \eta , \eta \rq  (Y,Y') 
= \half 
\big[\eta (\r_\gamma \trl Y)+ \r_\gamma \trl (\eta( Y)), (\r^\gamma \trl \eta) (Y')\big]  .
\eeq
For a one-form $\eta$ and a two-form $\Phi$ one defines 
\beq\label{etaPhi}
\lq \eta, \Phi \rq  (X,Y,Z) := \big[\eta( \r_\gamma \trl X), (\r^\gamma \trl \Phi) (Y \ot Z)\big]  
+ \mbox{b.c.p.}
\eeq
and
\beq\label{Phieta}
\lq \Phi, \eta \rq  (X,Y,Z) := \big[ \Phi (\r_\gamma \trl (X \ot Y)), (\r^\gamma \trl \eta) (Z)\big]  
+ \mbox{b.c.p.}  
\eeq
with $\mbox{b.c.p.}$ standing for braided cyclic permutations of elements $(X,Y,Z)$  in $N$.
The braided antisymmetry of the two-form $\Phi$ implies the same property for both expressions above.  A short computation shows that 
$$
 \lq \eta, \Phi \rq   = - \lq \r_\lambda \trl \Phi, \r^\lambda \trl \eta \rq  \, . 
$$
From this one also computes that 
\beq\label{etaPhieta}
\big( \lq \eta , \Phi \rq  - \lq  \Phi , \eta \rq  \big) (X,Y,Z) 
 = \big[\eta (\r_\gamma \trl X)+ \r_\gamma \trl (\eta( X)), (\r^\gamma \trl \Phi) (Y \ot Z)\big]  + \mbox{b.c.p.}
\eeq
for $X,Y,Z \in N$, and in parallel with \eqref{etaeta}.

Finally, the $K$-equivariance of the bracket $[,]$ in $\g$ leads to the $K$-equivariance
 of the bracket in \eqref{comm-gen} ,
$$
 k \trl \lq \eta, \phi \rq  = \lq \one{k} \trl \eta, \two{k} \trl \phi \rq  \; , 
$$
and of those in \eqref{etaPhi} and \eqref{Phieta} with similar expressions. 

The bracket satisfies a braided Jacobi identity. For one-forms this reads as in the following proposition.
\begin{prop}
For one-forms $\omega, \eta, \varphi$, we have the identity
\beq\label{jacobi-forms}
\lq \omega , \lq \eta , \varphi \rq  \rq  
= \lq \lq \omega, \eta \rq , \varphi  \rq  
-  \lq \r_\alpha \trl \eta , \lq \r^\alpha \trl \omega  , \varphi \rq  \rq  .
\eeq
\end{prop}
\begin{proof}
We first compute separately the three terms in \eqref{jacobi-forms}: from   \eqref{etaPhi} and \eqref{comm-gen}, 
\begin{align}\label{j1}
2 &  \lq \omega, \lq \eta, \varphi  \rq \rq (X,Y,Z)  =  \big[ \omega (\r_\lambda \r_\gamma \trl X), [ (\r^\lambda \trl \eta) 
(\r_\alpha \trl Y) , (\r^\alpha \r^\gamma \trl \varphi) (Z) ]  \big]
\nn \\ 
& \qquad 
-  \big[  \omega (\r_\lambda \r_\gamma \trl X), [ (\r^\lambda \trl \eta) (\r_\alpha \trl Z) 
, \r^\alpha \trl ((\r^\gamma \trl \varphi) (Y)) ] \big]
+  \mbox{b.c.p.}  
\end{align}
and
\begin{align}\label{j3}
2 &  \lq \r_\tau \trl \eta, \lq \r_\tau \trl \omega, \varphi  \rq \rq (X,Y,Z) 
\\
& = \big[  (\r_\tau \trl \eta) (\r_\lambda \r_\gamma \trl X), [ (\r^\lambda \r^\tau \trl \omega) 
(\r_\alpha \trl Y) , (\r^\alpha \r^\gamma \trl \varphi) (Z) ]  \big]  \nn 
\\ 
& \quad 
- \big[  (\r_\tau \trl \eta) (\r_\lambda \r_\gamma \trl X), [ (\r^\lambda \r^\tau \trl \omega) (\r_\alpha \trl Z) 
, \r^\alpha \trl ((\r^\gamma \trl \varphi) (Y)) ] \big]
+  \mbox{b.c.p.}  \nn
\, .
\end{align}
While, from  \eqref{Phieta},  
\begin{align}\label{j2}
2 \lq \lq \omega, \eta \rq , \varphi  \rq (X,Y,Z) 
&  =  \big[ [ \omega (\r_\lambda \r_\gamma \trl X), (\r^\lambda \trl \eta) 
(\r_\alpha \trl Y)] , (\r^\alpha \r^\gamma \trl \varphi) (Z) \big] 
\\ 
&  
-  \big[  [ \omega (\r_\lambda \r_\alpha \trl Y), \r^\lambda \trl (\eta (\r_\gamma \trl X)) 
] , (\r^\alpha \r^\gamma \trl \varphi) (Z) \big]
+  \mbox{b.c.p.}  \nn .
\end{align}

\noindent
Using the braided cyclic permutation we substitute the second term in \eqref{j3} with the following one:
$$
 \big[(\r_\tau \trl \eta) (\r_\lambda \r_\gamma \r_\mu \trl Y), [ (\r^\lambda \r^\tau \trl \omega) (\r_\alpha \r^\nu \r^\mu \trl X) 
, \r^\alpha \trl ((\r^\gamma \trl \varphi) ( \r_\nu\trl Z)) ] \big] 
$$
which in turn, using \eqref{iii}, can be rewritten as
$$
\big[ \r_\tau \trl ( (\r^\lambda \trl \eta) (\r^\nu \r_\gamma \r_\mu \trl Y) ), 
[ \r^\tau \trl (\omega (\r_\lambda \r_\nu \r_\alpha \r^\mu \trl X)) 
	,( \r^\alpha \r^\gamma \trl \varphi) (Z) ] \big]  \, .
$$
Next, using the Yang--Baxter equation on the indices $\nu,\gamma,\alpha$ this reduces to
\begin{align*}
\big[ \r_\tau \trl ( (\r^\lambda \trl \eta) ( \r_\alpha \r^\nu \r_\mu \trl Y) ), 
[ \r^\tau \trl (\omega (\r_\lambda \r_\gamma \r_\nu \r^\mu \trl X)) 
,( \r^\alpha \r^\gamma \trl \varphi) (Z) ] \big] 
\\
=
\big[ \r_\tau \trl ( (\r^\lambda \trl \eta) ( \r_\alpha \trl Y) ), 
[ \r^\tau \trl (\omega (\r_\lambda \r_\gamma \trl X)) 
,( \r^\alpha \r^\gamma \trl \varphi) (Z) ]\big] \, .
\end{align*}
Finally, in the expression
$$
\Big( \lq \omega, \lq \eta, \varphi  \rq \rq - \lq \lq \omega, \eta \rq , \varphi \rq + 
\lq \r_\tau \trl \eta, \lq \r^\tau \trl \omega, \varphi  \rq \rq \Big) (X,Y,Z) 
$$
one finds that this term and the two positive ones in \eqref{j1} and \eqref{j2} sum up to zero due to the Jacobi identity for derivations. A similar computation shows that the other three terms sum up to zero as well, thus establishing \eqref{jacobi-forms}.
\end{proof}

\subsection{Brackets of Lie algebra valued forms on bimodules} 
{For a subalgebra $B \subseteq A$, the space $\Hom(A,A)$ is a left $B$-module via the left multiplication by elements in $B$. In general this is not the case for $\Der{A}$. A sufficient condition for that is the 
quasi-centrality of $B$ in $A$, that is 
\beq\label{qcAB} 
b\, a  =   (\r_\alpha \trl{a}) \, (\r^\alpha\trl{b})~, 
\eeq   
for all $ b\in B$, $a\in A$.
 When the $K$-module algebra $B$ is quasi-central in $A$, }
 the  braided Lie algebra $\Der{A}$ inherits the left $B$-module structure
\beq \label{LAmodderA}
 (b Y)(a):= b \, Y(a)\, ,  
\eeq
  for  $Y \in \Der{A}$, $b\in B$ and $a \in A$.  
  Then a right $B$-module structure is  given by
    \begin{equation}\label{RAmodderA}
Y \cdot b:=(\r_\alpha \trl b)(\r^\alpha\trl Y) .
  \eeq
We write $ Y \cdot b $ to  distinguish the right $B$-module structure from the evaluation 
 of a derivation on an element in $B$. 
 {From the definition of the bimodule structure it follows compatibility} with the  $K$-action:  
$$
 k\trl (bY \cdot b')=(\one{k}\trl b)(\two{k}\trl Y) \cdot (\three{k}\trl b') 
$$
 for all $k\in K,
b, b'\in B$ and $Y\in \Der{A}$. The Lie bracket   of $\Der{A}$ then satisfies   
      \begin{align}\label{moduloLie}
 [b Y, Y' \cdot b'] & = b [ Y, Y'] \cdot  b' 
+  b \, ( \r_\alpha\trl Y') \cdot  ((\r^\alpha \trl Y)(b')) \nn \\ &  \quad  -\big(\!\big(\r_\alpha\r_\beta\trl_{\Der{A}} Y'\big)\!\!\:( \r^\alpha \trl b)\!\big) 
   (\r^\beta\trl_{\Der{A}} Y) \cdot  b' 
\end{align}
  for all
$b,b'\in B$ and $Y,Y'\in \Der{A}$. 

Let $N$ and $N'$ be $K$-modules {and $B$-bimodules as above}. 
The space of right $B$-linear maps $\Hom_B(N, N')$ is a
$K$-module with action 
\begin{align}\label{action-hom-bis}
  k \trl_{\Hom(N,N')} \eta : \; Y \mapsto \one{k} \trl ( \eta(S(\two{k})\trl Y) )
\end{align}
(cf. \eqref{action-hom})
and a $B$-bimodule with actions 
$(b \eta)(Y)=b \, \eta(Y)$
and $\eta \, \cdot b= (\r_\alpha\trl b)(\r^\alpha\trl \eta)$, for $b\in B$, $\eta\in\Hom_B(N, N')$ and $Y \in N$.

With $\g$ a braided Lie algebra over $K$, the definition of the bracket in \S\ref{se:bblg} can be repeated for $\g$-valued forms on $N$, i.e. $\g$-valued right $B$-linear maps $\eta : N \ot_B \dots \ot_B N \to \g$ which are  braided antisymmetric in their arguments. 
As we shall see this results into a (super) braided Lie--Rinehart algebra.

\section{The  Atiyah sequence of braided Lie algebras} \label{sec:AS}

The classical Atiyah sequence \cite{At57,AB} is a sequence of vector bundles over a base space $M$ with Lie algebra structures on the corresponding modules of sections, this resulting in a short exact sequence of (infinite dimensional) Lie algebras.  In this paper we generalise the construction to a sequence of braided Lie algebras with braiding implemented by the triangular structure of a symmetry Hopf algebra $K$. 
In \S \ref{sec:AS-HG} we shall describe the basic example of this setting, the  sequence of braided Lie algebras of a $K$-equivariant Hopf--Galois extension $B = A^{coH} \subset A$, with structure Hopf algebra $H$.

 \subsection{The Atiyah sequence and the Lie--Rinehart structures}\label{se:seqrh}
 
Let  $(K,\r)$ be a triangular Hopf algebra and consider an exact sequence of $K$-braided Lie algebras,
\beq\label{as-gen}
0 \to \g \stackrel{\imath}{\to} {P} \stackrel{\pi} {\to} {T} \to 0 \,  
\eeq
 with  $\imath$ and $\pi$ braided Lie algebra morphisms. 
 We identify $\g$ with its image $\imath(\g)$ in $P$, while we write $\pi(Y) = Y^\pi $ for $Y\in P$, as the image in $T$ via the projection $\pi$. By construction $\g$ is an ideal in $P$ for the braided commutator:
$ [Y, V]  \in \g $, when $Y\in P$, $V\in \g$. 

{We further take $\g, P$ and $T$ to be $B$-bimodules with structures as in \eqref{LAmodderA} and
\eqref{RAmodderA} for $B$ a quasi-commutative algebra (that is \eqref{qcAB} holds for $A=B$)
and $\iota$, $\pi$ to be right $B$-module morphisms (and hence left ones).

Finally,} we take $(B,T)$ to be a braided Lie--Rinehart pair: the algebra 
$B$ is a $T$-module with
elements of $T$ acting as braided derivations of $B$,
$$
X (b b')=X(b) b' + (\r_\alpha\trl b)(\r^\alpha\trl X)(b') \, , \qquad b,b'\in B, \quad X \in T \, ,
$$
and there is compatibility with the $B$-module structure of $T$, 
\beq\label{blrp2}
[X, b X']  = X(b) X' + (\r_\alpha\trl b) [(\r^\alpha\trl X), X']  \, , \qquad b \in B , \quad X, X' \in T~,
\eeq
(which also fixes $[bX,  X' \cdot b'] $, cf. \eqref{moduloLie}).
The pair $(B,P)$ is as well taken to be a braided Lie--Rinehart pair (and so is $(B,\g)$ for the trivial action): 
elements of $P$ act as braided derivations of $B$ via the map $\pi$, 
$Y(b)= Y^\pi(b)$, for $Y \in P$ and $b \in B$.  The subalgebra $\g$ acts trivially on $B$: elements of $\g$ are `vertical'. 
For simplicity we take $T=\DerR{B}$. The sequence \eqref{as-gen} is  a braided Lie algebroid with anchor  $\pi$. 

\subsection{Splittings, connections and curvatures}\label{sec:sac}
A connection on the sequence \eqref{as-gen} is a splitting of it, that
is a right $B$-module map which is a section of $\pi$, 
\beq\label{hcon}
\rho : T \to P , \qquad \pi \circ \rho = \id_T \, . 
\eeq
 In general $\rho$ is not $K$-equivariant, that is it does not satisfy $k \trl \rho = \varepsilon(k) \rho$ for $k \in K$. 
 
 Using that $\rho(X) (b) = X(b)$ when $b\in B$, and recalling the adjoint action property $k \trl  (\rho (X)) = (\one{k} \trl \rho) (\two{k} \trl X)$
 for $X\in T$ and $k\in K$, 
 one easily shows that
 \begin{align}
(k \trl  \rho (X)) (b) = (k \trl X)(b) \label{uno} \\
 ((k \trl \rho) (X))(b) = \varepsilon(k)\, X(b)  \label{due} .
 \end{align}
From the $B$-bimodule structure of $T$ one also has $\rho(bX) = (\rho(\r_\alpha \trl X)) \cdot (\r^\alpha \trl b)$.

 \begin{rem}\label{rem:eq}
 In accordance with the $K$-adjoint action, $k \trl  (\rho (X)) = (\one{k} \trl \rho) (\two{k} \trl X)$, the connection $\rho$ ``acts from the left". Therefore it is natural to take it to be right $B$-linear. When $\rho$ is equivariant 
 right $B$-linearity also implies left $B$-linearity, $\rho( bX) = b \rho(X)$. 
 \end{rem}

The corresponding vertical projection, the retract of $\imath$,  is the right $B$-module map
\beq\label{vcon}
\vp : P \to \g , \qquad \vp(Y) := Y - \rho(Y^\pi) \, , \qquad Y \in P .
\eeq
Clearly $\vp(\rho(X)) =0$, for each $X \in T$. We denote  
\beq\label{horproj}
h:=\rho \circ \pi : P \to P 
\eeq
the horizontal projection. By definition $\omega + h = \id$. 

In general,   $\rho$ (and so $\vp$) is not a braided Lie algebra morphism. The extent to which it fails to be such is measured by the \textit{(basic)} curvature 
\begin{align} \label{hcur}
& \Omega(X,X') :=\rho([ X,X']  ) - \lq \rho, \rho \rq (X,X') \\
&\quad = \rho([ X,X']  ) - 
\half \Big( \left[\rho (\r_\alpha\trl  X), (\r^\alpha \trl \rho)(X') \right] 
- \left[\rho (\r_\alpha\trl  X'), \r^\alpha \trl (\rho( X ) ) \right]  \Big) , \nn 
\end{align}
for $X,X' \in T$.
From the $K$-equivariance of $\pi$ and $\pi \circ \rho = \id_T$ it follows that $\pi \circ \Omega = 0$,   and so $\Omega$ is $\g$-valued. Also, from $\omega\circ \rho = 0$ it follows that
$$
\Omega (X,X')= - (\vp\circ  \lq \rho, \rho \rq )(X,X') .
$$
{A braided antisymmetric right $B$-linear map from $T \ot_B T$  to $ \g $ is called a
 $\g$-valued braided two-form on $T$.

\begin{prop}
The curvature $\Omega$ is a $\g$-valued braided two-form on $T$. 
   \end{prop}
   }
  \begin{proof}
We first show that $\Omega$ is well-defined on $T\ot_B T$, that is $\Omega(X, b \, X') = \Omega(X \cdot b, X')$,  for $b\in B$ and $X,X' \in T$. Indeed, using  \eqref{blrp2}  and \eqref{RAmodderA}, we compute
\begin{align*}
 [ X,b  X'] 
& =  X(b) \, X' + (\r_\alpha\trl b) [\r^\alpha\trl X, X']  \\
& = (\r_\alpha \trl X') \cdot \r^\alpha \trl (X(b)) + \r_\beta \trl ([\r^\alpha\trl X, X'] ) \cdot (\r_\beta \r_\alpha\trl (b)) \\
& = (\r_\alpha \trl X') \cdot \r^\alpha \trl (X(b)) + [X, \r_\alpha \trl X']  \cdot (\r^\alpha\trl b). 
\end{align*}
and
\begin{align*}
[X \cdot b , X'] =  [X, \r_\alpha \trl X']  \cdot (\r^\alpha\trl b) + X \cdot (\r_\alpha \trl X') (\r^\alpha \trl b) . 
\end{align*} 
 Then, the right $B$-linearity of $\rho$ yields
 \beq\label{r-lin}
 \rho([X, b X']) - \rho([X \cdot  b, X']) = \rho(\r_\alpha \trl X') \cdot \r^\alpha \trl (X(b)) - \rho(X) \cdot (\r_\alpha \trl X') (\r^\alpha \trl b).
 \eeq
 
Next, using the expression in \eqref{etaeta} for the bracket: 
 \begin{align}\label{hbh1}
 2 \, \lq \rho, \rho \rq  (X, b X') &= \big[ \rho(\r_\alpha \trl X) + \r_\alpha \trl (\rho(X)), (\r^\alpha \trl \rho)(b X') \big]  \nn \\
 & = \big[ \rho(\r_\alpha \trl X) + \r_\alpha \trl (\rho(X)), (\r^\alpha \trl \rho)(\r_\tau \trl X') \big]  \cdot (\r^\tau \trl b) \nn \\
 & \quad + \r_\sigma \trl \big( (\r^\alpha \trl \rho)(\r_\tau \trl X') \big) \cdot \big( \r^\sigma \trl \big(\rho(\r_\alpha \trl X) + \r_\alpha \trl (\rho(X) ) \big) (\r^\tau \trl b)\big) 
 \nn \\
 & =  2 \, \lq \rho, \rho \rq  (X, \r_\tau X') \cdot (\r^\tau \trl b) \nn \\
   & \quad + 2 \, \r_\sigma \trl \big( (\r^\alpha \trl \rho)(\r_\tau \trl X') \big) \cdot \big( (\r^\sigma \r_\alpha \trl X) (\r^\tau \trl b)\big) \nn \\
 & =  2 \, \lq \rho, \rho \rq  (X, \r_\tau X') \cdot (\r^\tau \trl b) 
   + 2 \, \rho(\r_\tau \trl X') \cdot \r^\tau \trl (X(b)), 
 \end{align} 
using \eqref{uno} in the last but one equality. 
 \noindent
 While, using $(k \trl \rho) (X \cdot b) = (k \trl \rho) (X) \cdot b$, 
 \begin{align}\label{hbh2}
 2 \, \lq \rho, \rho \rq  (X \cdot b, X') &= \big[ \rho(\r_\alpha \trl (X \cdot b)) + \r_\alpha \trl (\rho(X \cdot b)), (\r^\alpha \trl \rho)(X') \big]  \nn \\
 &=\big[ \big( \rho(\r_\alpha \trl X) + \r_\alpha \trl (\rho(X)) \big) \cdot (\r_\beta \trl b), ( \r^\beta \r^\alpha \trl \rho)(X') \big]  \nn \\
&=\big[ \big( \rho(\r_\alpha \trl X) + \r_\alpha \trl (\rho(X)) \big) , \r_\lambda \trl \big( (\r^\beta \r^\alpha \trl \rho)(X') \big) \big]  \cdot (\r^\lambda \r_\beta \trl b) \nn \\
& \quad + \big( \rho(\r_\alpha \trl X) + \r_\alpha \trl (\rho(X)) \big) \cdot \r_\tau \trl \big((\r^\beta \r^\alpha \trl \rho)(X') \big) 
(\r^\tau \r_\beta \trl b) \nn \\
&=\big[ \big( \rho(\r_\alpha \trl X) + \r_\alpha \trl (\rho(X)) \big) , \r_\lambda \trl \big( (\r^\beta \r^\alpha \trl \rho)(X') \big) \big]  \cdot (\r^\lambda \r_\beta \trl b) \nn \\
& \quad + 2 \, \rho(X) \cdot (\r_\tau \trl X') (\r^\tau \trl b) , 
\end{align}
using \eqref{due}, in the last but one equality. 
Now, a direct computation shows that
\begin{align*}
\big[ \big( \rho(\r_\alpha \trl X) &+ \r_\alpha \trl (\rho(X)) \big) , \r_\lambda \trl \big( (\r^\beta \r^\alpha \trl \rho)(X') \big) \big]  \cdot (\r^\lambda \r_\beta \trl b) \\ 
&= \big[ \rho(\r_\alpha \trl X) + \r_\alpha \trl (\rho(X)), (\r^\alpha \trl \rho)(\r_\tau \trl X') \big]  \cdot (\r^\tau \trl b) \\ 
& = 2 \, \lq \rho, \rho \rq  (X, \r_\tau X') \cdot (\r^\tau \trl b) .
\end{align*}
With this, comparing \eqref{hbh1} and \eqref{hbh2}, we obtain
$$
\lq \rho, \rho \rq  (X, b X') - \lq \rho, \rho \rq  (X \cdot b, X') = \rho(\r_\tau \trl X')  \cdot \r^\tau \trl (X(b)) - \rho(X) \cdot (\r_\tau \trl X') (\r^\tau \trl b)
$$
which, when compared with \eqref{r-lin}, amounts to 
$\Omega(X, b \, X') - \Omega(X \cdot  b, X') = 0$.

\medskip
Next we show that $\Omega$ is right $B$-linear, that is $\Omega(X, X' \cdot b) = \Omega(X, X') \cdot b$.
From $[X, X' \cdot b] = [X, X' ] \cdot b + (\r_\alpha \trl X') \cdot(\r_\alpha \trl X) (b)$, the right $B$-linearity of $\rho$ gives
 \beq\label{rb}
 \rho([X, X' \cdot b]) = \rho([X, X' ]) \cdot b + \rho(\r_\alpha \trl X') \cdot(\r^\alpha \trl X) (b). 
 \eeq
On the other hand, 
recalling that the $K$-action defined in \eqref{action-hom-bis} closes on right $B$-linear maps, $(k \trl \rho) (X \cdot b) = (k \trl \rho) (X) \cdot b$ and using \eqref{due}, one computes 
 \begin{align*}
 2 \, \lq \rho, \rho \rq  (X, X' \cdot b ) &= \big[ \rho(\r_\alpha \trl X) + \r_\alpha \trl (\rho(X)), (\r^\alpha \trl \rho)(X' \cdot b) \big]  \nn \\
&= \big[ \rho(\r_\alpha \trl X) + \r_\alpha \trl (\rho(X)), (\r^\alpha \trl \rho)(X') \big]  \cdot b  \nn \\ 
& \quad + \r_\tau \trl \big( (\r^\alpha \trl \rho)(X') \big) \cdot \big(\r^\tau \trl \big( \rho(\r_\alpha \trl X) 
+ \r_\alpha \trl (\rho(X))\big)\big)(b) \nn \\
& = 
 2 \, \lq \rho, \rho \rq  (X, X') \cdot b + 2 \, \rho(\r_\tau \trl X') \cdot \big(\r^\tau \trl X)(b). 
\end{align*}
 A comparison with \eqref{rb} yields $\Omega(X, X' \cdot b) - \Omega(X, X') \cdot b = 0 $.
 
 \medskip
 Finally, both terms in \eqref{hcur} are braided antisymmetric and hence so is the curvature:
$\Omega(X,X') = -  \Omega( \r_\alpha\trl X', \r^\alpha\trl X )$.
\end{proof}  

The curvature can be given as $\g$-valued braided two-form on $P$ (the \textit{spatial} curvature):
\beq\label{bcur}
\Omega_{\vp}(Y,Y') :=  \Omega( Y^\pi, {Y'}^\pi ) , \qquad Y,Y' \in P .
\eeq
 This turns out to be \textit{basic}, that is $\Omega_{\vp}(Y,Y')=0$ whenever $Y$ or $Y'$ is vertical (an element of $\g$). 
Using \eqref{vcon} one also computes: 
$$
  \Omega_{\vp}(Y,Y') 
  = [\vp(Y), Y']  + [\r_\alpha \trl Y, (\r^\alpha \trl \vp) (Y')] -\vp([Y,Y'] ) - \lq \vp, \vp \rq (Y,Y').   
$$
This expression can be read as a {\it structure equation}:
\beq\label{streq}
d \vp = \Omega_{\vp} + \lq \vp, \vp \rq  \; .
\eeq
Here the $\g$-valued  two-form $\lq \vp, \vp \rq $ on $P$ is defined as
in \eqref{comm-gen}.  
 Also,  given  a $\g$-valued one-form $\eta$ on $P$, its exterior derivative $d \eta$ is the  
$\g$-valued  two-form on $P$ defined by 
 \begin{align}\label{der-1form}
 d \eta(Y, Y') 
& := [\r_\alpha \trl Y, (\r^\alpha \trl \eta) (Y')]  -[\r_\alpha \trl Y',  \r^\alpha \trl (\eta (Y))]  - \eta([Y,Y'] ) \nn \\
  &\:=2 \, \lq\id, \eta \rq  (Y,Y') - \eta([Y,Y'])
  \end{align}
for $Y,Y' \in P$.  
 The above 
is well-defined since $\g$, as mentioned, is an ideal in $P$ for the
braided commutator and braided antisymmetric by construction.  

One could construct an exterior algebra of $\g$-valued forms on $P$ by 
extending the definition of the exterior derivative $d$ to a form of any degree.  
For the sake of the present paper we just need it on one- and two-forms. 
For $\Phi$ a $\g$-valued two-form on $P$, its exterior derivative is given as
\begin{align}\label{d2}
d \Phi(X,Y,Z) &:= 
[\r_\lambda \trl X, (\r^\lambda \trl \Phi)(Y,Z)]  + \Phi(X, [Y,Z] ) + \mbox{b.c.p.} 
 \end{align}
for $X,Y, Z  \in P$, and $\mbox{b.c.p.}$ standing for braided cyclic permutations of $(X,Y,Z)$. 
The result is a $\g$-valued three-form on $P$.
Using this definition and \eqref{der-1form}, a lengthy and intricate computation that uses Jacobi identity 
and Yang--Baxter equation shows $d^2=0$.

 \subsection{The covariant derivative}   
Having a connection one can define the covariant derivative of (in particular) $\g$-valued braided forms on $P$. This uses the horizontal projection $h$ in \eqref{horproj}.
For a one-form $\eta$ we define
\begin{align}\label{cov1form}
D \eta(X,Y) 
   &:= \big([h(\r_\tau \trl X),  (\r^\tau \trl \eta)(Y)]  -  ( X \ot Y \to \r_\alpha \trl Y \ot  \r^\alpha \trl X ) \big) - \eta([X,Y]) \nn \\
   &\:= 2 \, \lq h, \eta \rq (X,Y) - \eta([X,Y]) , 
 \end{align}
while for a two-form $\Phi$  its covariant derivative is defined as  
\begin{align} \label{cov2form}
D \Phi(X,Y,Z) &:= \half [h(\r_\tau \trl X) + \r_\tau \trl (h(X)),  (\r^\tau \trl \Phi)(Y,Z)] + \Phi(X, [Y,Z]) + \mbox{b.c.p.} \, . \nn \\
&\:=  \half \big( \lq h , \Phi \rq  - \lq  \Phi, h \rq  \big) (X,Y,Z) 
+ \big(\Phi (X, [Y,Z] ) + \mbox{b.c.p.}\big)
\end{align}
using \eqref{etaPhieta} for the second equality. 
These definitions can be generalised to forms of any degree,  obtaining braided versions of the classical formulas \cite[eq.~(I.1)]{cat71}, 
\cite[eq.~(19)]{Fer02}.

From \eqref{bcur} and \eqref{hcur} the curvature is written in terms of the horizontal projection as 
\beq
\label{hcurh}
\Omega_\vp(Y,Y') = h([Y,Y']  ) - \lq h, h \rq (Y,Y').
\eeq
Then one shows the following proposition. 
   \begin{prop} \label{prop:bianchi}
  For the connection and the curvature there is a structure equation 
$$
\Omega_\vp =  D \vp + \lq \omega, \omega \rq 
$$
  and a Bianchi identity
\beq\label{bi0}
  D \Omega_\vp = \half \lq \r_\alpha \trl h , \lq \r^\alpha \trl h , h \rq  \rq .
\eeq
The covariant derivative $D \Omega_\vp$ vanishes for an equivariant connection.
\end{prop}
\begin{proof}
From the structure equation \eqref{streq} and \eqref{der-1form}, and 
from the definition \eqref{cov1form} of the covariant derivative, we have
\begin{align*}
\Omega_\vp = 2 \, \lq \id, \omega \rq  - \omega\circ [~, ~] - \lq \omega, \omega \rq = 
2 \lq h, \omega \rq  - \omega\circ [~, ~] + \lq \omega, \omega \rq
= D\omega + \lq \omega, \omega \rq .
\end{align*}

Then, 
for $X,Y,Z \in P$, from definitions \eqref{cov2form} and \eqref{hcurh} one computes
\begin{align*}
D \Omega_{\vp}(X,Y,Z)  & = \half \big[h(\r_\gamma \trl X) + \r_\gamma \trl (h( X)), (\r^\gamma \trl h) ([Y, Z]) \big]
   \\
 & \quad - 
 \half \big[h(\r_\gamma \trl X) + \r_\gamma \trl (h( X)), (\r^\gamma \trl \lq h , h \rq ) (Y \ot Z) \big]
 \\ 
 & \quad + h ([X, [Y,Z]  ] )    - \lq h , h \rq  (X \ot [Y,Z] ) + \mbox{b.c.p.} \, .
\end{align*}
The term $h ([X, [Y,Z]  ] ) + \mbox{b.c.p.}$ vanishes by Jacobi identity while the first and fourth terms cancel each other. 
The remaining one, the second term, then gives
$$
D \Omega_{\vp} = 
- \half \big( \lq h , \lq h , h \rq  \rq  - \lq \lq h, h \rq , h  \rq  \big) 
= \half \lq \r_\alpha \trl h , \lq \r^\alpha \trl h , h \rq  \rq. 
$$
The second equality follows from the Jacobi identity \eqref{jacobi-forms}.
When the connection is equivariant, $ k \trl h = \varepsilon(k) h$ and the right-hand side of 
\eqref{bi0} vanishes by Jacobi identity, 
$$
\big( \lq h , \lq h , h \rq  \rq  \big) (X,Y,Z) = [ h(X) , [ h(Y) , h(Z) ]  ]  + \mbox{b.c.p.} = 0 .
$$
Then the relation in \eqref{bi0} is the usual Bianchi identity. 
\end{proof} 
 An easy computation also gives that the covariant derivative of the curvature as in \eqref{cov2form} can be written as
$$
D \Omega_\vp = d \Omega_{\vp} - \half ( \lq \omega , \Omega_{\vp} \rq  - \lq \Omega_\vp, \omega \rq ).
$$
 
\begin{rem}
Our result in \eqref{bi0} compares with a similar one in
\cite[\S~5.4]{RZ21}. In the (dual) notation of \cite{Du} it says that
our connection is regular but not necessarily multiplicative.  
\end{rem}
 
\subsection{The space of connections and the gauge transformations} \label{se:scgt}
The space $C(T,\g)$ of connections on the sequence \eqref{as-gen} is an affine space modelled on the linear space of  right $B$-module maps $\eta : T \to \g$ (the $\g$-valued one-forms on $B$). Indeed, given a connection $\rho : T \to P$ and such a map $\eta : T \to \g$, one has
 $$
\pi \circ (\rho + \eta)  =\pi \circ \rho = \id_T  
$$
and thus $\rho' = \rho + \eta$ is a connection as well, with 
$\omega_{\rho'} = \omega_{\rho} - \eta \circ \pi$.
In the examples of the present paper
we shall use this decomposition with $\rho$ a $K$-equivariant connection. 

The space of connections $C(T,\g)$ is a subset of the $K$-module $\Hom_B(T, P)$. The latter carries an action of the braided Lie algebra $P$ (cf. \cite[\S 5.3]{pgc-main}). 

\begin{prop}\label{prop:actPcon}
The map $\delta: P \ot \Hom_B(T, P) \to \Hom_B(T, P)$ given by
\beq\label{actPcon}
(\delta_Y \rho) (X) := [Y, \rho(X)] - (\r_\alpha\trl \rho) ([ \r^\alpha \trl Y^\pi, X] ) 
\eeq
is an action of the $K$-braided Lie algebra $P$, that is 
\beq\label{K-delta-rho}
k \trl (\delta_Y \rho) = \delta_{\one{k}\trl Y} (\two{k} \trl \rho) 
\eeq
and for $Y,Y' \in P$,  
\beq\label{Pact}
\delta_{[Y,Y'] } = \delta_Y \circ \delta_{Y'} - \delta_ {\r_\alpha \trl Y'} \circ \delta_{\r^\alpha \trl Y}.
\eeq
\end{prop}
\begin{proof}
Formula \eqref{K-delta-rho} follows from the quasi-cocommutativity \eqref{iiR} and the explicit expression \eqref{action-hom-bis} for the $K$-action on morphisms. 
Then, using formula \eqref{K-delta-rho} we compute  
\begin{align*}
(\delta_Y  (\delta_{Y'} \rho)) (X) 
& = [Y, \delta_{Y'} \rho(X)] - (\r_\alpha\trl \delta_{Y'} \rho) ([ \r^\alpha \trl Y^\pi, X])  
\\
& = [Y, [Y', \rho(X)]] 
- [Y, (\r_\gamma\trl \rho) ([ \r^\gamma \trl Y'^\pi, X])]
\\
& \quad
- [\r_\alpha\trl Y', (\r_\beta \trl \rho)([ \r^\beta\r^\alpha \trl Y^\pi, X])  ] 
\\
& \quad
+ (\r_\gamma \r_\beta \trl \rho) ([ \r^\gamma \r_\alpha\trl Y'^\pi, [ \r^\beta\r^\alpha \trl Y^\pi, X]])  .
\end{align*}
This also yields
\begin{align*}
(\delta_{\r_\alpha \trl Y'}  (\delta_{\r^\alpha \trl Y} \rho)) & (X) 
= [\r_\alpha \trl Y', [\r^\alpha \trl Y, \rho(X)]] 
- [\r_\alpha \trl Y', (\r_\gamma\trl \rho) ([ \r^\gamma \r^\alpha \trl  Y^\pi, X])]
\\
& 
- [ Y, (\r_\beta \trl \rho)([ \r^\beta  \trl  Y'^\pi, X])  ] 
+ (\r_\gamma \r_\beta \trl \rho) ([ \r^\gamma   \trl Y^\pi, [ \r^\beta   \trl Y'^\pi, X]])  \; .
\end{align*}
These two expressions have two terms in common so
their difference is given by
\begin{align*}
\delta_Y  (\delta_{Y'} \rho)& (X) - \delta_{\r_\alpha \trl Y'}  (\delta_{\r^\alpha \trl Y} \rho) (X) 
\\
& = [Y, [Y', \rho(X)]] - [\r_\alpha \trl Y', [\r^\alpha \trl Y, \rho(X)]] 
\\
& \quad
+ (\r_\gamma \r_\beta \trl \rho) \Big([ \r^\gamma \r_\alpha\trl Y'^\pi, [ \r^\beta\r^\alpha \trl Y^\pi, X]]- [ \r^\gamma   \trl Y^\pi, [ \r^\beta   \trl Y'^\pi, X]] \Big) 
\\
& = [Y, [Y', \rho(X)]] - [\r_\alpha \trl Y', [\r^\alpha \trl Y, \rho(X)]] 
\\
& \quad
+ (\r_\gamma \r_\beta \trl \rho) \Big([  \r_\alpha  \r^\beta\trl Y'^\pi, [\r^\alpha \r^\gamma \trl Y^\pi, X]]- [ \r^\gamma   \trl Y^\pi, [ \r^\beta   \trl Y'^\pi, X]] \Big) 
\end{align*}
using Yang--Baxter equation for the second summand. Finally Jacobi identity gives 
\begin{align*}
\delta_Y  (\delta_{Y'} \rho)(X) - \delta_{\r_\alpha \trl Y'}  (\delta_{\r^\alpha \trl Y} \rho) (X)
& = [[Y, Y'], \rho(X)] 
- ( \r_\beta \trl \rho) \Big( [ \r^\beta   \trl [     Y^\pi, Y'^\pi], X] \Big) 
 \end{align*}
which is just $(\delta_{[Y,Y']} \rho) (X)$ and coincides with the right hand side of \eqref{Pact}.
\end{proof}

The braided Lie algebra action of $P$ on $\Hom_B(T, P)$ gives rise to a map 
$$
P \times C(T,\g) \longrightarrow C(T,\g) , \quad (Y, \rho)  \mapsto \rho' = \rho + \delta_Y \rho .
$$
Indeed, from $\pi \circ \rho = \id_T$,  it follows that $\pi \circ (\delta_Y \rho) = 0$ and so $\delta_Y \rho: T \to \g$.

Proposition \eqref{prop:actPcon} can be generalised.
Let $M$ be a right $B$-module with a compatible $P$-action $\trl_{\scriptscriptstyle{P}}$. Then 
there is a $P$-action  $\delta: P \ot \Hom_B(M, P) \to \Hom_B(M, P)$,
\beq
(\delta_Y \Phi) (m) := [Y, \Phi(m)] - (\r_\alpha\trl \Phi ) \big( (\r^\alpha \trl Y^\pi ) \trl_{\scriptscriptstyle{P}} m \big) . 
\eeq
Since the $P$-action on $T\ot T$ is  by braided derivations, when $\Phi$ is a $P$-valued two-form on $T$, the above becomes 
\begin{align}\label{var-phi-gen}
(\delta_Y \Phi) (X, X') 
= [Y, \Phi(X, X')]  
& - (\r_\alpha\trl \Phi) ([ \r^\alpha\trl Y^\pi, X] , X')  
\\ \nn &  
 -(\r_\alpha\trl \Phi)\big([ \r_\lambda \trl X, [\r^\lambda  \r^\alpha\trl Y^\pi, X']  \big) .  
\end{align} 

\begin{prop} 
The variation $\delta_Y \Omega$ of the curvature  of the connection $\rho$ {for the action of an element $Y \in P$}, as defined in \eqref{var-phi-gen}, explicitly reads 
$$
\delta_Y \Omega = (\delta_Y \rho) \circ [~,~ ] -  \lq \delta_Y \rho, \rho \rq - \lq \r_\alpha \trl \rho, \delta_{\r^\alpha \trl Y}  \rho \rq .
$$
\end{prop}
\begin{proof}
This follows from linearity of the action $\delta_Y$ and from its braided derivation rule. 
 \end{proof}
{
In analogy with the classical case, an element $V \in  \g$ acts on a connection and on the corresponding curvature as an infinitesimal gauge transformation:
\begin{cor}
The variation \eqref{actPcon} of a connection $\rho$ for the action of a vertical element $V\in \g$ is   given by
$$
(\delta_V \rho) (X) = [V, \rho(X)]  \, ,
$$
while the variation \eqref{var-phi-gen} of the curvature  is  
$$
(\delta_V \Omega) (X, X') = [V, \Omega(X, X')]  .
$$
\end{cor}
Then in view of \eqref{Pact}, $\g$ is the braided Lie algebra of 
infinitesimal gauge transformations. The universal enveloping algebra $\U(\g)$ is the braided Hopf algebra of such transformations \cite[\S 5.3]{pgc-main}. }

\begin{rem}\label{lem:var-curv}
Let $\Omega'$ be the curvature of the transformed connection 
$\rho' = \rho + \delta_Y \rho$, for $Y\in P$. Since the action $\delta_Y$ is a braided derivation, the variation  $\delta_Y\Omega$ 
differs from the first order term in the difference $\Omega'-\Omega$ (which by construction is a derivation):
\beq\label{diff-om}
\Omega' - \Omega = \delta_Y \Omega + \lq  \r_\beta \trl \rho, \delta_{\r^\beta \trl Y} \rho \rq  - \lq  \rho, \delta_Y \rho \rq + \lq \delta_Y \rho, \delta_Y \rho \rq  . 
\eeq
\end{rem}

When the splitting $\rho$ is $K$-equivariant, $k \trl \rho = \varepsilon(k) \rho$, the variations of the connection and of the curvature reduce to 
\begin{align}
(\delta_Y \rho) (X) &= [Y, \rho(X)]   - \rho([   Y^\pi, X] )   
\nn \\[.2cm]
\delta_Y \Omega ( X,X')& =
[Y, \Omega( X,X')]   - \Omega([Y^\pi, X]   , X')
- \Omega(\r_\alpha \trl X, [\r^\alpha \trl  Y^\pi, X']  ) .
\label{var-cur-eq}
\end{align}
 Moreover the extra term in the right-hand side of \eqref{diff-om} vanishes.

\section{Atiyah sequences for Hopf--Galois extensions}\label{sec:AS-HG}
 
Recall that an algebra $A$ is a right $H$-comodule algebra for a Hopf algebra $H$ 
if it carries a right coaction $\delta: A \to A \ot H$ which is a morphism of algebras.  We write $\delta(a)= \zero{a} \ot \one{a}$ in Sweedler notation with an implicit sum.   The subspace of coinvariants $B:= A^{coH}=\big\{b\in A ~|~ \delta (b) = b \otimes 1_H \big\}$ is a subalgebra of $A$.  
There is a canonical  map
$$
 \chi := (m \ot \id) \circ (\id \otimes_B \delta ) : A \otimes_B A  \longrightarrow A \ot H~  ,  \quad a' \ot_B a \longmapsto a' a_{\;(0)} \ot a_{\;(1)} .
$$ 
The algebra extension $B\subseteq  A$ is called  an $H$-Galois extension if this map is bijective. 
 
In the present paper we deal with $H$-Galois extensions  which are
$K$-equivariant.
That is $A$ also carries  a left $K$-action $\trl \, : K\otimes A \to A$, for $K$  a  Hopf algebra,  that commutes with the right $H$-coaction,
$ \delta \, \circ \trl = (\trl \ot \id) \circ (\id\otimes \delta)$ (the coaction $\delta$ is a $K$-module map where $H$ has trivial $K$-action).
On elements  $k\in K, a\in A$,
$$
\zero{(k \trl a)} \ot \one{(k \trl a)} = ( k \trl \zero{a} ) \ot \one{a} \; .
$$

Given a $K$-equivariant Hopf--Galois
extension $B=A^{co H} \subseteq A$, with triangular Hopf algebra $(K,\r)$, the Lie algebra $\Der{A}$ of $K$-braided derivations of $A$ has two distinguished Lie subalgebras.  Firstly, the Lie subalgebra of braided derivations that are $H$-equivariant  (that is $H$-comodule maps), 
\begin{align}\label{LieGG}
\DerH{A} = \{X \in \Der{A} \;|\;  \delta (X(a))=X(\zero{a})\otimes \one{a} \, , \, \,  a \in A \}  
\end{align}
and then its Lie subalgebra of vertical derivations  
$$
\aut{A} := \{X\in \DerH{A} \; | \; X(b)=0 \, ,  \, \,  b\in B\}~. 
$$
The linear spaces $\DerH{A}$ and  $\aut{A}$ are $K$-braided Lie subalgebras of $\Der{A}$,  \cite[Prop. 7.2]{pgc-main}. 
Elements of $\aut{A}$ are regarded  as infinitesimal gauge
transformations of the $K$-equivariant Hopf--Galois extension $B=A^{co H} \subseteq A$, \cite[Def. 7.1]{pgc-main}.

Each  derivation in $\DerH{A}$, being $H$-equivariant, restricts to a derivation on the subalgebra of coinvariant elements $B=A^{co H} $. Thus, associated to $B=A^{co H} \subseteq A$, there is the
sequence of braided Lie algebras
$\aut{A} \to \DerH{A} \to \DerR{B} $. 
When exact,
\beq\label{as}
0 \to \aut{A} \to \DerH{A} \to \DerR{B}   \to 0 
\eeq
 is a version of the Atiyah sequence of a (commutative) principal fibre bundle.

 {
When  the $K$-module algebra $B$ is quasi-central in $A$, see \eqref{qcAB}, 
the  braided Lie algebras  in the sequence are also $B$-bimodules with module structures as in \eqref{LAmodderA} and \eqref{RAmodderA}.   The above is then a sequence of 
braided Lie--Rinehart pairs as in \S \ref{se:seqrh}. 
}

 As studied in the previous section, a connection on the bundle can be given as an $H$-equivariant splitting of the sequence, associating to a derivation $X$ in $\DerR{B} $ a unique horizontal derivation in $\DerH{A} $ projecting onto $X$. The curvature of the connection  measures the extend to which this map fails to be a braided Lie algebra morphism.  
 
An example of this construction with the study of moduli spaces of connections is given in the next two subsections. In \S\ref{sec:riem-geo} we present  a sequence for the orthogonal frame bundle over the sphere $S^{2n}_\theta$ with a splitting which  leads  to the Levi-Civita connection. 

\subsection{The sequence for the $SU(2)$-bundle over the sphere $S^{4}_\theta$} \label{sec:ASib}
 
This section is devoted to the Atiyah sequence of braided Lie algebras associated with the
$\O(SU(2))$ 
Hopf--Galois extension $\O(S^{4}_\theta) \subset \O(S^{7}_\theta)$ constructed in \cite{LS0}
and related connections. 

Let $\theta \in \IR$.
The $*$-algebra $\mathcal{O}(S^7_\theta)$ has generators $z_r, z_r^*$, $r=1,2,3,4$ with commutation relations determined by the action of a $2$-torus.  
For $K$ the Hopf algebra generated by two commuting elements $H_1, H_2$,  the generators $z_r$  are eigenfunctions for 
the action of $H_1$, $H_2$ of eigenvalues 
 $\mu^r= (\mu^r_1, \mu^r_2)= \frac{1}{2}(1, -1), \, \frac{1}{2}(-1, 1), \, \frac{1}{2}(-1, -1),
 \, \frac{1}{2}(1, 1)$, for $r=1,2,3,4$, and so are their $*$-conjugated $z_r^*$ with  eigenvalues $-\mu^r$.
 Then, the commutation relations among the $z_r$ read
\beq\label{zcr}
 z_r  \dott z_s =  \lambda^{ 2 r \wedge s}  z_s  \dott z_r  
 \, , \qquad 
 \lambda = e^{-\pi  i\theta} 
  \, , \qquad 
  r \wedge s  := \mu^r_1 \mu^s_2 - \mu^r_2 \mu^s_1  \,,
\eeq
for $r,s= \pm 1, \pm 2, \pm 3, \pm4$ and $z_{-r}:=z_r^*$.
 In addition, the generators satisfy a sphere relation
 $z_1\dott z^*_1+z_2\dott z^*_2+z_3\dott z^*_3+z_4\dott z^*_4=1$.  
   
   The algebra $\mathcal{O}(S^7_\theta)$  is a right
$\O(SU(2))$-comodule algebra with right coaction which is defined  on the algebra generators as
\begin{eqnarray}\label{princ-coactSU2}
\delta:   \O(S^7_\theta) &\longrightarrow & 
\O(S^7_\theta) \ot \O(SU(2))
\\ \nn
\, \mathsf{u}
 &\longmapsto&
\mathsf{u}
\overset{.}{\otimes}
\mathsf{w}
\;, \quad 
\mathsf{u}:=
\begin{pmatrix}
z_1& z_2 & z_3& z_4
\vspace{2pt}
\\
-z_2^*  &
  z_1^* &
  -z_4^*
& z_3^*
\end{pmatrix}^t 
, \quad 
\mathsf{w}:=
\begin{pmatrix}
 w_1 & -w_2^*
\vspace{2pt}
\\
w_2 & w_1^*\end{pmatrix} .
\end{eqnarray}
Here $\overset{.}{\otimes}$ denotes the composition of the tensor product $\ot$ with matrix multiplication.
As usual the coaction is extended to the whole $\O(S^7_\theta)$ as a
$*$-algebra morphism.

The subalgebra   
$B=\O(S^7_\theta)^{co\O(SU(2))}$  
of coinvariant elements for the coaction $\delta$ is generated by the entries of the matrix $p:=\mathsf{u} \dott \mathsf{u}^\dagger$ and is identified with the algebra $\O(S^4_\theta)$ of coordinate  functions  on the 
 4-sphere $S^4_\theta$ in \cite{cl01}.  The $K$-action on  $\O(S^7_\theta)$
commutes with the $\O(SU(2))$-coaction  \eqref{princ-coactSU2}
and the sphere $\O(S^4_\theta)$ carries the induced action of $K$. We  denote the generators of $\O(S^4_\theta)$ 
by
$\bb_\mu $, with eigenvalues $\mu=(\mu_1,\mu_2)=(0,0), (\pm 1,0), (0,\pm 1)$  of the action of $H_1$, $H_2$.  
 Explicitly,
\begin{align*} 
\bb_{10} &:= \sqrt{2}(z_1 \dott z_3^* + z^*_2 \dott z_4), \quad 
\bb_{01}:= \sqrt{2}(z_2 \dott z_3^* - z^*_1 \dott z_4), \nn \\
\bb_{00} &:= z_1 \dott z_1^* + z_2 \dott z_2^* - z_3 \dott z_3^* -z_4 \dott z_4^* 
\end{align*}
and $\bb_{-\mu}= \bb_{\mu}^*$.
They have commutation relations  
\beq\label{bscr}
\bb_\mu \dott \bb_\nu = \lambda^{ 2 \mu \wedge \nu} \, \bb_\nu \dott \bb_\mu \, , \qquad \lambda = e^{-\pi  i\theta} \; , \quad \mu \wedge \nu  := \mu_1 \nu_2 - \mu_2 \nu_1
\eeq
and satisfy the sphere relation $\sum\nolimits_{\mu}  \bb_\mu^* \dott \bb_\mu =1$. 

From the above, one works out \cite[Eq.~(4.27)]{pgc-examples} the following mixed relations: 
  \begin{align}\label{regole-grado3-theta}
&   (1-\bb_{00}) \dott z_1 = \sqrt{2} (\bb_{10} \dott z_3 - z_4 \dott \bb_{0-1} ) \, , 
&&  (1-\bb_{00}) \dott z_2 = \sqrt{2} (z_4 \dott \bb_{-10} + \bb_{01} \dott z_3)  \, , 
\nn \\
&(1+\bb_{00}) \dott z_3 = \sqrt{2} (\bb_{-10} \dott z_1 + \bb_{0-1} \dott z_2) \, ,
&& (1+\bb_{00}) \dott z_4 = \sqrt{2} (z_2 \dott \bb_{10} - z_1 \dott \bb_{01})  \, .
\end{align}

\subsubsection{The braided Lie algebra ${\rm{Der}}^{ \r}(\O(S^4_\theta)) $}\label{sec:derS4theta}
{The commutation relations \eqref{bscr} among the generators $\bb_\mu$ 
of $\O(S^4_\theta)$
 can be written as 
$$
\bb_\mu \dott \bb_\nu = (\r_\alpha \trl \bb_\nu) \dott (\r^\alpha \trl \bb_\mu) \, , 
$$
with braiding implemented by the triangular structure of $K$,
$$ 
\r:= e^{-2 \pi i\theta (H_1 \ot H_2 -H_2 \ot H_1)} .
$$
{which, as in Remark \ref{rem:topo}, belongs to a topological completion
of $K\otimes K$.} 
Similarly for the commutation relations \eqref{zcr} of the generators $z_r$ of $\O(S^7_\theta)$.
In particular, both of the algebras $\O(S^7_\theta)$ and $\O(S^4_\theta)$ are quasi-commutative.}
Then  
the braided Lie algebra ${\rm{Der}}^{ \r}(\O(S^4_\theta)) $ of braided derivations is a $\O(S^4_\theta)$-bimodule. 
As  a left $\O(S^4_\theta)$-module, it is generated by operators  $ \ch_\mu $  defined  on the algebra generators as
\beq\label{gen-derS4theta}
\ch_\mu (\bb_\nu) := \delta_{\mu \nu^* } - \bb_\mu \dott \bb_\nu  \; , \qquad \mu =(0,0), (\pm 1,0), (0,\pm 1)~.
\eeq 
These are extended to the whole algebra $\O(S^4_\theta)$ as braided derivations: 
$$
 \ch_\mu (\bb_\nu \dott \bb_\tau) =  \ch_\mu (\bb_\nu) \dott \bb_\tau + \lambda^{ 2 \mu \wedge \nu} \bb_\nu \dott  \ch_\mu (\bb_\tau) .
 $$
  It is easy to verify that 
$$
 \ch_\mu \, (\sum\nolimits_\nu \bb_\nu^* \dott \bb_\nu) =0 \, ,
$$
 showing that  the 
 $ \ch_\mu$'s are well-defined as derivations of  $\O(S^4_\theta)$.
 Moreover, using the sphere relation, one also sees that they are not independent but rather satisfy   the relation
\beq\label{relazione-horbis}
\sum\nolimits_\mu b^*_\mu \, \ch_\mu = 0 
\eeq
for the left module structure as in \eqref{LAmodderA}.
The action of $H_1$, $H_2$, the generators of $K$, on the generators $b_\mu$ lifts to  the adjoint action  \eqref{action-der} on the derivations 
$ \ch_\mu$:
\beq\label{actT}
H_j \trl \ch_\mu = [H_j,\ch_\mu] = \mu_j  \ch_\mu
\eeq
being $[H_j,\ch_\mu] (\bb_\nu)= H_j \trl (\ch_\mu (\bb_\nu))- \ch_\mu (H_j \trl \bb_\nu)$.

The derivations $\ch_\mu$ are the simplest combinations of the five basis derivations of
$\O(\mathbb{R}^5_\theta)$  that preserve the sphere. 
In the classical limit $\theta=0$, the derivations $ \ch_\mu$ reduce
to $T_\mu= \partial_{\mu^*} - \bb_\mu D$,  
for $D = \sum\nolimits_\mu \bb_\mu \partial _{\mu}$ the Liouville vector field.
The five weights $\mu$ are those of the five-dimensional representation $[5]$ of $so(5)$. Indeed, the vector fields $H_\mu$ carry such a representation.  
 The bracket in ${\rm{Der}}^{\r}(\O(S^4_\theta))$ is the
braided commutator in \eqref{Abla}. Using \eqref{moduloLie}  it is determined by its computation on
the generators. 
 \begin{prop}\label{prop:Ders4theta}
The braided Lie algebra structure of  ${\rm{Der}}^{\r}(\O(S^4_\theta))$ is given by 
\begin{align}\label{bracket-s4t}
[\ch_\mu , \ch_\nu]  & = \bb_\mu  \, \ch_\nu - \lambda^{ 2 \mu \wedge \nu} \bb_\nu \,  \ch_\mu   \; .
\end{align}
\end{prop}
 \begin{proof}
We compute the braided commutator of two generators $\ch_\mu,\ch_\nu$ of 
${\rm{Der}}^{\r}(\O(S^4_\theta))$: 
\begin{align*}
[\ch_\mu , \ch_\nu]  (\bb_\tau)
&= 
(\ch_\mu \circ  \ch_\nu -  \lambda^{ 2 \mu \wedge \nu}  \ch_\nu \circ  \ch_\mu )(\bb_\tau)
\\ &=
-  \ch_\mu( \bb_\nu \dott \bb_\tau) + \lambda^{ 2 \mu \wedge \nu}  \ch_\nu (\bb_\mu \dott \bb_\tau)
\\ &=
- \big(\ch_\mu (\bb_\nu) \dott \bb_\tau + \lambda^{ 2 \mu \wedge \nu} \bb_\nu \dott  \ch_\mu (\bb_\tau)  \big)
+ \lambda^{ 2 \mu \wedge \nu}  \big(  \ch_\nu (\bb_\mu) \dott \bb_\tau + \lambda^{ 2 \nu \wedge \mu} \bb_\mu \dott  \ch_\nu (\bb_\tau)  \big)
\\ &=
-  \ch_\mu (\bb_\nu) \dott \bb_\tau - \lambda^{ 2 \mu \wedge \nu} \bb_\nu \dott  \ch_\mu (\bb_\tau)  
+ \lambda^{ 2 \mu \wedge \nu}     \ch_\nu (\bb_\mu) \dott \bb_\tau +  \bb_\mu \dott  \ch_\nu (\bb_\tau) 
\\ &=
- \lambda^{ 2 \mu \wedge \nu} \bb_\nu \dott  \ch_\mu (\bb_\tau)  
 +  \bb_\mu \dott  \ch_\nu (\bb_\tau) \; ,
\end{align*}
for  $\mu, \nu, \tau =(0,0), (\pm 1,0), (0,\pm 1)$. 
For the last equality we have used that $\nu^* \wedge \nu =0$ implies 
$\ch_\mu (\bb_\nu)=  \lambda^{ 2 \mu \wedge \nu}  \ch_\nu (\bb_\mu) $ for each $\mu, \nu$.  Then,  identity \eqref{bracket-s4t} is verified.
\end{proof}

We conclude this subsection with two Lemmas that we use later on.
We first observe that  formula \eqref{bracket-s4t}  allows us to write the derivations 
$\ch_\mu$ in terms of their commutators:
\begin{lem}\label{lem:vf-hor}
The generators $\ch_\mu$ can  be expressed in terms  of their commutators
as 
$$
\ch_\nu  = \sum\nolimits_{\mu }   b^*_\mu \,  \, [\ch_\mu , \ch_\nu]  .
$$
\end{lem}
\begin{proof}
From the expression \eqref{bracket-s4t} of the commutators, we compute
\begin{align*}
\sum\nolimits_{\mu }   b^*_\mu \,  [\ch_\mu , \ch_\nu]  
& = \sum\nolimits_{\mu }   b^*_\mu \dott \bb_\mu  \,  \ch_\nu -  \sum\nolimits_{\mu } \lambda^{ 2 \mu \wedge \nu}  b^*_\mu \dott  \bb_\nu \,  \ch_\mu
\\ & =  \ch_\nu -  \bb_\nu \,  \sum\nolimits_{\mu }    b^*_\mu \,    \ch_\mu = \ch_\nu \, ,
\end{align*}
where we  used  the sphere relation for the second equality and   \eqref{relazione-horbis} for the last one. 
\end{proof}
For the braided antisymmetric commutators we introduce the notation
\beq \label{vf-hor2}
\wl_{\mu, \nu}^\pi:= [\ch_\mu , \ch_\nu]  =  - \lambda^{ 2 \mu \wedge \nu} \wl_{\nu, \mu}^\pi \, .
\eeq
The action \eqref{actT} of $H_1$ and $H_2$ on the derivations $T_\mu$ implies 
$$
H_j \trl \wl_{\mu, \nu}^\pi = (\mu_j  + \nu_j) \wl_{\mu, \nu}^\pi.
$$

From  \eqref{bracket-s4t}, on the generators of   $\O(S^4_\theta)$ these derivations are given by  
\beq\label{vf-hor3}
\wl_{\mu, \nu}^\pi (\bb_\sigma) =  \bb_\mu \delta_{\nu^* \sigma}  - \lambda^{2 \mu \wedge \nu } \bb_\nu \delta_{\mu^* \sigma}  \, .
\eeq
\begin{prop}\label{prop:crdown}
The derivations  $\wl_{\mu, \nu}^\pi$  
{ give a faithful  representation of }
the braided Lie algebra $so_\theta(5)$, that is:
$$
[\wl_{\mu, \nu}^\pi, \wl_{\tau, \sigma}^\pi]  
= \delta_{\nu^* \tau }  \wl_{\mu, \sigma}^\pi - \lambda^{ 2 \mu \wedge \nu}   \delta_{\mu^* \tau } \wl_{\nu, \sigma}^\pi
  - \lambda^{ 2 \tau \wedge \sigma} ( \delta_{\nu^* \sigma }   \wl_{\mu, \tau}^\pi  - \lambda^{2 \mu \wedge \nu }  \delta_{\mu^* \sigma}   \wl_{\nu, \tau}^\pi  )
$$
\end{prop}
\begin{proof}
From \eqref{vf-hor3} a direct computation yields:
\begin{align}\label{hbc}
\wl_{\mu, \nu}^\pi \circ \wl_{\tau, \sigma}^\pi (b_\alpha) 
=  ( \bb_\mu \delta_{\nu^* \tau}  - \lambda^{2 \mu \wedge \nu } \bb_\nu \delta_{\mu^* \tau} ) \delta_{\sigma^* \alpha}
- \lambda^{2 \tau \wedge \sigma} ( \bb_\mu \delta_{\nu^* \sigma}  - \lambda^{2 \mu \wedge \nu } \bb_\nu \delta_{\mu^* \sigma} ) \delta_{\tau^* \alpha} 
\end{align}
and the one with indices exchanged
\begin{align}\label{hbc1}
\wl_{\tau, \sigma}^\pi \circ \wl_{\mu, \nu}^\pi (b_\alpha) 
=  ( b_\tau \delta_{\sigma^* \mu}  - \lambda^{2 \tau \wedge \sigma } b_\sigma \delta_{\tau^* \mu} ) 
\delta_{\nu^* \alpha}
- \lambda^{2 \mu \wedge \nu} ( b_\tau \delta_{\sigma^* \nu}  - \lambda^{2 \tau \wedge \sigma } b_\sigma \delta_{\tau^* \nu} ) \delta_{\mu^* \alpha} .
\end{align}
The braided commutator is just given by the terms in \eqref{hbc} minus the terms in \eqref{hbc1} with an extra factor coming from the braiding, that is $\lambda^{2(\mu + \nu) \wedge (\tau + \sigma)}$. The result is obtained by pairing the terms with the same $\delta$'s. For instance, the one coming from 
$\delta_{\nu^* \tau} = \delta_{\tau^* \nu}$ is 
$$
\bb_\mu \delta_{\sigma^* \alpha} - \lambda^{2(\mu + \nu) \wedge (\tau + \sigma) 
+ 2 (\mu \wedge \nu + \tau \wedge \sigma)} b_\sigma \delta_{\mu^* \alpha} = 
\bb_\mu \delta_{\sigma^* \alpha} - \lambda^{2(\mu \wedge \sigma)} b_\sigma \delta_{\mu^* \alpha}  
= \wl_{\mu, \sigma}^\pi (b_\alpha)
$$
using that  $\tau=-\nu$ to simplify the exponent of $\lambda$. 
The other terms behave similarly. 
\end{proof}
In the classical limit, the derivations $\wl_{\mu, \nu}^\pi$ give the representation [10] of $so(5)$, thus 
the name $so_\theta (5)$ for the braided Lie algebra generated by the  deformed generators
$\wl_{\mu, \nu}^\pi$. 

 \subsubsection{The sequence and the connection} \label{sec:spS7}

For the $SU(2)$ bundle on the sphere $S^{4}_\theta$ we have   the short 
sequence of braided Lie algebras  
\beq\label{as-S7theta}
0 \to \mathrm{aut}_{\O(S^4_\theta) }^{\r}(\O(S^7_\theta)) \stackrel{\imath}{\to} {\rm{Der}}_{\M^H}^{\r}(\O(S^7_\theta)) \stackrel{\pi}{\to}  {\rm{Der}}^{\r}(\O(S^4_\theta))    \to 0  \, . 
\eeq
The algebra ${\rm{Der}}^{\r}(\O(S^4_\theta))$ was given in \S \ref{sec:derS4theta}. 
The braided Lie algebra ${\rm{Der}}^{\r}_{\M^H}(\O(S^7_\theta)) $ of derivations 
of $\O(S^7_\theta)$ which are equivariant in the sense of \eqref{LieGG} is also given as an 
$\O(S^4_\theta)$-module generated by a set of elements. We use `partial derivatives' 
$\partial_r$,  $\partial_r^*= \partial_{-r}$ defined by  $\partial_r(z_s) 
= \delta_{r s}$,  for $r, s=\pm 1, \pm 2,\pm 3, \pm 4$, on the algebra generators and extended as braided derivations. 
There are two commuting elements (the Cartan subalgebra) 
\begin{align}\label{der-sopra-theta0}
&\wl_{1}=\wl_{10,-10} := \tfrac{1}{2}( z_1   \partial_1 - z_1^* \,  \partial^*_1 - z_2 \,  \partial_2  + z_2^* \,  \partial^*_2 - z_3 \,  \partial_3 + z_3^* \,  \partial^*_3 + z_4 \,  \partial_4  - z_4^* \,  \partial^*_4) 
\nn
\\
&\wl_{2}=\wl_{01,0-1} :=  \tfrac{1}{2}(- z_1 \,  \partial_1 + z_1^* \,  \partial^*_1 + z_2 \,  \partial_2  - z_2^* \,  \partial^*_2 - z_3 \,  \partial_3 + z_3^* \,  \partial^*_3 + z_4 \,  \partial_4  - z_4^* \,  \partial^*_4) .
\end{align}
The remaining generators can be given, for $\mu, \nu =(0,0), (\pm 1,0), (0,\pm 1)$, as
$$
\wl_{\mu + \nu} = \wl_{\mu, \nu} \; , \qquad \mu \neq \nu^* , \; \mu < \nu   \; ,
$$
with the  lexicographic order $(-1,0)<(0,-1)<(0,0)< (0,1)<(1,0)$. 
Then, the vectors $ \mathsf{r}=(\mu + \nu)  \in \{(\pm 1 , 0), (0, \pm 1), (\pm1,\pm 1)\}$ are the `root vectors' of the braided Lie algebra $so_\theta(5)$. Explicitly, 
\begin{align}\label{der-sopra-theta}
  \wl_{10} &= \wl_{00,10} :=  \stwo (  z_1 \,  \partial_3  - \lambda\, z_3^* \,  \partial^*_1 - \lambda\, z_4 \,  \partial_2 +   z_2^* \,  \partial^*_4)
\nn
\\   
\wl_{-10} &=\wl_{-10,00} := \stwo (   z_3 \,  \partial_1 - \lambda^{-1}  z_1^* \,  \partial^*_3 - \lambda^{-1}  z_2 \,  \partial_4  +   z_4^* \,  \partial^*_2)
\nn
\\
 \wl_{01} &= \wl_{00,01} := \stwo (   z_2 \,  \partial_3  - \lambda^{-1}  z_3^* \,  \partial^*_2 +\lambda^{-1}   z_4 \,  \partial_1 -   z_1^* \,  \partial^*_4) 
\nn
\\    
\wl_{0-1}&=\wl_{0 -1,00} := \stwo (   z_3 \,  \partial_2  - \lambda \, z_2^* \,  \partial^*_3 + \lambda \,   z_1 \,  \partial_4  -  z_4^* \,  \partial^*_1 )
\nn
\\
 \wl_{11} &= \wl_{01,10}:=  - \lambda^{-1} z_4 \,  \partial_3  + \lambda^{-1} z_3^* \,  \partial^*_4 
\nn
\\    
\wl_{-1-1} &= \wl_{-10,0-1} := \lambda \, z_4^* \,  \partial^*_3  -  \lambda \, z_3 \,  \partial_4 
\nn
\\
 \wl_{1-1} &= \wl_{0-1,10}:= - \lambda^{2} z_1 \,  \partial_2  + \lambda^{2} z_2^* \,  \partial^*_1  
\nn
\\ 
\wl_{-11} & = \wl_{-10, 01}:= - \lambda^{-2} z_2 \,  \partial_1  + \lambda^{-2} z_1^* \,  \partial^*_2 . 
\end{align}
For $\mu >\nu$ we set $\wl_{\mu, \nu}=  - \lambda^{ 2 \mu \wedge \nu} \wl_{\nu, \mu}$.

\noindent
These derivations 
 { give a faithful  representation of } the braided Lie algebra $so_\theta(5)$:
\beq\label{crup}
[\wl_{\mu, \nu}, \wl_{\tau, \sigma}]  
= \delta_{\nu^* \tau }  \wl_{\mu, \sigma} - \lambda^{ 2 \mu \wedge \nu}   \delta_{\mu^* \tau } \wl_{\nu, \sigma}
  - \lambda^{ 2 \tau \wedge \sigma} ( \delta_{\nu^* \sigma }   \wl_{\mu, \tau}  - \lambda^{2 \mu \wedge \nu }  \delta_{\mu ^* \sigma}   \wl_{\nu, \tau}  ) .
\eeq
In particular, with the Cartan generators \eqref{der-sopra-theta0}, one finds that
$$
[\wl_{j}, \wl_{\mu + \nu}]  = [\wl_{j}, \wl_{\mu + \nu}] = (\mu + \nu)_j \wl_{\mu + \nu}, \qquad j=1,2.
$$
In the classical limit $\theta = 0$, the derivations $\wl_{\mu, \nu}$ or $\wl_{\mu + \nu}$ are a representation of the Lie algebra $so(5)$ as vector fields on the sphere $S^7$.

The braided Lie subalgebra $\mathrm{aut}^{\r}_{\O(S^4_\theta) }(\O(S^7_\theta)) \subset {\rm{Der}}^{\r}_{\M^H}(\O(S^7_\theta))$ consists of vertical derivations. 
It turns out it is generated, as  an  $\O(S^4_\theta)$-module, by the ten derivations
\begin{align}\label{ver-theta}
\wy_{\mu, \nu} &:= \sum\nolimits_{\gamma } \bb_\gamma^* \dott \Big( \bb_\gamma \,   \wl_{\mu, \nu}  - \lambda^{- 2 \mu \wedge \gamma} \bb_\mu \,    \wl_{\gamma, \nu} + \lambda^{  -  2\nu \wedge (\gamma + \mu)}  \, \bb_\nu \,     \wl_{\gamma, \mu} \Big) \nn \\
 &\:= \wl_{\mu, \nu} -
  \sum\nolimits_{\gamma,  \tau } \Big(   \bb_\mu \dott \bb_\gamma^* \, \, 
  \delta_{\nu \tau}
 - \lambda^{  2  \mu \wedge  \nu}  \,  \bb_\nu \dott \bb_\gamma^* \,\,  \delta_{\mu \tau} 
 \Big) \,  \wl_{\gamma, \tau}
\end{align}
for $\mu, \nu $  and sums on $\gamma, \tau$ in $\{(0,0), (\pm 1,0), (0,\pm 1)\}$.
It is direct to see that these derivations are vertical, that is $\wy_{\mu, \nu}(b) = 0$, for $b\in \O(S^4_\theta)$. The proof that they generate $\mathrm{aut}^{\r}_{\O(S^4_\theta) }(\O(S^7_\theta))$ follows the corresponding classical results in \cite[Prop.~4.1]{pgc-examples}.

The sequence \eqref{as-S7theta} is exact. Any derivation  $X \in {\rm{Der}}^{\r}_{\M^H}(\O(S^7_\theta))$, being $H$-equivariant, 
restricts to a derivation $X^\pi$ of the subalgebra $\O(S^4_\theta)$. 
This determines 
the  map $\pi: {\rm{Der}}^{\r}_{\M^H}(\O(S^7_\theta)) \to  {\rm{Der}}^{\r}(\O(S^4_\theta)) $.
By construction, $\ker{\pi}= \imath(\mathrm{aut}^{\r}_{\O(S^4_\theta) }(\O(S^7_\theta )))$. 
By a direct computation, one verifies that
the restrictions $\pi(\wl_{\mu, \nu}) = \wl_{\mu, \nu}^\pi$ of the derivations $\wl_{\mu, \nu}$ in 
\eqref{der-sopra-theta0}  and \eqref{der-sopra-theta} to $\O(S^4_\theta)$ 
coincide with the derivations $\wl_{\mu, \nu}^\pi= [\ch_\mu , \ch_\nu] $ defined in \eqref{vf-hor2} (thus the use of the same symbol).
Moreover,  from Lemma \ref{lem:vf-hor},    the derivations 
$\ch_\nu$   can be written as   $\ch_\nu = \sum\nolimits_{\mu }   b^*_\mu \,  \wl_{\mu, \nu}^\pi$, 
  also giving that 
\begin{align}\label{pi-hor-S7}
 \ch_\nu  = \pi\big( \sum\nolimits_{\mu }    b^*_\mu \,  \wl_{\mu, \nu} \big) \; .
\end{align}
This shows the surjectivity of $\pi : {\rm{Der}}^{\r}_{\M^H}(\O(S^7_\theta)) 
\to {\rm{Der}}^{\r}(\O(S^4_\theta))$, and thus the exactness of the sequence in \eqref{as-S7theta} (since the elements of $\mathrm{aut}^{\r}_{\O(S^4_\theta) }(\O(S^7_\theta))$ are vertical). 
 We know from the general theory in \S \ref{se:scgt} that the braided Lie algebra 
 $\mathrm{aut}^{\r}_{\O(S^4_\theta) }(\O(S^7_\theta))$ is that of infinitesimal gauge transformations.
 
In the limit $\theta=0$, the sequence of Lie algebras in \eqref{as-S7theta} is  the Atiyah sequence of the $SU(2)$-principal Hopf bundle $S^7 \to S^4$, with $\mathrm{aut}^{\r}_{\O(S^4) }(\O(S^7))$ the (infinite dimensional classical)  Lie algebra of infinitesimal gauge transformations.

\begin{rem}\label{rem:pgcII} 
The braided Lie algebras in the sequence \eqref{as-S7theta} were obtained in \cite{pgc-examples}
from a twist deformation quantization of the corresponding Lie algebras of the $SU(2)$-classical fibration $S^7 \to S^4$.  The braided Lie algebra ${\rm{Der}}^{\r}_{\M^H}(\O(S^7_\theta))$
was generated by derivations  $\widetilde{H}_j , \, j=1,2$, and $\widetilde{E}_\mathsf{r}$ 
for $\mathsf{r}$ the roots of the Lie algebra $so(5)$. 
 These are related to the generators $\wl_{\mu, \nu} $ by
$$
\wl_{j} = \widetilde{H}_j \, , \quad j = 1,2 , \qquad 
\wl_{\mu, \nu} =   \lambda^{\mu \wedge \nu} \widetilde{E}_{\mu + \nu}  
\, , \quad \mu \neq \nu^* , \; \mu < \nu   \; ,
$$
with the previous lexicographic order on the $\mu$'s.  
 Moreover,
$$
\wl_{j}^\pi = \widetilde{H}_j^\pi \, , \quad j = 1,2 , \qquad 
\wl_{\mu, \nu}^\pi = \lambda^{\mu \wedge \nu} \varphi_\mu \varphi_\nu  \widetilde{E}_{\mu + \nu}^\pi  
\, , \quad  \mu \neq \nu^* , \; \mu < \nu   \; ,
$$
for their restrictions, where
$ \varphi_{00}:=1, \, \varphi_{\pm1 0}:= \lambda^{\mp \half} , \, \varphi_{0 \pm 1}:= \lambda^{\pm \half} $
(see \cite[Remark 4.7]{pgc-examples}).
 \end{rem}
Moreover, each of the non vanishing terms in
parenthesis in \eqref{ver-theta} is just one of the derivations $\widetilde{K}$'s and $\widetilde{W}$'s in \cite[Prop.~4.14]{pgc-examples}. Then, the ten vertical derivations 
$\wy_{\mu, \nu}$ are those associated to the quantization  
 $\wy_{(11)} \propto \wl_{(0,1)+(1,0)}$  of the
operator $Y_{11}$, the highest weight vector of the representation
$[10]$. 

\medskip
A connection on the $SU(2)$ Hopf--Galois extension $\O(S^{4}_\theta) \subset \O(S^7_\theta)$ is 
a splitting
 of its Atiyah sequence \eqref{as-S7theta}. This results into an equivariant direct sum decomposition of 
 ${\rm{Der}}^{ \r }_{\M^H}(\O(S^7_\theta))$ in horizontal and vertical components.
\begin{prop}
The right $\O(S^4_\theta)$-module map $\rho : {\rm{Der}}^{ \r}(\O(S^4_\theta)) \to {\rm{Der}}^{ \r}_{\M^H}(\O(S^7_\theta)) $, defined 
on  the generators $ \ch_\nu$ of $ {\rm{Der}}^{ \r }(\O(S^4_\theta))$ as 
\beq\label{hor-lift0}
\rho(\ch_\nu) := \sum\nolimits_{\mu }   \bb_\mu^* \,  \wl_{\mu, \nu} \; 
\eeq
is a right {splitting}  of the sequence \eqref{as-S7theta}.
\end{prop}
\begin{proof}
From \eqref{pi-hor-S7}, it is immediate to see that $\pi \circ \rho=\id$.
\end{proof}

From Remark \ref{rem:eq}, this connection is left $\O(S^4_\theta)$-linear as well since it is equivariant:
 \begin{prop} \label{prop:equiv-conn}
The connection $\rho$ is invariant under the action (by braided commutators) of  the braided
Lie algebra $so_\theta(5)$ of $\O(SO_\theta(4))$-equivariant derivations on $\O(SO_\theta(5))$:  
for every $\wl_{\mu,\nu} \in \DerH{\O(S_\theta^7)}$ we have  $ (\delta_{\wl_{\mu,\nu}} \rho)(X)=0$, for $X\in  {\rm{Der}}^{\r}(\O(S^{4}_\theta))$,  that is 
$$
 [\wl_{\mu,\nu} ,\rho(X)] -\rho([\wl_{\mu,\nu}^\pi , X] )=0 .
$$
 \end{prop}
 \begin{proof}
 From right $\O(S^{4}_\theta)$-linearity it is enough to show that 
$$
 [\wl_{\mu,\nu}, \rho(\ch_\sigma)] -\rho([\wl_{\mu,\nu}^\pi , \ch_\sigma] )=0 ,
$$
for all generators 
$\ch_\sigma$ of $ {\rm{Der}}^{\r}(\O(S^{4}_\theta))$.

From \eqref{moduloLie} a simple computation gives $ [\bb_\mu  \,  \ch_\nu  ,  \ch_\sigma]   
  = \bb_\mu  \dott  \bb_ \nu  \,  \ch_ \sigma
- \lambda^{2   \nu \wedge \sigma} \delta_{\sigma^* \mu }   \ch_\nu
$. 
Thus, using $\wl_{\mu, \nu}^\pi= [\ch_\mu , \ch_\nu]  =  \bb_\mu  \,  \ch_\nu - \lambda^{ 2 \mu \wedge \nu} \bb_\nu \, \ch_\mu$ from \eqref{bracket-s4t}, we have
\begin{align*}
[\wl^\pi_{\mu,\nu} , \ch_\sigma]  
&= 
[\bb_\mu  \,  \ch_\nu  ,  \ch_\sigma]    - \lambda^{ 2 \mu \wedge \nu} [\bb_\nu \,  \ch_\mu,  \ch_\sigma] 
\\ & =
 \bb_\mu  \dott  \bb_ \nu  \,  \ch_ \sigma
- \lambda^{2   \nu \wedge \sigma} \delta_{\sigma^* \mu }   \ch_\nu
- \lambda^{ 2 \mu \wedge \nu} \bb_\nu  \dott  \bb_\mu \, \ch_ \sigma
+ \lambda^{ 2 \mu \wedge (\nu + \sigma)}\delta_{\sigma^* \nu }   \ch_\mu
\\ & =
\delta_{\sigma^* \nu } \ch_\mu - \lambda^{2   \mu \wedge \nu} \delta_{\sigma^* \mu }   \ch_\nu \; .
\end{align*}
Then 
$$
\rho([\wl_{\mu,\nu}^\pi , \ch_\sigma] ) = 
 \sum\nolimits_{  \tau}    \bb_\tau^* \,  \left( \delta_{\sigma^* \nu }    \wl_{\tau , \mu}  
-  \lambda^{2 \mu \wedge \nu } \delta_{\sigma^* \mu }    \wl_{\tau, \nu}   \right)
. 
$$

\noindent
On the other hand, using \eqref{vf-hor3} and the braided commutator in \eqref{crup}, we compute
\begin{align*}
[ \wl_{  \mu,  \nu} \,  ,\rho(\ch_ \sigma)]  
& = \sum\nolimits_{\tau } \, \big[  \wl_{  \mu,  \nu}  ,   \bb_\tau^* \,  \wl_{\tau,  \sigma} \big]  
\\ & =
  \sum\nolimits_{  \tau}  \left(    \wl_{  \mu,  \nu} ( \bb_\tau^*) \,  \wl_{\tau,  \sigma}
+  \lambda^{-2  (\mu +\nu) \wedge \tau}\bb_\tau^* \,     [\wl_{  \mu,  \nu} , \wl_{\tau,  \sigma}]  \right)
\\ &=   
\sum\nolimits_{\tau}   ( \bb_\mu \delta_{\nu \tau}  - \lambda^{2 \mu \wedge \nu } \bb_\nu \delta_{\mu \tau} ) \,  \wl_{\tau,  \sigma}
\\ & 
\hspace{-2.2cm} + \sum\nolimits_{\tau} \bb_\tau^* \,    \left(    \lambda^{-2  \mu \wedge \tau}  \delta_{\nu^* \tau }  \wl_{\mu, \sigma} -  \lambda^{2  (\mu + \tau) \wedge \nu}  \delta_{\mu^* \tau } \wl_{\nu, \sigma}  
  - \lambda^{-2  \mu \wedge \tau}  \delta_{\nu^* \sigma }   \wl_{\mu, \tau} 
 + \lambda^{2  (\mu + \tau) \wedge \nu } \delta_{\sigma^* \mu }   \wl_{\nu, \tau}   
 \right)
\\ &=   
 \bb_\mu \,   \wl_{\nu,  \sigma}  - \lambda^{2 \mu \wedge \nu } \bb_\nu   \,  \wl_{\mu,  \sigma}
+    \lambda^{2  \mu \wedge \nu}    \bb_\nu \,       \wl_{\mu, \sigma} 
-    \bb_\mu \,  \wl_{\nu, \sigma} 
  \\ & 
  \quad
  + \sum\nolimits_{\tau} \bb_\tau^* \,    \left( - \lambda^{-2  \mu \wedge \tau}  \delta_{\nu^* \sigma }   \wl_{\mu, \tau} + \lambda^{2  (\mu + \tau) \wedge \nu} \delta_{\sigma^* \mu }   \wl_{\nu, \tau}   \right)
\\
 & 
= \sum\nolimits_{\tau} \bb_\tau^* \,    \left( - \lambda^{-2  \mu \wedge \tau}  \delta_{\nu^* \sigma }   \wl_{\mu, \tau} + \lambda^{2  (\mu + \tau) \wedge \nu} \delta_{\sigma^* \mu }   \wl_{\nu, \tau}   \right).
 \end{align*}
This expression coincides   with  $\rho([\wl_{\mu,\nu}^\pi , \ch_\sigma] )$    due to the braided antisymmetry   property 
  $\wl_{\mu, \nu}=  - \lambda^{ 2 \mu \wedge \nu} \wl_{\nu, \mu}$.
 \end{proof}
As mentioned, the splitting gives an equivariant  direct sum decomposition of 
 \begin{align}\label{dec-S7}
{\rm{Der}}^{\r}_{\M^H}(\O(S^7_\theta)) 
& = \rho({\rm{Der}}^{\r}(\O(S^4_\theta)) )  \oplus \imath (\mathrm{aut}^{\r}_{\O(S^4_\theta) }(\O(S^7_\theta)) )
\nn \\ 
& \simeq  
 {\rm{Der}}^{\r}(\O(S^4_\theta))   \oplus  \mathrm{aut}^{\r}_{\O(S^4_\theta) }(\O(S^7_\theta) ) 
\end{align}
in horizontal and vertical parts.
For the corresponding vertical projection, we have the following.
\begin{cor}
The $\O(S^4_\theta)$-module map 
$\vp: {\rm{Der}}^{\r}_{\M^H}(\O(S^7_\theta)) \to \mathrm{aut}^{\r}_{\O(S^4_\theta) }(\O(S^7_\theta))$ 
$$
\vp( \wl_{\mu, \nu} ) := 
\wl_{\mu, \nu} - \rho(\wl_{\mu, \nu}^\pi)
=
\wl_{\mu, \nu}  -  \Big(  \bb_\mu \,  \rho(\ch_\nu)  - \lambda^{2  \mu \wedge \nu}  \, \bb_\nu \,  \rho(\ch_\mu) \Big)
$$
is a left {splitting}  of the sequence \eqref{as-S7theta};  $\vp$
is the vertical projection. 
\end{cor}

The decomposition \eqref{dec-S7} is then given, equivalently, by the
idempotent $(\imath \circ \vp)^2 = \imath \circ \vp$.  By
construction, the horiziontal derivations are the kernel  of the
vertical projection  $\vp$. A direct check uses the sphere relation and \eqref{relazione-horbis}: 
\begin{align*}
\vp \circ \rho (\ch_\nu) 
&= \sum\nolimits_{\mu }   
 \bb_\mu^* \,  \vp (\wl_{\mu, \nu} )
\\
& =  \sum\nolimits_{\mu }    
\bb_\mu^* \,  \wl_{\mu, \nu}  - 
\sum\nolimits_{\mu }   \bb_\mu^* \dott  \bb_\mu \,  \rho(\ch_\nu)  + \sum\nolimits_{\mu }    \lambda^{2 \mu \wedge \nu}  \bb_\mu^*   \dott \bb_\nu \,  \rho(\ch_\mu)  
\\
& =   \rho(\ch_\nu) -  \rho(\ch_\nu)   +      \bb_\nu \, \rho \, ( \sum\nolimits_{\mu } \bb_\mu^* \,  \ch_\mu)  
=0 \; .
\end{align*} 
The derivations $\vp( \wl_{\mu, \nu} )$ generate the module $\mathrm{aut}^{\r}_{\O(S^4_\theta) }(\O(S^7_\theta))$.
More explicitly: 
 \begin{align*}
\vp( \wl_{\mu, \nu} ) 
&= \wl_{\mu, \nu}  -  \bb_\mu \,  \rho(\ch_\nu)  + \lambda^{ 2 \mu \wedge \nu}  \, \bb_\nu \,  \rho(\ch_\mu)
\\ 
&= \sum\nolimits_{\tau }   \bb_\tau^* \dott \bb_\tau \,   \wl_{\mu, \nu} - \bb_\mu \,  \sum\nolimits_{\tau }   \bb_\tau^* \,  \wl_{\tau, \nu} + 
\lambda^{ \mu \wedge \nu}  \, \bb_\nu \,  \sum\nolimits_{\tau }   \bb_\tau^* \,  \wl_{\tau, \mu}
\\ 
 &= \sum\nolimits_{\tau } \bb_\tau^* \,  \Big( \bb_\tau \,   \wl_{\mu, \nu}  - \lambda^{- 2 \mu \wedge \tau} \bb_\mu \,    \wl_{\tau, \nu} + \lambda^{  -  \nu \wedge (\tau + \mu)}  \, \bb_\nu \,     \wl_{\tau, \mu}
\Big).
\end{align*}
The last sum is actually limited to $\tau \neq \mu, \nu$ since for these two choices for the index $\tau$ the term 
in parenthesis vanishes. 
Then, the ten vertical derivations 
$\vp(\wl_{\mu, \nu})$ coincide with the generators  $\wy_{\mu, \nu}$ in \eqref{ver-theta} of the
braided Lie algebra $\mathrm{aut}^{\r}_{\O(S^4_\theta) }(\O(S^7_\theta))$.

\medskip
The connection  $\rho$ assigns to every derivation $X$ on $\O(S^{4}_\theta)$  a unique 
horizontal derivation on $\O(S^7_\theta)$ that projects on $X$ via $\pi$. As mentioned, this needs not be a braided Lie algebra morphism in general: the commutators of horizontal vector fields need not be horizontal. A measure of this failure is the curvature \eqref{hcur}.
 \begin{prop}
On the generators $\ch_\mu$ of the braided Lie algebra ${\rm{Der}}^{\r}(\O(S^4_\theta)) $ the curvature 
of the connection in \eqref{hor-lift0} is given explicitly by
\beq\label{omega-mu-nu}
- \Omega(\ch_\mu, \ch_\nu) =  \vp (\wl_{\mu, \nu}) = \wl_{\mu, \nu}  -  \big( \bb_\mu \,  \rho(\ch_\nu)  - \lambda^{2 \mu \wedge \nu}  \, \bb_\nu \,  \rho(\ch_\mu) \big) 
{=\wy_{\mu, \nu}}\, ,
\eeq
{where the  $ \wy_{\mu, \nu}$ are the derivations defined in \eqref{ver-theta}.}
\end{prop}
\begin{proof}
We show that 
\beq\label{rrl}
\big[\rho(\ch_\mu), \rho(\ch_\nu)\big]  = \wl_{\mu, \nu}.
\eeq  
Using the module structure, we compute
\begin{align*}
&[\rho(\ch_\mu), \rho(\ch_\nu)]  = \big[ \sum\nolimits_{\sigma }   \bb_\sigma^* \,  \wl_{\sigma, \mu} \,  , \, \sum\nolimits_{\tau }   \bb_\tau^* \,  \wl_{\tau, \nu} \big]
\\
&=  \sum\nolimits_{\sigma , \tau}  \Big( \bb_\sigma^* \dott \wl_{\sigma, \mu} ( \bb_\tau^*) \,  \wl_{\tau, \nu}
- \lambda^{2 \mu \wedge \nu}\bb_\tau^* \dott \wl_{\tau, \nu} (\bb_\sigma^*) \,  \wl_{\sigma, \mu}
+ \lambda^{-2 \mu \wedge \tau}\bb_\tau^* \dott  \bb_\sigma^* \,  [\wl_{\sigma, \mu} , \wl_{\tau, \nu}]  
\Big) .
\end{align*}
Now, $\wl_{\mu, \nu} (\bb_\sigma)=\wl_{\mu, \nu}^\pi (\bb_\sigma)$ on the generators of $\O(S^4_\theta)$, and thus from   \eqref{vf-hor3} the first two terms in the above expression are equal to 
\begin{align*}
& \sum\nolimits_{\sigma , \tau}  \Big( \bb_\sigma^* \dott \big( \delta_{\mu \tau} \bb_\sigma - \delta_{\tau \sigma} \lambda^{2 \sigma \wedge \mu } \bb_\mu \big) \,  \wl_{\tau, \nu}
 - \lambda^{2 \mu \wedge \nu}\bb_\tau^* \dott \big( \delta_{\nu \sigma} \bb_\tau - \delta_{\tau \sigma} \lambda^{2 \tau \wedge \nu } \bb_\nu \big) \,  \wl_{\sigma, \mu}
\Big)
\\
&=  \wl_{\mu, \nu}
- \sum\nolimits_{\tau}   \bb_\tau^* \dott     \lambda^{2 \tau \wedge \mu } \bb_\mu  \,  \wl_{\tau, \nu} 
 -    \lambda^{2 \mu \wedge \nu}  \wl_{\nu, \mu}
+ \sum\nolimits_{\tau}  \lambda^{2 \mu \wedge \nu}\bb_\tau^* \dott   \lambda^{2 \tau \wedge \nu } \bb_\nu \,  \wl_{\tau, \mu}
\\
& =  2 \wl_{\mu, \nu}
-     \bb_\mu  \dott  \sum\nolimits_{\tau} \bb_\tau^*  \,  \wl_{\tau, \nu} 
 + \lambda^{2 \mu \wedge \nu }   \bb_\nu \dott  \sum\nolimits_{\tau}  \bb_\tau^* \,     \wl_{\tau, \mu} .
\end{align*}
As for the third term,  using the braided commutator in \eqref{crup}, one gets
\begin{align*} 
\lambda^{-2 \mu \wedge \tau}\bb_\tau^* \dott  \bb_\sigma^* \,  [\wl_{\sigma, \mu} , \wl_{\tau, \nu}]  & = 
\sum\nolimits_{\sigma }  \lambda^{2 \sigma  \wedge \mu}\bb_\sigma^* \dott  \bb_\mu \,  
\wl_{\sigma, \nu}  -  \wl_{\mu, \nu} \\
 & - \delta_{\mu^* \nu }  \sum\nolimits_{\sigma , \tau}  \lambda^{-2 \sigma  \wedge \tau}\bb_\sigma^* \dott  \bb_\tau^* \,       \wl_{\sigma, \tau}  + \sum\nolimits_{  \tau}  \lambda^{-2 ( \mu \wedge \tau + \nu \wedge \mu ) }    \bb_\nu \dott  \bb_\tau^* \,    \wl_{\mu, \tau}  .
\end{align*}
 Putting the three terms together one then arrives at 
\begin{align*}
[\rho(\ch_\mu), \rho(\ch_\nu)]  =  \wl_{\mu, \nu} 
- \delta_{\mu^* \nu }  \sum\nolimits_{\sigma , \tau}    \bb_\tau^* \dott  \bb_\sigma^* \,    
    \wl_{\sigma, \tau} .
\end{align*}
For the last term, the $\O(S^4_\theta)$-linearity of the horizontal lift and the identity \eqref{relazione-horbis} give 
\begin{align*}
\sum\nolimits_{\sigma , \tau}    \bb_\tau^* \dott  \bb_\sigma^* \,        \wl_{\sigma, \tau}   
=  \sum\nolimits_{\tau}    \bb_\tau^* \,   \rho(\ch_\tau)
=  \rho \big( \sum\nolimits_{\tau}    \bb_\tau^* \,   \ch_\tau \big) =0  \, .
\end{align*}
Thus, $[\rho(\ch_\mu), \rho(\ch_\nu)]  = \wl_{\mu, \nu}$ as claimed.
\end{proof}

\subsubsection{The connection one-form}
The connection on the $SU(2)$-bundle $\O(S^4_\theta) \subset \O(S^7_\theta)$
given by splitting the sequence \eqref{as-S7theta} corresponds to a Lie algebra valued one-form on the bundle.
Indeed, the partial isometry 
$\mathsf{u}$ in \eqref{princ-coactSU2} gives a projection
$p=\mathsf{u} \dott \mathsf{u}^*$ and a canonical covariant derivative
$\nabla = p \circ d $ on the module of sections
$\Gamma=p\;\!(\O(S^4_\theta)^4)$ of the vector bundle
 associated to the fundamental representation of $SU(2)$
 \cite{LS0}. When translated to the principal bundle this corresponds to an $su(2)$-valued one-form 
$\omega : \DerH{\O(S^7_\theta)} \to su(2)$. Explicitly
\beq\label{bein}
\omega = - d \mathsf{u}^* \dott  \mathsf{u} 
= \omega_{22} \begin{pmatrix} 1 & 0 \\ 0 & -1 \end{pmatrix} 
+ \omega_{21} \begin{pmatrix} 0 & 0 \\ 1 & 0 \end{pmatrix}
- \omega^*_{21} \begin{pmatrix} 0 & 1 \\ 0 & 0 \end{pmatrix}
\eeq
with one-forms 
\begin{align*}
\omega_{22} & = d z_1\dott  z^*_1 + d z_2\dott  z^*_2  + d z_3\dott  z^*_3  + d z_4 \dott  z^*_4  , \\
\omega_{21} & = - d z_1\dott  z_2  + d z_2 \dott  z_1 - d z_3 \dott  z_4 + d z_4 \dott  z_3  .
\end{align*}
The elements $dz_a, dz_a^*$, $a=1,\dots, 4$, are the degree $1$
generators of the graded differential algebra  $\Omega(S^7_\theta)$ of
the canonical differential calculus $(\Omega(S^7_\theta), d) $ on the
algebra $\O(S^7_\theta)$ \cite{cdv}. The commutation relations of the
$dz$'s with the $z$'s are the same as those of the $z$'s. The pairing between derivations
and one-forms is defined by
$$
\langle~,~\rangle: {\rm{Der}}^{\r}(\O(S^7_\theta)) \otimes
\Omega^1(S^7_\theta)\to \O(S^7_\theta)~,~~\langle X, d a
\dott a'\rangle:=X(a)\dott a' \, .
$$

By construction $\omega$ transforms under the adjoint coaction of
  $\O(SU(2))$: 
  $$
  \delta \omega_{j k}=\sum\nolimits_{s,t}\omega_{s t} \ot \mathsf{w}^*_{j s} \mathsf{w}_{t k} ,
  $$ for 
  $\mathsf{w} = (\mathsf{w}_{j k}) $ as in \eqref{princ-coactSU2}, with $\mathsf{w}_{j k}\in \O(SU(2))$. 
The form $\omega$ is the identity on the fundamental derivations of $\O(S^7_\theta)$ defined by
the $\O(SU(2))$-coaction 
 and is vertical, that is, it vanishes on any horizontal derivation in \eqref{hor-lift0}, 
$$
\langle \rho(\ch_\nu) , \omega \rangle = - (\rho(\ch_\nu)(\mathsf{u}^*)) \dott \mathsf{u} = 0 .
$$
This can also be seen by a direct computation using the expressions
 for the horizontal derivations in Table   \ref{table1} below (with slightly different notation: $\alpha = \sqrt{2} b_{10}$, $\beta = \sqrt{2} b_{01}$ and their conjugated)  
 and using the relations \eqref{regole-grado3-theta}.

\noindent Then, for instance,
\begin{align*}
2 \langle \rho(\ch_{00}), \omega_{22} \rangle & = (1-x)\dott (z_1 \dott z_1^*+ z_2 \dott z_2^*) - (1+x)\dott  (z_3 \dott z_3^*+ z_4\dott  z_4^*) 
\\
&= x - x \dott  \sum z_a \dott z_a^* =0
\\
2 \langle \rho(\ch_{00}), \omega_{21} \rangle &= (1-x) \dott (- z_1 \dott z_2+ z_2 \dott  z_1) - (1+x)\dott  (-z_3 \dott z_4+ z_4 \dott z_3)=0  \; .
\end{align*}
Likewise, using also the commutation relations between the $z_a$'s and the $\bb_\mu$'s, we compute
\begin{align*}
 2 \sqrt{2} \, \langle \rho(\ch_{10}), \omega_{22} \rangle & =  
 -  \alpha \dott  z_1   \dott z_1^*
+ \lambda (1-x) \dott  z_4   \dott z_2^* - \lambda  z_ 1 \dott \beta   \dott z_2^*  
\\ &
\quad +  (1+x) \dott  z_1  \dott z_3^* -   z_4 \dott \beta^* \dott z_3^*
 -  \lambda z_4 \dott \alpha  \dott  z_4^*
\\ &= -  [ 
  \alpha \dott  z_1   \dott z_1^*
- (1-x) \dott  z_2^* \dott z_4  +  z_ 1    \dott (z_2^*  \dott \beta)
\\ &
\quad -(1+x) \dott  z_1  \dott z_3^* +  (z_4 \dott \beta^*) \dott z_3^*
 + \alpha  \dott  z_4 \dott  z_4^*]
\\
&= 
-  [ 
  \alpha \dott  z_1   \dott z_1^*
- (1-x) \dott  z_2^* \dott z_4  +  z_ 1    \dott ((x+1) \dott z_3^* -  z_1^* \dott \alpha)
\\ &
\quad -(1+x) \dott  z_1  \dott z_3^* +  (- (1-x) \dott z_1 + \alpha \dott z_3 ) \dott z_3^*
 + \alpha  \dott  z_4 \dott  z_4^*]
\\
&= 
-  [ 
(x-1) \dott  ( z_2^* \dott z_4 +   z_1 \dott z_3^*)
+ \alpha \dott (z_3  \dott z_3^*+  z_4 \dott  z_4^*)]
\\
&= 
- \tfrac{1}{2} [ 
(x-1) \dott  \alpha
+ \alpha \dott (1-x)]=0
\end{align*} 
and
\begin{align*}
2 \sqrt{2}  \, \langle \rho(\ch_{10}), \omega_{21} \rangle &= 
 \alpha \dott  z_1 \dott z_2
+ \lambda(1-x) \dott  z_4  \dott  z_1  - \lambda z_ 1 \dott \beta \dott  z_1
\\ & \quad 
-    (1+x) \dott  z_1   \dott z_4 +   z_4 \dott \beta^*  \dott z_4
-  \lambda z_4 \dott \alpha   \dott z_3 
\\
&= 
  \alpha \dott  z_1 \dott z_2
+    (1-x) \dott  z_1   \dott z_4  -   z_ 1  \dott  z_1 \dott \beta
\\ & \quad 
-    (1+x) \dott  z_1   \dott z_4 +   (z_4 \dott \beta^* ) \dott z_4
-    \alpha   \dott z_4 \dott  z_3 
\\
&= 
   \alpha \dott  z_1 \dott z_2
- 2x \dott  z_1   \dott z_4  -   z_ 1  \dott  z_1 \dott \beta
\\ &\quad +   (\alpha \dott z_3 - (1-x) \dott z_1 ) \dott z_4
-   \alpha   \dott z_4 \dott  z_3 
\\
&= 
 z_1 \dott [ z_2  \dott \alpha - 2x \dott z_4  -  z_1 \dott \beta - (1-x)   \dott z_4]
\\
&= 
 z_1 \dott [ (x+1) \dott z_4  - 2x \dott z_4    - (1-x)   \dott z_4]
=0  .
\end{align*}
The computations for the other horizontal derivations are similar.

\begin{table}[H]\caption{Horizontal vector fields on $\O(S^7_\theta)$}   
\label{table1}
\begin{center}
\scalebox{0.67}{
\begin{tabular}{ c||c|c|c|c  }
\hline
 & $ z_1 $ & $ z_2 $ & $ z_3 $ & $ z_4    $
\\ \hline \rule[-4mm]{0mm}{1cm}
$2  \rho(\ch_{00}) $ & $    \alpha \dott  z_3 - z_4 \dott \beta^*   $ & $     \beta  \dott z_3  +   z_4 \dott \alpha^* $ & $   -   \beta^* \dott z_2 - \alpha^* \dott z_1  $ & $    z_1 \dott \beta  -  z_2  \dott \alpha  $
\\
 & $= (1-x) \dott  z_1 $ & $= (1-x) \dott  z_2 $ & $= -(1+x) \dott  z_3 $ & $ = -(1+x) \dott  z_4  $
\\ \hline  \rule[-4mm]{0mm}{1cm}  
$2  \sqrt{2} \rho(\ch_{01}) $ & $   \lambda^{-1}(  z_ 1 \dott \beta +  2 x  \dott  z_4 - 2  z_2 \dott \alpha)$ & $ -  \beta \dott  z_2  $ & $   \beta \dott  z_3 + 2x  \dott  z_2  +2  z_4  \dott \alpha^*  $ & $ -  \beta \dott  z_4$
\\
 & $= - \lambda^{-1}( (1-x) \dott  z_4 +  z_2 \dott \alpha)$ &  & $= (1+x) \dott  z_2 + z_4  \dott \alpha^*$ & 
\\ \hline  \rule[-4mm]{0mm}{1cm}  
$ 2 \sqrt{2}   \rho(\ch_{0-1}) $ & $   - \beta^* \dott  z_1   $ & $  \beta^* \dott  z_2 -  2 x  \dott  z_3 + 2 \alpha^*  \dott  z_1 $ & $   -  \beta^* \dott  z_3 $ & $ \lambda(z_4 \dott \beta^* - 2x  \dott  z_1  -2\alpha  \dott  z_3 )$
\\
 &  & $= (1 - x) \dott  z_3 + \alpha^* \dott  z_1 $ &  & $= -\lambda( (1+x) \dott  z_1 + \alpha \dott  z_3) $
\\ \hline  \rule[-4mm]{0mm}{1cm}  
$2 \sqrt{2} \rho(\ch_{10}) $ & $   -  \alpha \dott  z_1     $ & $ \lambda (z_2 \dott \alpha  -2 x \dott  z_4 - 2 z_ 1 \dott \beta)$ & $   ( \alpha \dott  z_3 + 2x  \dott  z_1  - 2 z_4 \dott \beta^* )$ & $ -  \alpha \dott z_4 $
\\
&  & $=\lambda ( (1-x) \dott  z_4 - z_ 1 \dott \beta)$ & $=  ( (1+x) \dott  z_1 - z_4 \dott \beta^* )$ & 
\\ \hline  \rule[-4mm]{0mm}{1cm} 
$2 \sqrt{2}  \rho(\ch_{-10}) $ & $ \alpha^* \dott  z_1   -2 x \dott  z_3 + 2 \beta^* \dott  z_2   $ & $ -    \alpha^* \dott z_2 $ & $  -   \alpha^* \dott  z_3  $ & $  \lambda^{-1} (z_4 \dott  \alpha^*  +2x \dott  z_2 + 2 \beta \dott  z_3 )$
\\
 & $=  (1-x) \dott  z_3 +  \beta^* \dott  z_2  $ &  &  & $ =\lambda^{-1} ( (1+x) \dott  z_2 +  \beta \dott  z_3)$
\end{tabular} 
}
\end{center}
\end{table}
 
 \subsection{The conformal algebra and its action on connections}
 
The `basic' connection $\rho$ we have described in the previous sections is an instanton in the sense that the curvature of the connection one-form in \eqref{bein} is anti-selfdual \cite{cl01,cdv}. An infinitesimal action of twisted  conformal transformations 
 yields a five parameter family of instantons \cite{LS1}. 
{We obtain here} a  five parameter family of splittings of the Atiyah 
sequence \eqref{as-S7theta} associated with the braided conformal Lie algebra $so_\theta(5,1)$. 

The braided Lie algebra $so_\theta(5,1)$ 
 is spanned as a linear space by generators $\ct_\mu$ and 
$\cl_{\mu, \nu} = - \lambda^{ 2 \mu \wedge \nu} \cl_{\nu, \mu}$, for 
 $\mu, \nu  =(0,0), (\pm 1,0), (0,\pm 1)$.  It is a $K$-module
with action  
 \beq\label{actc}
H_j \trl \ct_\mu = \mu_j  \ct_\mu \, , \qquad H_j \trl \cl_{\mu, \nu} = (\mu_j  + \nu_j) \cl_{\mu, \nu}
\eeq
and has braided brackets
\begin{align*}
[\ct_\mu , \ct_\nu]  & = \cl_{\mu, \nu} \, , \nn \\
[\cl_{\mu,\nu} , \ct_\sigma]  & = \delta_{\sigma^* \nu } \ct_\mu - \lambda^{2   \mu \wedge \nu} \delta_{\sigma^* \mu }   \ct_\nu\, , \nn \\
[\cl_{\mu, \nu}, \cl_{\tau, \sigma}]  
& = \delta_{\nu^* \tau }  \cl_{\mu, \sigma} - \lambda^{ 2 \mu \wedge \nu}   \delta_{\mu^* \tau } \cl_{\nu, \sigma}
  - \lambda^{ 2 \tau \wedge \sigma} ( \delta_{\nu^* \sigma }   \cl_{\mu, \tau}  - \lambda^{2 \mu \wedge \nu }   \delta_{\mu ^* \sigma}    \cl_{\nu, \tau}  )  \, .
\end{align*}
A representation as braided derivations on $\O(S^4_\theta)$ is given by the operators $\wl_{\mu,\nu}^\pi$ and $\ch_\sigma$ in \eqref{vf-hor3} and \eqref{gen-derS4theta} respectively. On the other hand, 
a representation as braided derivations on $\O(S^7_\theta)$ is given by the operators $\wl_{\mu,\nu}$ in \eqref{der-sopra-theta0} and \eqref{der-sopra-theta} together with the following additional generators  
\begin{align*}
\ct_{00} 
 &:= - \tfrac{1}{2}  b_{00} \Delta +  \tfrac{1}{2}    ( z_1 \,  \partial_1 + z_2 \,  \partial_2 + z_1^* \,  \partial^*_1   + z_2^* \,  \partial^*_2
 - z_3 \,  \partial_3 - z_4 \,  \partial_4 - z_3^* \,  \partial^*_3 - z_4^* \,  \partial^*_4) 
\nn \\ 
\ct_{10} &:= - \tfrac{1}{2} b_{10} \, \Delta
+ \tfrac{\sqrt{2}}{2} (\lambda \, z_4 \,  \partial_2  + z_1 \,  \partial_3 + \lambda \, z_3^* \,  \partial^*_1 
+z_2^* \,  \partial^*_4 ) 
\nn \\  
\ct_{01} &:= - \tfrac{1}{2} b_{01}  \, \Delta
+ \tfrac{\sqrt{2}}{2} (- \lambda^{-1} z_4 \,  \partial_1  + z_2 \,  \partial_3 +  \lambda^{-1}  z_3^* \,  \partial^*_2 
-z_1^* \,  \partial^*_4 ) 
\nn \\ 
\ct_{-10} &:= - \tfrac{1}{2} b_{-10}  \, \Delta
+ \tfrac{\sqrt{2}}{2} (z_3 \,  \partial_1 
+\lambda^{-1} \,  z_2 \,  \partial_4 +   z_4^* \,  \partial^* _2  + \lambda^{-1} \,  z_1^*  \,  \partial^* _3 ) 
\nn \\ 
\ct_{0-1} &:= - \tfrac{1}{2} b_{0-1}  \, \Delta
+ \tfrac{\sqrt{2}}{2} (  z_3 \,  \partial_2  - \lambda z_1 \,  \partial_4 -   z^* _4 \,  \partial^* _1  + \lambda z^* _2 \,  \partial^* _3   ) ,
\end{align*}
with $\Delta = \sum\nolimits_r (z_r \,  \partial_r + z_r^* \,  \partial^*_r)$ the Liouville derivation on $\O(S^7_\theta)$. 

A direct computation shows that  they project to the derivations $\ch_\mu$ on $S^4_\theta: \ct_{\mu}^\pi= \ch_\mu$.
Moreover, using the mixed relations \eqref{regole-grado3-theta} one also computes 
that $\sum\nolimits_{\tau} b_\tau^* \, \ct_{\tau} = 0 $.

We know from Proposition \ref{prop:equiv-conn} that the basic connection in \eqref{hor-lift0} is invariant under the action of the braided Lie algebra generated by the elements $\cl_{\mu,\nu}$. 
The remaining five generators $\ct_\mu$ will give new, not gauge equivalent connections 
$\rho_\mu = \rho + \delta_\mu \rho$, with 
$$
(\delta_\mu \rho) (X) := (\delta_{\ct_\mu} \rho) (X) = [\ct_{\mu} , \rho(X)] -\rho([\ct_{\mu}^\pi , X] ) , \quad X \in {\rm{Der}}^{\r}(\O(S^4_\theta)) .
$$   
While the connection $\rho$ is $K$-equivariant  the connections $\rho_\mu$ are not: due to \eqref{actc} one finds $H_j \trl \rho_\mu = \mu_j \, \delta_\mu \rho$. 
On the generators $\ch_\nu$ one computes, using \eqref{hor-lift0} and \eqref{rrl},
\begin{align*}
(\delta_\mu \rho) (\ch_\nu) 
& = -\rho([\ch_{\mu} , \ch_\nu] )  + [\ct_{\mu} , \rho(\ch_\nu)] \\
& = -\rho([\ch_{\mu} , \ch_\nu] )  + \sum\nolimits_{\tau }   \ct_{\mu} (\bb_\tau^*) \,  \wl_{\tau, \nu} 
+ \sum\nolimits_{\tau }   \lambda^{-2 \mu \wedge \tau} \bb_\tau^* \,  [\ct_{\mu} ,  \wl_{\tau, \nu} ]  
\\
& = -\rho([\ch_{\mu} , \ch_\nu] )  + \wl_{\mu, \nu}  -     \bb_{\mu}   \,  \rho(\ch_\nu)  
- \sum\nolimits_{\tau }   \lambda^{2 \mu \wedge \nu} \bb_\tau^* \,  
\left( \delta_{\mu^* \nu } \ct_\tau - \lambda^{2   \tau \wedge \nu} \delta_{\mu^* \tau }   \ct_\nu \right)
\\
& = -\rho([\ch_{\mu} , \ch_\nu] )  + \big[\rho(\ch_\mu), \rho(\ch_\nu)\big]   -     \bb_{\mu}   \,  \rho(\ch_\nu)  
+    \bb_\mu \,    \ct_\nu -   \delta_{\mu^* \nu } \sum\nolimits_{\tau }    \bb_\tau^* \,  
 \ct_\tau 
\\
& = - \Omega(\ch_\mu, \ch_\nu) + b_\mu \,  ( \ct_\nu - \rho(\ch_\nu) ) \,  \\
& = - \Omega(\ch_\mu, \ch_\nu) + b_\mu \,  \vp(\ct_\nu) \, .
\end{align*} 
The last equality follows from $\vp(\ct_\nu) = \ct_\mu - \rho(\ct_\mu^\pi) = 
\ct_\mu - \rho(\ch_\mu)$ and $\sum\nolimits_{\tau} b_\tau^* \, \ct_{\tau} = 0 $.   

\begin{lem}
On generators of ${\rm{Der}}^{\r}(\O(S^4_\theta)) $ the variation of the curvature is given by
\begin{align*}
(\delta_\sigma \Omega) (\ch_\mu, \ch_\nu) &= - 2 \, b_\sigma \Omega (\ch_\mu, \ch_\nu) 
- \ch_\sigma(b_\mu) \, \vp(\ct_\nu) + \lambda^{2 \mu \wedge \nu} \, \ch_\sigma(b_\nu) \, \vp(\ct_\mu) \\
&= - 2 \, b_\sigma \Omega (\ch_\mu, \ch_\nu) 
- \vp \Big( \ch_\sigma(b_\mu) \, \ct_\nu - \lambda^{2 \mu \wedge \nu} \, \ch_\sigma(b_\nu) \, \ct_\mu \Big) .
\end{align*}
\end{lem}
\begin{proof}
From the general theory, see \eqref{var-cur-eq}, the variation of the curvature is  
$$
(\delta_\sigma \Omega) (\ch_\mu, \ch_\nu) = [\ct_\sigma, \Omega(\ch_\mu, \ch_\nu)] - \Omega([\ch_\sigma, \ch_\mu], \ch_\nu)  
+ \lambda^{2 \mu \wedge \nu}\Omega([ \ch_\sigma, \ch_\nu],   \ch_\mu ).
$$
For the last two summands, using 
formula \eqref{bracket-s4t} for the commutator of the $\ch_\mu$ one gets: 
\begin{align*}
- \Omega([\ch_\sigma, \ch_\mu], \ch_\nu)  
&=
- \Omega \Big(\bb_\sigma  \,  \ch_\mu - \lambda^{ 2 \sigma \wedge \mu} \bb_\mu \,  \ch_\sigma , \ch_\nu \Big)
=
- \bb_\sigma  \,  \Omega( \ch_\mu   , \ch_\nu)
+ \lambda^{ 2 \sigma \wedge \mu} \bb_\mu \,  \Omega(  \ch_\sigma , \ch_\nu)
\end{align*}
and similarly
\begin{align*}
\lambda^{2 \mu \wedge \nu} \Omega([\ch_\sigma, \ch_\nu], \ch_\mu)  
=
\lambda^{2 \mu \wedge \nu}  \bb_\sigma  \,  \Omega( \ch_\nu   , \ch_\mu)
- \lambda^{2 (\mu + \sigma)\wedge \nu}   \bb_\nu \,  \Omega(  \ch_\sigma , \ch_\mu).
\end{align*}
For the first summand, from $\Omega(\ch_\mu, \ch_\nu) = - \vp (\wl_{\mu, \nu}) = -\wy_{\mu, \nu}$
and their explicit expression    in \eqref{ver-theta}, we compute
\begin{align*}
[\ct_\sigma, \Omega(\ch_\mu, \ch_\nu)] 
&= 
-[\ct_\sigma , \wy_{\mu, \nu}]
\\
&= 
-[\ct_\sigma , \wl_{\mu, \nu} ]
 +  
  \sum\nolimits_{\gamma,  \tau } \ct_\sigma  \Big(   \bb_\mu \dott \bb_\gamma^*    \, 
  \delta_{\nu \tau}
 - \lambda^{  2  \mu \wedge  \nu}  \,  \bb_\nu \dott \bb_\gamma^*   \delta_{\mu \tau} 
 \Big)     \wl_{\gamma, \tau}
  \\
 & \quad
 +
  \sum\nolimits_{\gamma,  \tau } \Big(  \lambda^{ 2 \sigma \wedge(  \mu - \gamma)}   \bb_\mu \dott \bb_\gamma^*    \, 
  \delta_{\nu \tau}
 - \lambda^{  2 \sigma \wedge(  \nu - \gamma) + 2  \mu \wedge  \nu}  \,  \bb_\nu \dott \bb_\gamma^*   \delta_{\mu \tau} 
 \Big)    [\ct_\sigma ,     \wl_{\gamma, \tau}]
 \\
&=  
-\delta_{\sigma^* \mu }   \ct_\nu + \lambda^{2   \sigma \wedge \mu} \delta_{\sigma^* \nu } \ct_\mu 
\\
& \quad
 +  
  \sum\nolimits_{\gamma,  \tau }  \Big(  
  (\delta_{\sigma^*  \mu } - \bb_\sigma \dott   \bb_\mu )  \dott \bb_\gamma^* + \lambda^{2 \sigma \wedge \mu}
\bb_\mu \dott  (\delta_{\sigma  \gamma } - \bb_\sigma \dott   \bb^*_\gamma ) 
  \Big)    \delta_{\nu \tau}   \wl_{\gamma, \tau}
 \\
& \quad
 -  \lambda^{  2  \mu \wedge  \nu} 
  \sum\nolimits_{\gamma,  \tau }    \Big( 
(\delta_{\sigma^*  \nu } - \bb_\sigma \dott   \bb_\nu )  \dott \bb_\gamma^* + \lambda^{2 \sigma \wedge \nu}
\bb_\nu \dott  (\delta_{\sigma  \gamma } - \bb_\sigma \dott   \bb^*_\gamma ) 
 \Big)    \delta_{\mu \tau}   \wl_{\gamma, \tau}
  \\
 & \quad
 +
  \sum\nolimits_{\gamma,  \tau } \Big(  \lambda^{ 2 \sigma \wedge(  \mu - \gamma)}   \bb_\mu \dott \bb_\gamma^*    \, 
  \delta_{\nu \tau}
 - \lambda^{  2 \sigma \wedge(  \nu - \gamma) + 2  \mu \wedge  \nu}  \,  \bb_\nu \dott \bb_\gamma^*   \delta_{\mu \tau} 
 \Big)   \delta_{\sigma^* \gamma }   \ct_\tau
 \\
 & \quad
 -
  \sum\nolimits_{\gamma,  \tau } \Big(  \lambda^{ 2 \sigma \wedge(  \mu - \gamma)}   \bb_\mu \dott \bb_\gamma^*    \, 
  \delta_{\nu \tau}
 - \lambda^{  2 \sigma \wedge(  \nu - \gamma) + 2  \mu \wedge  \nu}  \,  \bb_\nu \dott \bb_\gamma^*   \delta_{\mu \tau} 
 \Big)     \lambda^{2 \sigma \wedge \gamma} \delta_{\sigma^* \tau } \ct_\gamma
 \\
&=  
-\delta_{\sigma^* \mu }   \ct_\nu + \lambda^{2   \sigma \wedge \mu} \delta_{\sigma^* \nu } \ct_\mu 
\\
& \quad
 +  (\delta_{\sigma^*  \mu } - \bb_\sigma \dott   \bb_\mu )  \,  \rho(\ch_\nu) 
 + \lambda^{2 \sigma \wedge \mu}
\bb_\mu \,      \wl_{\sigma, \nu}
  -
  \lambda^{2 \sigma \wedge \mu}
\bb_\mu \dott   \bb_\sigma \,   \rho( \ch_\nu)
 \\
& \quad
 -  \lambda^{  2  \mu \wedge  \nu} (\delta_{\sigma^*  \nu } - \bb_\sigma \dott   \bb_\nu )  \,   \rho(\ch_\mu)
 -  \lambda^{  2 ( \mu +  \sigma) \wedge  \nu} 
\bb_\nu   \,      \wl_{\sigma, \mu}
 +  \lambda^{  2 ( \mu +  \sigma) \wedge  \nu}  \bb_\nu \dott    \bb_\sigma \,  
  \rho(\ch_\mu)
  \\
 & \quad
 +
  \lambda^{ 2 \sigma \wedge  \mu }   \bb_\mu \dott \bb_\sigma    \, 
   \,       \ct_\nu
 - \lambda^{  2 \sigma \wedge  \nu  + 2  \mu \wedge  \nu}  \,  \bb_\nu \dott \bb_\sigma    \,       \ct_\mu
 \\
 & \quad
 -
(  \lambda^{ 2 \sigma \wedge  \mu  }   
  \delta_{\nu \sigma^*} \bb_\mu       
 - \lambda^{  2 \sigma \wedge  \nu  + 2  \mu \wedge  \nu}  \,  \delta_{\mu \sigma^*} \bb_\nu   )\,    
   \sum\nolimits_{\gamma }  \bb_\gamma^* \,      \ct_\gamma
\end{align*}
where we used formulas  $[\ct_\sigma, \wl_{\mu,\nu} ] = \delta_{\sigma^* \mu }   \ct_\nu - \lambda^{2   \sigma \wedge \mu} \delta_{\sigma^* \nu } \ct_\mu $ for the braided commutators 
and 
$\ct_\sigma  (   \bb_\mu \dott \bb_\gamma^*)= (\delta_{\sigma^*  \mu } - \bb_\sigma \dott   \bb_\mu )  \dott \bb_\gamma^* + \lambda^{2 \sigma \wedge \mu}
\bb_\mu \dott  (\delta_{\sigma  \gamma } - \bb_\sigma \dott   \bb^*_\gamma ) 
$ for the third equality,
and \eqref{hor-lift0} for the fourth equality.
Finally, recalling that $\sum\nolimits_{\gamma }  \bb_\gamma^* \,       \ct_\gamma =0$, we have
\begin{align*}
[\ct_\sigma, \Omega(\ch_\mu, \ch_\nu)] 
&= 
-\delta_{\sigma^* \mu }  \big( \ct_\nu -  \rho(\ch_\nu) \big)
+ \lambda^{2   \sigma \wedge \mu} \delta_{\sigma^* \nu } \big(\ct_\mu 
 -   \rho(\ch_\mu)\big)
\\
& \quad
    -  \bb_\sigma \dott   \bb_\mu   \,  \rho(\ch_\nu) 
    + \bb_\sigma \dott   \bb_\mu   \,  \big(\ct_\nu - \rho(\ch_\nu) \big)
     \\
& \quad
 +   \lambda^{  2  \mu \wedge  \nu}   \bb_\sigma \dott   \bb_\nu   \,   \rho(\ch_\mu)
 -   \lambda^{  2  \mu \wedge  \nu}   \bb_\sigma \dott   \bb_\nu   \,  \big(   \ct_\mu - \rho(\ch_\mu) \big)
   \\
 & \quad
 + \lambda^{2 \sigma \wedge \mu}
\bb_\mu \,      \wl_{\sigma, \nu}
 -  \lambda^{  2 ( \mu +  \sigma) \wedge  \nu} 
\bb_\nu   \,      \wl_{\sigma, \mu}
\\
 &= 
(\bb_\sigma \dott   \bb_\mu-\delta_{\sigma^* \mu })   \,   \, \vp(\ct_\nu)
-   \lambda^{  2  \mu \wedge  \nu}   (\bb_\sigma \dott   \bb_\nu   
-  \delta_{\sigma^* \nu } )\,  \, \vp(\ct_\mu)    
\\
& \quad
    + \lambda^{2 \sigma \wedge \mu}
\bb_\mu \dott   (  \wl_{\sigma, \nu} -  \bb_\sigma     \, \rho(\ch_\nu) )
  -  \lambda^{  2 ( \mu +  \sigma) \wedge  \nu} 
\bb_\nu   \dott  (   \wl_{\sigma, \mu}
- \bb_\sigma  \, \rho(\ch_\mu))
\\
 &= 
- \ch_\sigma(b_\mu) \,     \vp(\ct_\nu)
+   \lambda^{  2  \mu \wedge  \nu}   \ch_\sigma(b_\nu) \,    \vp(\ct_\mu)    
\\
& \quad
    + \lambda^{2 \sigma \wedge \mu}
\bb_\mu \dott   (  \wl_{\sigma, \nu} -  \bb_\sigma     \,\rho(\ch_\nu) )
  -  \lambda^{  2 ( \mu +  \sigma) \wedge  \nu} 
\bb_\nu   \dott  (   \wl_{\sigma, \mu}
- \bb_\sigma  \,  \rho(\ch_\mu)).
\end{align*}
Finally, using \eqref{omega-mu-nu}, we obtain
\begin{align*}
(\delta_\sigma \Omega) (\ch_\mu, \ch_\nu) &=
- 2 \bb_\sigma  \,  \Omega( \ch_\mu   , \ch_\nu)
- \ch_\sigma(b_\mu) \,   \,   \, \vp(\ct_\nu)
+   \lambda^{  2  \mu \wedge  \nu}   \ch_\sigma(b_\nu) \, \,  \, \vp(\ct_\mu)
\\&
\quad
+ \lambda^{ 2 \sigma \wedge \mu} \bb_\mu \dott \Big( \Omega(  \ch_\sigma , \ch_\nu) 
 +  \wl_{\sigma, \nu} -  \bb_\sigma     \,  \rho(\ch_\nu) 
\Big)
\\
& \quad
- \lambda^{2 (\mu + \sigma)\wedge \nu}   \bb_\nu \dott \Big(
\Omega(  \ch_\sigma , \ch_\mu)
 +   \wl_{\sigma, \mu}
- \bb_\sigma  \,   \rho(\ch_\mu)
\Big)
\\
&=
- 2 \bb_\sigma  \,  \Omega( \ch_\mu   , \ch_\nu)
- \ch_\sigma(b_\mu) \,   \,   \, \vp(\ct_\nu)
+   \lambda^{  2  \mu \wedge  \nu}   \ch_\sigma(b_\nu) \, \,  \, \vp(\ct_\mu)
\\&
\quad
+ \lambda^{ 2 \sigma \wedge \mu} \bb_\mu \dott \Big( - \lambda^{2 \sigma \wedge \nu}  \, \bb_\nu \,  \rho(\ch_\sigma)   
\Big)
- \lambda^{2 (\mu + \sigma)\wedge \nu}   \bb_\nu \dott \Big(
- \lambda^{2 \sigma \wedge \mu}  \, \bb_\mu \,  \rho(\ch_\sigma)  
\Big)
\\
 &= 
 - 2 \, b_\sigma \Omega (\ch_\mu, \ch_\nu) 
- \ch_\sigma(b_\mu) \, \vp(\ct_\nu) + \lambda^{2 \mu \wedge \nu} \, \ch_\sigma(b_\nu) \, \vp(\ct_\mu) 
\end{align*}
hence concluding the proof.
\end{proof}

\section{The Riemannian geometry of the noncommutative  sphere $S^{2n}_\theta$} \label{sec:riem-geo}

We consider the  $\O(SO_\theta(2n,\mathbb{R}))$-Hopf--Galois extension $\O(S^{2n}_\theta) \subset \O(SO_\theta(2n+1,\mathbb{R}))$  on the noncommutative $\theta$-spheres 
$S^{2n}_\theta$.  In analogy with the classical case this is thought
of as the bundle of orthonormal frames on $S^{2n}_\theta$ via the
identification of derivations ${\rm{Der}}^{\r}(\O(S^{2n}_\theta))$ with 
sections of the associated bundle for the fundamental corepresentation of the Hopf algebra  $\O(SO_\theta(2n,\mathbb{R}))$ on the algebra $\O(\IR^{2n}_\theta)$.
An equivariant splitting of the Atiyah sequence of the
frame bundle then leads to an explicit and globally defined expression for the
Levi-Civita connection of the `round' metric on $\O(S^{2n}_\theta)$.
The corresponding curvature and Ricci tensors and the scalar curvature
are then computed.

The Hopf algebra
$\O(SO_\theta(2n+1,\mathbb{R}))$     
is  the noncommutative algebra generated  by the entries  of a matrix $N=(n_{I J})$, $I,J=1,\ldots 2n+1$, modulo the orthogonality conditions
\beq\label{ort-con-Q}
N^t \dott Q \dott N =Q \quad , \quad N \dott Q \dott N^t=Q
\eeq
and $\det(N)=1$ for the quantum determinant. There the matrix $Q$ is given as 
\beq\label{Qmatrix}
Q:= \left( \begin{smallmatrix}
0 & \II_n &0
\\
\II_n & 0 &0
\\
0 & 0 & 1
\end{smallmatrix} \right).
\eeq
 In slight more generality, we may assume  the matrix  $Q$ to be symmetric and have  a single entry equal to $1$ in each row. 
For fixed index $J$, we set $J'$ to be the unique index such that $Q_{JJ'}=1$. Clearly $(J')'=J$.
We denote by $0$ the unique  index $J\in\{1,\ldots , 2n+1\}$ such that $Q_{JJ}=1$ and by $\cI$ the subset consisting of the $n$ indices $J$ such that $J < J'$. 
For the choice of $Q$ in \eqref{Qmatrix} one has $J'=n +J$ for $J\leq n$,  $0=2n+1$ and $\cI = \{1, \dots, n\}$.
In components, the orthogonality conditions
 give
\beq\label{orth-cond-comp}
\sum\nolimits_K n_{K' I'} \dott n_{K J} = \delta_{I J} \, ,  \quad 
\sum\nolimits_K n_{I K} \dott n_{J' K'} = \delta_{I J} \; .
\eeq

The commutation relations  among the generators  $n_{I J}$ are given by 
\begin{equation} \label{comm-rel-so}
n_{IJ} \dott n_{KL}= \defp_{IK}\defp_{LJ}  \, n_{KL} \dott n_{IJ} \, , \quad \defp_{I J}= \exp(-2i\pi \theta_{I J}) 
\end{equation}
and depend on an antisymmetric matrix of real deformation parameters 
\beq\label{thetagen}
\theta_{IJ}=-\theta_{JI}=-\theta_{IJ'} \,   , \quad I,J=1, \dots, 2n + 1 .
\eeq
Thus $\defp_{I J}=1$ if $I$ or $J$ is equal to $0$.

 The algebra $\mathcal{O}(SO_{{\theta}}(2n+1, \IR))$  becomes a Hopf algebra with  (in matrix notation)  coproduct and counit  $\Delta(N)=N \ot N$, $\cun (N)=\II_{2n+1} $, and antipode $S(N)= Q \dott N^t \dott Q$, 
 $S(n_{J K})= n_{K' J'}$. It is a $*$-Hopf algebra with
$*$-structure   $*(N)=QNQ$ 
{or 
$$
(n_{J K})^* = S(n_{K J}) = n_{J' K'} . 
$$ 
This corresponds to Euclidean signature $2n+1$ (see \cite[\S4.1.1]{ppca} for details).  
}

Consider the $n$ functionals $t_j:\O(SO_\theta(2n+1, \IR))\to \mathbb{C}$, $j \in \cI$, defined by $t_j(n_{KL})=\delta_{j K}\delta_{jL} - \delta_{j K'}\delta_{j L'}$ and
$t_j(aa')=t_j(a)\cun(a')+\cun(a)t_j(a')$, for any $a, a' \in \O(SO_\theta(2n+1, \IR))$. They commute under the
convolution product and give the
cotriangular structure of $\O(SO_\theta(2n+1, \IR))$
as $R=e^{2\pi i\theta_{jk}t_j\otimes t_k}$ (sum over $j,k \in \cI $ is understood).
They generate the Hopf algebra $U(\mathfrak{t}^n)$, the universal enveloping algebra of the
Lie algebra $\mathfrak{t}^n$ of the (commutative) $n$-torus. This Hopf
algebra has 
triangular structure $R=e^{2\pi i\theta_{jk}t_j\otimes t_k}$ and
therefore the Hopf algebra $K=U(\mathfrak{t}^n)^{op} \otimes
U(\mathfrak{t}^n)$ 
has triangular structure
$$
\r=(\id \otimes {\rm flip}  \otimes \id)(R^{-1}\otimes R)~.
$$
The functionals $t_j$ are lifted to the derivations
$ H_j=(t_j\otimes \id)\circ\Delta$ and $ \widetilde H_j=(\id\otimes
t_j)\circ \Delta$  of $\O(SO_\theta(2n+1, \IR))$, which are right and
left  $\O(SO_\theta(2n+1, \IR))$-invariant, respectively.
They define an action of the Hopf algebra $K=U(\mathfrak{t}^n)^{op} \otimes
U(\mathfrak{t}^n)$ on $\O(SO_\theta(2n+1, \IR))$ given explicitly by 
\beq\label{hhtact}
H_j \trl n_{KL} = (\delta_{jK} - \delta_{j'K}  ) n_{KL}  \, , \qquad \
\widetilde{H}_j \trl n_{KL} = (\delta_{jL} - \delta_{j'L} ) n_{KL} \, .
\eeq

It is not difficult
to see that cotriangularity of
$\O(SO_\theta(2n+1, \IR))$ is then equivalent to quasi-commutativity
of  the $(K,\r)$-module algebra $\O(SO_\theta(2n+1, \IR))$ (cf. also \cite[Ex. 5.10]{pgc-main}).
Explicitly, 
$a \dott a'=m_\theta \big(e^{2\pi i\theta_{jk}(  H_j\otimes  H_k - \widetilde H_j\otimes \widetilde H_k)}(a'\otimes a)\big)$, for $a,a'\in  \O(SO_\theta(2n+1, \IR))$ and $m_\theta$ the algebra multiplication in  $\O(SO_\theta(2n+1, \IR))$.

Similarly one defines  $\O(SO_\theta(2n, \IR))$. This  Hopf $*$-algebra
is a quantum subgroup of $\mathcal{O}(SO_\theta(2n+1, \IR))$, with  Hopf $*$-algebra surjection 
$
\eta:\mathcal{O}(SO_\theta(2n+1, \IR)) \to \mathcal{O}(SO_\theta(2n, \IR)) $. 
It is the quotient of $\O(SO_\theta(2n+1, \IR))$ by the Hopf ideal  
 $$
 \langle n_{00}-1, n_{J0}, n_{0J}, J=1, \dots , 2n+1, J \neq 0 \rangle \,.
 $$ 
 Hence there is a natural right coaction of $\O(SO_\theta(2n, \IR))$ 
 on $\O(SO_\theta(2n+1, \IR))$, 
$$
\delta:= (\id \ot \eta)\Delta, \; N \mapsto N \overset{.}{\otimes} \eta(N) \, , \quad 
\delta (n_{J K} ) = \sum\nolimits_{L} n_{J L } \ot m_{L K}, 
$$
having denoted $\eta(N) = M =(m_{I J})$.
The corresponding subalgebra of coinvariants is the 
quantum homogeneous space $\O(S_\theta^{2n})$. It is
generated by elements $u_I:=n_{I \, 0}$, $I=1, \dots, 2n+1$, with 
induced $*$-structure  $u_I^*=u_{I'}$ and  commutation relations 
$$
u_I\dott u_J=\defp_{IJ}u_J\dott u_I~.
$$
The orthogonality relations \eqref{orth-cond-comp} imply the sphere relation 
$ \sum\nolimits_{I}  u_{I'}\dott u_I=\sum\nolimits_{I}  u_{I}^*\dott u_I =1$.

 The algebra inclusion $\O(S_\theta^{2n}) \subset \O(SO_\theta(2n+1, \IR))$ is a $K$-equivariant  Hopf--Galois  extension, {with $\O(S_\theta^{2n})$ quasi-central in $\O(SO_\theta(2n+1, \IR))$.}
 We study the corresponding braided Lie algebra of derivations.

The braided Lie algebra ${\rm{Der}}^{\r}_{\M^H}(\O(SO_\theta(2n+1, \IR)))$ of $\O(SO_\theta(2n))$-equivariant derivations on  $\O(SO_\theta(2n+1))$ is generated, as  an  $\O(S^{2n}_\theta)$-module, 
by elements $\wl_{IJ}$ which are defined on the algebra generators as
\beq\label{der-so2n+1}
\wl_{I J }(n_{S T}) :=  \delta_{ J S'} \, n_{I T} - \defp_{I J} \delta_{ I' S} \, n_{J T}  \; 
\eeq
and are extended to the whole algebra as braided derivations. 
They are braided anti-symmetric, $\wl_{I J} = -\defp_{I J} \wl_{J I}$, and in particular $L_{jj'}=H_j$.
The $K$-action on the derivations $\wl_{IJ}$ is the lift of the action on $\O(SO_\theta(2n+1, \IR))$ given in \eqref{hhtact}. 
Since these derivations are right $\O(SO_\theta(2n+1, \IR))$-invariant they commute with the left invariant ones  
$\widetilde H_j$ and the action of the latter is trivial. As for the action of $H_j$ one finds
$$
H_j \trl \wl_{ST} = (\delta_{jS} + \delta_{jT} - \delta_{j'S} - \delta_{j'T} ) \wl_{ST} .
$$
It follows that for these generators the braided commutator in \eqref{bracket-der} explicitly reads
$$
[\wl_{I J} ,  \wl_{S T }] =
\wl_{I J}   \wl_{S T } - \defp_{I S} \defp_{I T}\defp_{J S} \defp_{J T} 
\wl_{S T} \wl_{I J}  \; .
$$

\begin{prop}\label{prop:crdown-so}
The derivations  $\wl_{I J}$  close the braided Lie algebra $so_\theta(2n+1)$:
\beq\label{blason}
[\wl_{I J} ,  \wl_{S T }] 
= \delta_{J' S}\wl_{I T} - \defp_{I J} \delta_{I' S}\wl_{J T}
-\defp_{S T} (\delta_{J' T}\wl_{I S}-\defp_{I J} \delta_{I' T}\wl_{J S})  \; .
\eeq
\end{prop}
\begin{proof}
From \eqref{der-so2n+1} one computes  
\begin{align*}
\wl_{I J}  \wl_{S T }(n_{K L}) = 
(n_{I L} \delta_{ S' J} - \defp_{I J} n_{ J L} \delta_{ S' I} ) \delta_{ T' K}  
- \defp_{S T} (n_{I L} \delta_{ T' J} - \defp_{I J} n_{ J L} \delta_{ T' I} ) \delta_{ S' K} 
 \end{align*}
and similarly
\begin{align*}
\wl_{S T}  \wl_{I J }(n_{K L}) = 
(n_{S L} \delta_{ I' T} - \defp_{S T} n_{ T L} \delta_{ I' S} ) \delta_{ J' K}  
- \defp_{I J} (n_{S L} \delta_{ J' T} - \defp_{S T} n_{ T L} \delta_{ J' S} ) \delta_{ I' K} \; .
 \end{align*}
Then, using \eqref{thetagen} to simplify products of $\defp$'s, one gets 
\begin{align*}
[\wl_{I J} ,&   \wl_{S T }] (n_{K L})  =
 \\
 &= \delta_{ S' J}  (n_{I L} \delta_{ T' K}  -   \defp_{I T}   n_{ T L} \delta_{ I' K} )
- \delta_{ S' I}  (\defp_{I J} n_{ J L}  \delta_{ T' K}  -  \defp_{J S} \defp_{J T}  n_{ T L} \delta_{ J' K}  )
\\
& 
-  \delta_{ T' J} (\defp_{S T} n_{I L}  \delta_{ S' K}  -  \defp_{I S}  \defp_{J S}   n_{S L}  \delta_{ I' K} )
+  \delta_{ T' I} (  \defp_{S T}  \defp_{I J} n_{ J L} \delta_{ S' K}  - \defp_{I S} \defp_{J S} \defp_{J T}  n_{S L}  \delta_{ J' K} )
\\
&=  \delta_{ S' J}  \wl_{I T} (n_{K L})  
- \delta_{ S' I}  \defp_{I J} \wl_{J T} (n_{K L}) 
-  \delta_{ T' J} \defp_{S T}   \wl_{I S} (n_{K L})
+  \delta_{ T' I}   \defp_{S T}  \defp_{I J}  \wl_{J S} (n_{K L})
 \end{align*}
and equation  \eqref{blason} is verified.
\end{proof}

It is clear that any  derivation in 
\eqref{der-so2n+1} 
restricts to a derivation of the subalgebra  of $\O(SO_\theta(2n+1))$ generated by the entries of any column  of $N$, thus in particular to a derivation of $\O(S^{2n}_\theta)$.
We denote
$\pi: {\rm{Der}}^{\r}_{\M^H}(\O(SO_\theta(2n+1, \IR))) \to {\rm{Der}}^{\r}(\O(S^{2n}_\theta))$  
the map 
which associates to $X\in {\rm{Der}}^{\r}_{\M^H}(\O(SO_\theta(2n+1, \IR)))$  its restriction $X^\pi$ to $\O(S^{2n}_\theta)$,  
$\pi(X):=X^\pi$.
For the  $\wl_{I J}$ in \eqref{der-so2n+1}, one easily computes
\beq\label{wl-sotto}
\wl_{I J }^\pi (u_{K }) = u_{I } \delta_{ J' K} - \defp_{I J} u_{J } \delta_{ I' K} \; .
\eeq
The restrictions  $\wl^\pi_{IJ}$ close the  braided Lie algebra $so_\theta(2n+1)$, as in \eqref{blason}, too.

\begin{lem}
The elements
\beq\label{der-s2n}
\ch_J:=\sum\nolimits_I u_{I'} \, \wl_{IJ}^\pi
\eeq
 generate the $\O(S^{2n}_\theta)$-module  ${\rm{Der}}^{\r}(\O(S^{2n}_\theta))$ of derivations of
$\O(S^{2n}_\theta)$.
\end{lem}  
\begin{proof}
We establish the lemma by showing that
$$
 \wl^\pi_{I J} = u_I \ch_J - \defp_{IJ} u_J \ch_I. 
$$
Since both sides are braided derivations it is enough to show the equality on the generators of $\O(S^{2n}_\theta)$.
From \eqref{wl-sotto}, and using the relation $ \sum\nolimits_{I}  u_{I'}\dott u_I=1$, one computes
\beq\label{wh-s2n}
 \ch_J(u_K) = \delta_{J K' } - \defp_{K' J} u_K \dott u_J = \delta_{J K'} - u_{J} \dott u_K  \, .
\eeq
Then, 
\begin{align*}
(u_I \ch_J  -\defp_{IJ} u_J \ch_I) (u_K) &= (u_I \delta_{J' K} - u_I \dott u_{J} \dott u_K) 
 - \defp_{IJ}   (u_J \delta_{I' K} - u_J \dott u_{I} \dott u_K) \\ &= u_I \delta_{J' K}  - \defp_{IJ} u_J \delta_{I' K}.
\end{align*}
A comparision with   \eqref{wl-sotto} shows that this  coincides with the evaluation $\wl_{I J }^\pi (u_{K })$.
\end{proof}
Notice that the sphere relation  implies the generators are constrained as   $\sum\nolimits_J u_{J'} \ch_J=0$. 

\begin{prop}\label{prop:Ders2n}
The braided Lie algebra structure of  ${\rm{Der}}^{\r}(\O(S^{2n}_\theta))$ is given by 
$$
 [\ch_I , \ch_J]  = u_I \ch_J - \defp_{IJ} u_J \ch_I  .
$$
\end{prop}
 \begin{proof}
From \eqref{wh-s2n} we compute
\begin{align*}
\ch_I \ch_J(u_K) & =    - \ch_I ( u_{J} \dott u_K) 
=  - (\delta_{I' J} - u_I \dott u_J ) \dott u_K - \defp_{IJ} u_{J} \dott (\delta_{I' K} - u_I \dott u_K)
\end{align*}
Then,
\begin{align*}
[\ch_I & ,\ch_J] (u_K)  = (\ch_I \ch_J - \defp_{IJ} \ch_J \ch_I)  (u_K) 
\\ 
& = 
- \delta_{I' J} u_K + 2 u_I \dott u_J \dott u_K - \defp_{IJ} \delta_{I' K}   u_{J} 
+ \defp_{IJ} \delta_{J' I} u_K - 2  \defp_{IJ} u_J \dott u_I \dott u_K +   \delta_{J' K}   u_{I} 
\\ 
&  = 
  - \defp_{IJ} \,  \delta_{I' K}   u_{J} 
  +   \delta_{J' K} u_{I} \\ & = (u_I \ch_J  -\defp_{IJ} u_J \ch_I) (u_K) \; .
\qedhere\end{align*}
\end{proof}
 
 In the limit $\theta_{IJ} =0$, the derivations $\wl^\pi_{IJ}$ give a representation of the Lie algebra $so(2n+1)$ as vector fields on the sphere $S^{2n}$.
 
\subsection{The sequence and the equivariant connection}

From the definition of the generators $\ch_J =\sum\nolimits_I u_{I'} \, \wl_{IJ}^\pi$ of ${\rm{Der}}^{\r}(\O(S^{2n}_\theta))$ in 
\eqref{der-s2n}, the right $\O(S^{2n}_\theta)$-module morphism 
\begin{align} \label{hor-lift-SO}
\rho :  {\rm{Der}}^{\r}(\O(S^{2n}_\theta)) \to \DerH{\O(SO_\theta(2n+1))} , \quad 
\ch_J   \mapsto \rho(\ch_J):= \sum\nolimits_I u_{I'} \wl_{I J}
\end{align}
 is a section
of the projection $\pi$, that is $\pi\circ \rho=\id_{\O(S^{2n}_\theta)}$. 
Explicitly, 
\begin{align}\label{rhowj}
  \rho(\ch_J) (n_{S T})
&=   \sum\nolimits_K u_{K'} \dott(n_{K T} \delta_{ J' S} - \defp_{K J} n_{J T} \delta_{ K' S})  
=   \delta_{0 T} \delta_{ J S'} -  \defp_{S' J} u_{S} \dott  n_{J T} 
\nn
\\
& 
=  \delta_{0 T} \delta_{ J S'} -   n_{J T} \dott  u_{S} \; ,
\end{align}
where we  used  \eqref{der-so2n+1} and    the orthogonality condition \eqref{orth-cond-comp} for the third equality.

\begin{prop} 
The connection $\rho$ is invariant under the action of  the   braided Lie algebra $so_\theta(2n+1)$: 
for every $  \wl_{I J} \in \DerH{\O(SO_\theta(2n+1))}$ we have  $ (\delta_{\wl_{IJ}} \rho)(X) =0$, that is 
$$
 [ \wl_{IJ} \,  ,\rho(X)]   - \rho([\wl_{IJ}^\pi, X] ) = 0 
$$
 for all 
$X  \in  {\rm{Der}}^{\r}(\O(S^{2n}_\theta))$.
\end{prop}
\noindent
The proof is analogous to that of Proposition \ref{prop:equiv-conn} and   we omit it.
Due to this proposition the connection $\rho$ is left $\O(S^{2n}_\theta)$-linear as well.
Using \eqref{der-so2n+1}, the map $\rho$ satisfies 
 $\rho(\sum\nolimits_J u_{J'} \ch_J) = \sum\nolimits_{I,J} u_{J'} \dott u_{I'} \wl_{I J}=0$ as it should be due 
 to $\sum\nolimits_J u_{J'} \ch_J=0$.

The  kernel of the projection $\pi$ is generated, as an $\O(S^{2n}_\theta)$-module, by the
derivations
\beq\label{vert-son}
\wy_{IJ}: =\wl_{IJ}-\rho (\wl^\pi_{IJ}) \, ,
\eeq
where
$$
\rho(\wl^\pi_{I J})
=u_{I} \rho(\ch_J)- \defp_{I J} u_{J} \rho(\ch_I)
= \sum\nolimits_K  u_{I} \dott u_{K'} \wl_{K J} - \defp_{I J} u_{J} \dott u_{K'} \wl_{K I} \, .
$$  
We set  $\mathrm{aut}^{\r}_{\O(S^4_\theta) }(\O(SO_\theta(2n+1, \IR)) ) = \ker{\pi}$,  the braided Lie subalgebra of vertical derivations. 
 From  \S \ref{se:scgt} this is the braided Lie algebra of infinitesimal gauge transformations of the $\O(SO_\theta(2n,\mathbb{R}))$ Hopf--Galois extension $\O(S^{2n}_\theta) \subset
\O(SO_\theta(2n+1,\mathbb{R}))$  of the noncommutative $\theta$-spheres $S^{2n}_\theta$.
For $n=2$, the   braided Lie algebra $\mathrm{aut}^{\r}_{\O(S^{4}_\theta) }(\O(SO_\theta(5, \IR)))$ was  studied in \cite{pgc-examples} in the context of twist deformation quantization.

The braided Lie algebras above give rise to a short exact sequence, 
\beq\label{as-SOn}
0 \to \mathrm{aut}^{\r}_{\O(S^{2n}_\theta) }(\O(SO_\theta(2n+1, \IR)) ) \stackrel{\imath}{\to} {\rm{Der}}^{\r}_{\M^H}(\O(SO_\theta(2n+1, \IR))) 
\stackrel{\pi}{\to} {\rm{Der}}^{\r}(\O(S^{2n}_\theta)) \to 0 \, .
\eeq
The  sequence is split by the map $\rho$ in 
\eqref{hor-lift-SO}, the horizontal lift.
The corresponding vertical projection is then \eqref{vert-son}:
\begin{align*}
\vp &: {\rm{Der}}^{\r}_{\M^H}(\O(SO_\theta(2n+1, \IR)))  \to \mathrm{aut}^{\r}_{\O(S^{2n}_\theta) }(\O(SO_\theta(2n+1, \IR))) \, , \\
& \quad \wl_{I J} \mapsto \vp( \wl_{I J} ) := 
\wl_{I J} - \rho(\wl_{I J}^\pi)
= \wy_{I J}   
\; .
\end{align*} 
As for the curvature one has the following.
\begin{prop}
On the generators   $\ch_\nu \in  {\rm{Der}}^{\r}(\O(S^{2n}_\theta)) $, the curvature is given  by
$$
\Omega(\ch_I, \ch_J)=  - \vp (\wl_{I J})
 = - \wy_{I J}    \, .
$$
\end{prop}
\begin{proof}
We need to show that $[\rho(\ch_I ), \rho(\ch_J)]  = \wl_{I J}$. 
For this we use the expression for
$\rho(\ch_J)(n_{S T})$ in \eqref{rhowj} and that 
$\rho(\ch_J) (u_{S}) = \ch_J(u_{S}) = \delta_{ J' S} -  u_{J} \dott  u_{S}$.
Then
\begin{align*}
 \rho(\ch_I ) \big(\rho(\ch_J) (n_{S T})\big) 
& =  -  \rho(\ch_I )  (n_{J T} ) \dott  u_{S}   - \defp_{I J }   n_{J T} \dott \rho(\ch_I )  (  u_{S} ) 
\\
 & =  -  (\delta_{0 T} \delta_{ I' J} -   n_{I T} \dott  u_{J}) \dott  u_{S}   - \defp_{I J }   n_{J T} \dott (\delta_{ I' S} -  u_{I} \dott  u_{S}  ) \; .
\end{align*}
Next for the braided commutator:
\begin{align*}
\Big( \rho(\ch_I )& \rho(\ch_J)  - \defp_{I J} \rho(\ch_J ) \rho(\ch_I)\Big) (n_{S T}) 
= 
\\
&= -  \delta_{0 T} \delta_{ I' J}  u_{S}  
+    n_{I T} \dott  u_{J} \dott  u_{S}  
 - \defp_{I J }  \delta_{ I' S}    n_{J T} 
+ \defp_{I J }   n_{J T} \dott   u_{I} \dott  u_{S}  
\\
& \quad
+   \delta_{0 T} \delta_{ J' I}  \defp_{I J} u_{S}  
-    \defp_{I J} n_{J T} \dott  u_{I} \dott  u_{S}  
 +   \delta_{ J' S}    n_{I T} 
-     n_{I T} \dott   u_{J} \dott  u_{S}  
\\
&=   
 - \defp_{I J }  \delta_{ I' S}    n_{J T} 
 +   \delta_{ J' S}    n_{I T} 
= \wl_{I J }(n_{S T})
\end{align*}
where the last equality follows from \eqref{der-so2n+1}.
\end{proof}

The $so_\theta(2n)$-valued connection $1$-form on the  bundle $\O(S^{2n}_\theta) \subseteq \O(SO_\theta(2n+1, \IR))$, corresponding to the 
splitting of the Atiyah sequence  \eqref{as-SOn}, is the projection $\omega_{|_{so_\theta(2n)}} $ to $so_\theta(2n)$ of the Maurer-Cartan form 
$$
\omega = - d N^\dag \dott  N 
$$
  on $\O(SO_\theta(2n+1, \IR))$.   Here $(N^\dag)_{JK}= n_{K'J'}$.

The differential calculus on the algebras $\O(SO_\theta(2n+1, \IR))$ was constructed in \cite{AC,cdv}. The commutation relations among degree-zero and degree-one generators of the differential algebra $\Omega(SO_\theta(2n+1, \IR))$ are given by
$$
n_{IJ} \dott d n_{KL}= \defp_{IK}\defp_{LJ}  dn_{KL} \dott n_{IJ} 
\quad  , \quad
dn_{IJ} \dott dn_{KL}=- \defp_{IK}\defp_{LJ}  dn_{KL} \dott dn_{IJ}  \, .
$$
Moreover, 
from $N^\dagger N = \II_{n+1}$ 
 one has   $\sum\nolimits_L ( dn_{L' K'} \dott  n_{LJ }  + \defp_{J K'}    d n_{L J} \dott  n_{L' K' }) = 0$.
Then,
$$
\omega_{K'J'}=  - \sum\nolimits_L d n_{L' K } \dott  n_{LJ'}
 =  \defp_{J' K} \sum\nolimits_L d  n_{LJ'}  \dott n_{L' K } 
= - \defp_{K J}\, \omega_{JK} \;.
$$
 With $E_{J K}$ the elementary matrices (with component 1 in position
$JK$ and zero otherwise)  
  one computes
$$
\omega= \sum\nolimits_{J, L} \omega_{J L} E_{J L} 
= \tfrac{1}{2}\sum\nolimits_{J, L} \omega_{J L} (E_{J L} - \defp_{ L J}E_{L'J'})
=  \tfrac{1}{2} \sum\nolimits_{J, L} \omega_{J L}  \wk_{J L'}  
$$
where in the last equality we have defined the $d(d-1)/2$ matrices
$\wk_{J L}$
with $d=2n+1$. More in general, for any even or odd $d$ we have the following.
\begin{lem}
The matrices  $\wk_{JL}=E_{J L'} - \defp_{J L}E_{LJ'}$, for
$J, L=1,\ldots d$, satisfy the equality  $\wk_{JL}=-\defp_{J
  L}\wk_{LJ}$.
The entries of each matrix $K_{JL}$ satisfy  
$((K_{JL})^tQ)_{AB}=-\lambda_{AB}(QK_{JL})_{AB}$.
 The matrices  $\wk_{JL}$ close the braided Lie algebra
$so_\theta(d)$:
\beq\label{son-rep}
 [\wk_{I J} , \wk_{S T} ]  
= \delta_{J S'}\wk_{I T} 
- \defp_{I J} \delta_{I S'}\wk_{J T}
-\defp_{S T} (\delta_{J T'}\wk_{I S}
-\defp_{I J} \delta_{I T'}\wk_{J S}) \, .
\eeq
\end{lem}
\begin{proof}
The first equality is clear, the second follows from a direct
computation. For the last statement we compute
\begin{align*}
\wk_{I J} \wk_{S T} &=  
\delta_{J' S } E_{I   T'} 
- \delta_{J' T} \defp_{S T} E_{I  S'}
- \delta_{I' S } \defp_{I J} E_{J T'} 
+ \delta_{ I' T} \defp_{S T} \defp_{I J} E_{J S'}
\end{align*}
and renaming the indices, 
\begin{align*}
\wk_{S T} \wk_{I J} &=  
\delta_{T' I } E_{S   J'} 
- \delta_{T' J} \defp_{I J} E_{S  I'}
- \delta_{S' I } \defp_{S T} E_{T J'} 
+ \delta_{ S' J} \defp_{I J} \defp_{S T} E_{T I'} \; .
\end{align*}
Then
\begin{align*}
&\!\!\!\! [\wk_{I J} , \wk_{S T} ]  
=
\wk_{I J} \wk_{S T} - \defp_{I S} \defp_{I T}\defp_{J S} \defp_{J T} \wk_{S T} \wk_{I J}  
\\
&= 
\delta_{J' S } (E_{I   T'} - \defp_{I S} \defp_{I T}\defp_{J S} \defp_{J T}   \defp_{I J} \defp_{S T} E_{T I'} )
- \delta_{I' S } (\defp_{I J} E_{J T'} - \defp_{I S} \defp_{I T}\defp_{J S} \defp_{J T}   \defp_{S T} E_{T J'} )
\\
& \quad 
- \delta_{J' T} (\defp_{S T} E_{I  S'} - \defp_{I S} \defp_{I T}\defp_{J S} \defp_{J T}  \defp_{I J} E_{S  I'})
+ \delta_{ I' T} (\defp_{S T} \defp_{I J} E_{J S'} - \defp_{I S} \defp_{I T}\defp_{J S} \defp_{J T}   E_{S   J'} )
\\
&= 
\delta_{J' S } (E_{I   T'} -   \defp_{I T}     E_{T I'} )
- \delta_{I' S } \defp_{I J} (E_{J T'} -      \defp_{J T}   E_{T J'} )
\\
& \quad 
- \delta_{J' T} \defp_{S T} (E_{I  S'} -  \defp_{I S}  E_{S  I'})
+ \delta_{ I' T} \defp_{S T} \defp_{I J} (E_{J S'} -   \defp_{J S}    E_{S   J'} )
\end{align*}
proving \eqref{son-rep}.
\end{proof}

The projection $\omega_{|_{so_\theta(2n)}} $ of $\omega$ to    the braided Lie subalgebra $so_\theta(2n)$ 
of $so_\theta(2n+1)$ is just
\beq\label{con-SOn}
\omega_{|_{so(2n)}} =   \sum\nolimits_{ J, K \neq 0} \omega_{J K}  \wk_{J K'}  \,.
\eeq
As a consistency check we show that $\omega_{|_{so_\theta(2n)}}$ is zero on horizontal fields:
$$
\langle \rho(\ch_I) , \omega_{|_{so(2n)}}   \rangle = - (\rho(\ch_I)(N^\dag)) \dott N _{|_{so(2n)}}=0
$$
for each horizontal field  $\rho(\ch_I)$ 
 in \eqref{hor-lift-SO}. Firstly we compute
 \begin{align*}
\langle \rho(\ch_I) , \omega_{J K}   \rangle  
& = - \sum\nolimits_L \rho(\ch_I) (  n_{L' J'})  \dott n_{L  K } 
=- \sum\nolimits_L (\delta_{0 J} \delta_{ I L} -   n_{I J'} \dott  u_{L'} )  \dott n_{L  K } 
\\
&= -   \delta_{0 J}     n_{I  K } 
+   \delta_{0 K}    n_{I J'}  
\end{align*}
where we used the expression for $\rho(\ch_I) (n_{S T})$ computed in \eqref{rhowj} and the orthogonality condition \eqref{orth-cond-comp}. This then gives 
$\langle \rho(\ch_I) , \omega_{|_{so(2n)}}   \rangle =0$ when 
 considering the projection $\omega_{|_{so_\theta(2n)}}$ of $\omega$ to $so_\theta(2n)$  in \eqref{con-SOn}.

\subsection{The derivations as sections of an associated bundle}

Given the $\O(SO_\theta(2n,\mathbb{R}))$-Hopf--Galois extension $\O(S^{2n}_\theta) \subset \O(SO_\theta(2n+1,\mathbb{R}))$  of orthonormal frames on $S^{2n}_\theta$, 
we identify the derivations ${\rm{Der}}^{\r}(\O(S^{2n}_\theta))$ as the
sections of the associated bundle for the fundamental corepresentation of the Hopf algebra 
  $\O(SO_\theta(2n,\mathbb{R}))$ on the algebra $\O(\IR^{2n}_\theta)$.

Given a right $H$-comodule algebra $A$ with coaction
$
\delta: A \to A \ot H$, $\delta(a) = \zero{a} \ot \one{a}
$ 
and a left    $H$-comodule $V$  with coaction
$
\gamma: V \to H \ot V$, $\gamma(v) =  \mone{v} \ot \zero{v}$,
sections of the vector bundle associated with the corepresentation  $\gamma$ can be identified with linear maps
 $\phi: V \to A$
 which are $H$-equivariant
\beq\label{seeq}
\zero{\phi(v)} \ot \one{\phi(v)} = \phi(\zero{v}) \ot S(\mone{v}) \, .
\eeq
The collection $\mathcal{E}$ of such maps is a left $B$-module for $B\subseteq A$ the subalgebra of coinvariants for the $H$-coaction.

For $H=\O(SO_\theta(2n,\mathbb{R}))$ consider the fundamental corepresentation 
$$ \gamma: \IR^{2n} \to \O(SO_\theta(2n,\mathbb{R})) \ot \IR^{2n} 
\; , \quad 
e_\alpha \mapsto \sum\nolimits_{\beta}m_{\alpha\beta} \ot e_\beta 
$$
on the vector space $\IR^{2n}$ with the $2n$ basis elements $e_\alpha$,  $\alpha \neq \alpha' $. 
 We denote by $\mathcal{E}_T$ the $\O(S^{2n}_\theta)$-module of equivariant maps, defined as in \eqref{seeq},  
 associated via this corepresentation to the $\O(SO_\theta(2n,\mathbb{R}))$-Hopf--Galois extension 
$\O(S^{2n}_\theta) \subset \O(SO_\theta(2n+1,\mathbb{R}))$.

\begin{prop}
The $\O(S^{2n}_\theta) $-module $\mathcal{E}_T$ of equivariant maps  is generated by the $2n+1$ linear maps 
\begin{align}\label{phiJ}
\phi_{(J)} &:  \IR^{2n} \to  \O(SO_\theta(2n+1,\mathbb{R})) \, , \nn \\
\phi_{(J)}(e_\alpha) &:= (NQ)_{J\alpha} = n_{J\alpha'} \, \quad J = 1,  ..., 2n+1  \,.
\end{align}
 \end{prop}
\begin{proof}
With the coactions $\delta(n_{JK}) = \sum_L n_{JL} \ot m_{LK}$ and $\gamma(e_\alpha) = \sum_\beta m_{\alpha\beta} \ot e_\beta$, 
the  condition in \eqref{seeq} implies that an equivariant map is linear in the generators $n_{KL}$:
$$
\phi(e_\alpha) = \sum\nolimits_{J,K} b_\alpha^{JK} \dott n_{JK}, \qquad b_\alpha^{JK} \in \O(S^{2n}_\theta) .
$$ 
Then \eqref{seeq} yields 
$$ 
\sum\nolimits_{J,K,L} b_\alpha^{JK} \dott n_{JL} \ot m_{LK}  = 
\sum\nolimits_{\beta ,J,K} b_\beta^{JK} \dott n_{JK} \ot m_{\beta' \alpha'} .
$$ 
Thus $b_\alpha^{JK} = 0$ when $K \neq \alpha'$ and $\phi(e_\alpha)$ reduces to
$\phi(e_\alpha) = \sum\nolimits_{J} b_\alpha^{J} \dott n_{J\alpha'}$. The equivariance implies  
$b_\alpha^{J} = b_\beta^{J}$, for all $\beta$, that is the coefficients  do not depend on the basis element $e_\alpha$. 
The general equivariant map is thus written as $\phi(e_\alpha) = \sum\nolimits_{J} \beta^{J} \dott n_{J\alpha'}$ with $\beta^{J} \in \O(S^{2n}_\theta)$. This concludes the proof.
\end{proof}

The module generators above are not independent:
$$
\sum\nolimits_J u_{J'} \, \phi_{(J)} =0 
$$
 for $u_J=n_{J0}$ the generators of the sphere $\O(S^{2n}_\theta)$ and module structure  
written as $(b  \phi_{(J)})(e_\alpha)= b \dott \left( \phi_{(J)}(e_\alpha)\right) $, 
for $b \in \O(S^{2n}_\theta)$.
Indeed, from \eqref{orth-cond-comp},
 $$
\sum\nolimits_J u_{J'} \dott \phi_{(J)} (e_\alpha) =  \sum\nolimits_J u_{J'} \dott n_{J \alpha'} = \sum\nolimits_J n_{J' 0} \dott n_{J \alpha'} = \delta_{0 \alpha'}=0 \, .
$$

The $\O(S^{2n}_\theta)$-module of equivariant maps  $\mathcal{E}_T$  can be realised as the image of the free module via a suitable projection with entries in $\O(S^{2n}_\theta)$.
Define $2n$ vectors $|\varphi_\alpha \rangle$,  $\alpha \neq \alpha'$, with components
$$ 
|\varphi_\alpha \rangle_{J}:= n_{J\alpha'} \, , \quad J=1, \dots, 2n+1 .
$$
From the orthogonality condition of the matrix $N$ in \eqref{orth-cond-comp}, these are orthonormal 
$\langle  \varphi_\alpha , \varphi_\beta \rangle = \sum\nolimits_ J n_{J' \alpha} \, n_{J \beta'} = \delta_{\alpha \beta}$ and we get a matrix projection
$$
p := \sum\nolimits_{\alpha \neq \alpha'}   |\varphi_\alpha \rangle \langle  \varphi_\alpha| \, .
$$
The entries of $p$ are in $\O(S^{2n}_\theta)$: using again the orthogonality condition one computes 
\beq\label{p-tg}
p_{IJ}= \sum\nolimits_{\alpha}  n_{I \alpha} \dott n_{J' \alpha'} =  \delta_{I J} - n_{I0} \dott n_{J' 0} =  \delta_{I J} - u_{I} \dott u_{J'} \, .
\eeq
This projector has rank $2n$,  its trace is 
$$
tr(p)= \sum\nolimits_{J=1}^{2n+1} (\delta_{JJ} - u_{J} \dott u^*_{J}) = 2n \cdot 1 \, ,
$$
with $1$ the constant function. We can then identify 
$\mathcal{E}_T = (\O(S^{2n}_\theta)^{2n+1})p$. In this way the rows
of $p$ are a set of generators for the module $\mathcal{E}_T$.

 The module $\mathcal{E}_T$ gives the `tangent module', that is, the derivations 
${\rm{Der}}^{\r}(\O(S^{2n}_\theta))$. This can be seen in two
different ways. On the one hand, from the expression \eqref{p-tg}, the rows of
$p$, that is, the generators of $\mathcal{E}_T$, are the components of the generators $\ch_J$ in \eqref{wh-s2n}.

On the other hand, let   $ | \Delta \rangle$ be  the unit vector  with
components $(u_J, J = 1, \dots, 2n +1)$, and let $p_N= | \Delta
\rangle  \langle \Delta |$ be the `normal' projection  with corresponding `normal' bundle 
$\mathcal{E}_N$. From $p\oplus p_N={\rm{id}}$ we read the direct sum decomposition $\O(S^{2n}_\theta)^{2n+1}= \mathcal{E}_T \oplus \mathcal{E}_N$. 
We have as a consequence a module isomorphism between the $\O(S^{2n}_\theta)$-module  
${\rm{Der}}^{\r}(\O(S^{2n}_\theta))$ of derivations of
$\O(S^{2n}_\theta)$ and the $\O(S^{2n}_\theta)$-module
$\mathcal{E}_T$.

We then
identify the generators $\ch_J$  of ${\rm{Der}}^{\r}(\O(S^{2n}_\theta))$ in \eqref{der-s2n} 
with the generators $\phi_{(J)}$ of the module $\mathcal{E}_T$ in \eqref{phiJ} via the module map 
$$
\Gamma: {\rm{Der}}^{\r}(\O(S^{2n}_\theta)) \to \mathcal{E}_T, \quad \ch_J \mapsto \phi_{(J)} \, \qquad J = 1, \dots, 2n+1.
$$

The matrix $Q$ defining the orthogonality condition \eqref{ort-con-Q} 
is used for a metric on $\mathcal{E}_T$,
$$ 
g : {\rm{Der}}^{\r}(\O(S^{2n}_\theta)) \times  {\rm{Der}}^{\r}(\O(S^{2n}_\theta)) \to \O(S^{2n}_\theta) \, .
$$
This is the restriction of the standard metric on the free module $\O(S^{2n}_\theta)^{2n+1}$ to 
the module $\mathcal{E}_T = (\O(S^{2n}_\theta)^{2n+1}) p \simeq {\rm{Der}}^{\r}(\O(S^{2n}_\theta))$.
 Thinking of the generators $\ch_J$ as the rows of the projection \eqref{p-tg}, on these the metric is defined by
$$
g(\ch_J, \ch_K) := \sum\nolimits_{L, I} (\ch_J)_L \dott Q_{LI} \dott (\ch_K)_I 
$$
and computed to be \begin{align}\label{metr}
g(\ch_J, \ch_K) &=
\sum\nolimits_{L,I} \delta_{LI'} \, (\ch_J)_L \dott (\ch_K)_{I} = 
\sum\nolimits_L (\ch_J)_L \dott (\ch_K)_{L'} \nn \\ 
& \:= \sum\nolimits_L p_{JL} \dott p_{K L'} = \sum\nolimits_L p_{J L} \dott p_{L K'} = p_{J K'} \nn \\
&\: = \delta_{J K'} - u_{J} \dott u_K = \ch_J(u_K).
\end{align}
Here we used the properties $p^\dagger = p$  and $p^2=p$ read as $p_{K'L'} = p_{LK}$ and
$\sum_L p_{KL} \dott p_{LJ} = p_{KJ}$.
The expression \eqref{metr} is extended by $\O(S^{2n}_\theta)$-bilinearity. 
Clearly the metric $g$ is braided-symmetric: 
$g(\ch_K, \ch_J) = \lambda_{KJ} \, g(\ch_J, \ch_K)$. 

\subsection{The Levi-Civita connection} \label{sec:LC}
With the identification above, the connection (the splitting) $\rho$ of the Atiyah sequence 
in \eqref{hor-lift-SO} induces an affine connection on ${\rm{Der}}^{\r}(\O(S^{2n}_\theta))$. 

\begin{defi}
For $X,Y \in {\rm{Der}}^{\r}(\O(S^{2n}_\theta))$, the covariant derivative of $Y$ along $X$ is defined to be
\beq\label{cov-der}
\nabla_X Y := \Gamma^{-1} \big( \rho(X) (\Gamma(Y)) \big)
\eeq
for $\rho$ the splitting and the derivation $\rho(X)$ acting on the function $\Gamma(Y)$. 
\end{defi}
\noindent
Using that $\rho$ is equivariant and thus right and left $\O(S^{2n}_\theta)$-linear, one proves the following covariant derivative properties.

\begin{prop}\label{prop:cov-der}
The covariant derivative has {the} properties:
\begin{enumerate}[(i)]
\item $\nabla_{b X} Y= b \, \nabla_X Y$;
\item $\nabla_X (bY)= X(b) \ Y + (\r_\alpha \trl b) \, \nabla_{\r^\alpha \trl X} Y$
\end{enumerate}
for $X,Y \in {\rm{Der}}^{\r}(\O(S^{2n}_\theta))$ and $b \in  \O(S^{2n}_\theta)$.
\end{prop}

The Riemannian curvature of the covariant derivative is naturally defined as 
$$
\mathsf{R}(X,Y) \, Z:= \nabla_X \nabla_Y Z  -  \nabla_{\r_\alpha \trl Y} \nabla_{\r^\alpha \trl X} Z - \nabla_{[X,Y] }  Z\, ,
$$
for $X,Y,Z\in {\rm{Der}}^{\r}(\O(S^{2n}_\theta))$. 
The properties in Proposition \ref{prop:cov-der} imply that the curvature is left $\O(S^{2n}_\theta)$-linear while 
the equivariance of $\rho$ implies that it is also right $\O(S^{2n}_\theta)$-linear.

\begin{prop}
Let $\Omega(X,Y) = \rho([ X,Y] ) - [\rho(X), \rho(Y)]$ be the curvature of the equivariant splitting (the connection) $\rho$. Then
$$
\mathsf{R}(X,Y) \, Z = - \Gamma^{-1} \big(\Omega(X,Y) (\Gamma(Z)) \big).
$$
\end{prop}
\begin{proof} 
From $\Gamma(\nabla_Y Z)  = \rho(Y) (\Gamma(Z))$ one gets $\Gamma( \nabla_X \nabla_Y Z)  = \rho(X) \rho(Y) (\Gamma(Z))$. 
Thus,
\begin{align*}
\Gamma\big(\mathsf{R}(X,Y) \, Z\big) = \Big( \rho(X) \rho(Y) - \rho(\r_\alpha \trl Y) \rho(\r^\alpha \trl X) - \rho([ X,Y] )\Big) (\Gamma(Z))
\end{align*}
from which the stated equality follows.
\end{proof}

Furthermore, the  torsion tensor of the covariant derivative is  defined as 
$$
\mathsf{T}(X,Y):= \nabla_X Y  -  \nabla_{\r_\alpha \trl Y} (\r^\alpha \trl X) - [X,Y]   \, ,
$$
for $X,Y \in {\rm{Der}}^{\r}(\O(S^{2n}_\theta))$. 
Again it is $\O(S^{2n}_\theta)$-bilinear.
 
\medskip
As we see below, the covariant derivative in \eqref{cov-der} is the Levi-Civita one.
On the  generators $\ch_J$  of the  module  ${\rm{Der}}^{\r}(\O(S^{2n}_\theta))$, formula \eqref{cov-der} gives
\beq\label{cov-der-T}
\nabla_{\ch_J} \ch_K = - \ch_J \cdot u_K \, .
\eeq
(As before we write $ X \cdot b $ to  distinguish the right $\O(S^{2n}_\theta)$-module structure from the evaluation 
 of a derivation on an element in $\O(S^{2n}_\theta)$.) 
Indeed, using \eqref{rhowj} one computes,
$$
\Gamma(\nabla_{\ch_J} \ch_K)(e_\alpha) = \rho(\ch_J) (n_{K \alpha'}) 
= - n_{J \alpha'} \dott u_K = - \big( \Gamma(\ch_J)(e_\alpha) \big) \dott u_K .
$$

\begin{prop}
The covariant derivative in \eqref{cov-der} is torsion free.
\end{prop}
\begin{proof} 
Using \eqref{cov-der-T}, we compute
$$
\nabla_{\ch_I} \ch_J  - \lambda^{IJ} \nabla_{\ch_J} \ch_I= 
 - \ch_I  \cdot u_J +  \lambda_{IJ} \ch_J \cdot u_I  =  -  \lambda_{IJ} u_J \ch_I +   u_I  \ch_J =    [\ch_I, \ch_J]  
$$
where we used the module structure \eqref{RAmodderA} for the second equality and Proposition \ref{prop:Ders2n} for the last one. Then $\mathsf{T}(\ch_I, \ch_J)=0$ as claimed.
\end{proof}
 \begin{prop}
The covariant derivative is compatible with the metric \eqref{metr}:
$$
g(\nabla_X Y, Z) + \, g(\r_\alpha \trl Y, \nabla_{\r^\alpha \trl X} Z) = X (g(Y, Z)) . 
$$
for $X,Y,Z\in {\rm{Der}}^{\r}(\O(S^{2n}_\theta))$.
\begin{proof}
On generators the above becomes,
\beq\label{metcon}
g(\nabla_{\ch_L} \ch_J, \ch_K) + \lambda_{JL} \, g(\ch_J, \nabla_{\ch_L} \ch_K) = \ch_L \left(g(\ch_J, \ch_K) \right). 
\eeq
Then, from the explicit expression \eqref{cov-der-T} and \eqref{metr} one computes, 
\begin{align*}
g(\nabla_{\ch_L} \ch_J, \ch_K) & + \lambda_{JL} \, g(\ch_J, \nabla_{\ch_L} \ch_K) = 
- g(\ch_L \cdot u_J, \ch_K) - \lambda_{JL} \, g(\ch_J, \ch_L \cdot u_K)  \\
& = - \lambda_{JL}\, u_J \dott g(\ch_L, \ch_K) - g(\ch_L, \ch_J) \dott u_K   \\
& = - \lambda_{JL}\, u_J \dott \ch_L(u_K) - \ch_L(u_J) \dott u_K
\end{align*}
which coincides with the right hand side of \eqref{metcon} by the last line in \eqref{metr}.
\end{proof}
\end{prop}

\subsection{The Riemannian geometry}

For the covariant derivative $\nabla$ in \eqref{cov-der}, the Riemannian curvature on the  generators 
$\ch_J \in {\rm{Der}}^{\r}(\O(S^{2n}_\theta))$ is given by
\begin{align}
\mathsf{R}( \ch_I,\ch_J) \, \ch_K
& = \ch_I \, \cdot \ch_J(u_K)  - \lambda_{IJ}  \ch_J \, \cdot \ch_I(u_K)   \nn \\
& = \ch_I \, \cdot g(\ch_J, \ch_K)  - \lambda_{IJ} \, \ch_J \, \cdot g(\ch_I, \ch_K) . \label{rimmet}
\end{align}
 First, using Proposition \ref{prop:cov-der},  we have
\begin{align*}
 \nabla_{\ch_I} \nabla_{\ch_J} \ch_K  
 &
 =  \nabla_{\ch_I} (- \ch_J \cdot u_K)
 = - \lambda_{JK} \nabla_{\ch_I} (u_K \ch_J )
 = - \lambda_{IJ}  \ch_J \, \cdot \ch_I(u_K) + \ch_I \, \cdot u_J \dott u_K
\end{align*}
 from \eqref{cov-der-T},  
 and 
 \begin{align*}
 \nabla_{[\ch_I,\ch_J] }  \ch_K 
 &
=  \nabla_{-  \lambda_{IJ} u_J \ch_I +   u_I  \ch_J}  \ch_K
= - u_I \ch_J \cdot u_K + \lambda_{IJ}u_J \ch_I \cdot u_K \, .
  \end{align*}
Then,
\begin{align*}
\mathsf{R}( \ch_I,\ch_J) \, \ch_K
&
= \nabla_{\ch_I} \nabla_{\ch_J} \ch_K  -  \lambda_{IJ} \nabla_{\ch_J} \nabla_{\ch_I} \ch_K - \nabla_{[\ch_I,\ch_J] }  \ch_K 
\\
&= - \lambda_{IJ}  \ch_J \, \cdot \ch_I(u_K) + \ch_I \, \cdot u_J \dott u_K  
+  \ch_I \cdot \ch_J(u_K) \nn \\
& \qquad\qquad - \lambda_{IJ}  \ch_J \, \cdot u_I \dott u_K
+u_I \ch_J \cdot u_K - \lambda_{IJ}u_J \ch_I \cdot u_K
\\
&
= - \lambda_{IJ}  \ch_J \, \cdot \ch_I(u_K)   +  \ch_I \, \cdot \ch_J(u_K) .
\end{align*}
The second equality in \eqref{rimmet} follows from the last line in definition \eqref{metr}.
 
We next consider the dual $({\rm{Der}}^{\r}(\O(S^{2n}_\theta)))'$ of ${\rm{Der}}^{\r}(\O(S^{2n}_\theta))$
with respect to the  metric $g$ in \eqref{metr}. That is we have a map still denoted $g$ and defined on generators by 
\beq\label{dualbasis}
\theta_J := g(\ch_J) :  {\rm{Der}}^{\r}(\O(S^{2n}_\theta)) \to \O(S^{2n}_\theta), \quad \ch_K \mapsto \theta_J (\ch_K) = 
g(\ch_J, \ch_K) = \ch_J(u_K). 
\eeq

\begin{prop}
The Ricci tensor, defined by
$$
\mathsf{R}(\ch_J \, , \ch_K) := \sum\nolimits_I \theta_{I'} \big(\mathsf{R}( \ch_I,\ch_J) \, \ch_K\big)
$$
is computed to be  
\beq\label{ricci}
\mathsf{R}(\ch_J \, , \ch_K) = (2n-1) \, \ch_j(u_K) = (2n-1) \, g(\ch_J \, , \ch_K) \, .
\eeq 
The scalar curvature defined by
$$
\mathsf{r} = {\sum\nolimits_J}\mathsf{R}(g^{-1}(\theta_{J'}) \, , \ch_J)
$$
is computed to be
\beq\label{sc}
\mathsf{r} = 2n (2n-1).
\eeq
\end{prop}
\begin{proof}
For the Ricci tensor, from \eqref{rimmet} and \eqref{dualbasis}, 
$$
\mathsf{R}(\ch_J \, , \ch_K) := \sum\nolimits_I \big( \theta_{I'}(\ch_I) \, \ch_J (u_K) - \lambda_{IJ} \, \theta_{I'}(\ch_J) \, \ch_I(u_K) \big)
$$
The first term is given by $\sum_I (\delta_{I' I'} - u_{I'} \dott u_{I} ) \, \ch_J (u_K) = 2n \, \ch_J (u_K)$ while the second one is 
$ - \sum_I \lambda_{IJ} (\delta_{I' J'} - u_{I'} \dott u_{J} ) \, \ch_I (u_K)  = - \ch_J (u_K) + u_J \dott \sum_I u_{I'}(\ch_I (u_K))  = -\ch_J (u_K) $, being $\sum_I u_{I'} \ch_I =0$. When added they give \eqref{ricci} using that $\ch_J (u_K)=g(\ch_J \, , \ch_K)$.

For the scalar curvature one has then
$$
r = (2n-1) \sum\nolimits_J \, \ch_J(u_{J'}) = (2n-1) \sum\nolimits_J (\delta_{J J} - u_{J'} \dott u_{J} ) = (2n-1) 2n.
$$
as stated.
\end{proof}

\noindent
\textbf{Acknowledgments.}~\\[.5em]
PA acknowledges  partial support from INFN, CSN4, Iniziativa
Specifica GSS and from INdAM-GNFM. PA acknowledges financial support from Universit\`a del Piemonte Orientale. 
GL acknowledges partial support from INFN, Iniziativa Specifica GAST
and INFN Torino DGR4 
as well as hospitality from Universit\`a del  Piemonte Orientale and INFN Torino.
GL acknowledges partial support from 
INdAM-GNSAGA.
GL acknowledges support from PNRR MUR projects PE0000023-NQSTI.
CP acknowledges support from Universit\`a di Trieste (assegni
Dipartimenti di Eccellenza,  legge n. 232 del 2016) and University of Naples Federico II under the grant FRA 2022 GALAQ: Geometric and ALgebraic Aspects of Quantization.
This article is based upon work from COST Action CaLISTA CA21109 supported by COST 
(European Cooperation in Science and Technology).

\end{document}